\documentclass[11pt]{amsart}
\usepackage{amsmath}
\usepackage{amsfonts}
\usepackage{amssymb}
\usepackage{graphicx}
\usepackage{hyperref}
\usepackage{mathrsfs}
\usepackage{mathtools}
\usepackage{pb-diagram}
\usepackage{epstopdf}
\usepackage{amsmath,amsfonts,amsthm,enumerate,amscd,latexsym,curves}
\usepackage{bbm}
\usepackage[mathscr]{eucal}
\usepackage{epsfig,epsf,epic}
\usepackage{caption}
\usepackage{subcaption}
\usepackage{multirow}
\usepackage{color}
\usepackage[all]{xy}
\usepackage{fourier}
\usepackage{tikz-cd}
\usepackage{calc}
\usepackage{float}
\usepackage{comment}
\usepackage{lipsum}

\graphicspath{ {images/} }

\newtheorem{thm}{Theorem}[section]
\newtheorem{lemma}[thm]{Lemma}

\newtheorem{cor}[thm]{Corollary}
\newtheorem{prop}[thm]{Proposition}
\newtheorem{defn}[thm]{Definition}
\newtheorem{example}[thm]{Example}
\newtheorem{remark}[thm]{Remark}
\newtheorem{conjecture}[thm]{Conjecture}

\numberwithin{equation}{section}

\newcommand{\bb}{\boldsymbol{b}}

\newcommand{\be}{\mathbf{1}}

\newcommand{\AI}{A_\infty}

\newcommand{\Hom}{{\rm Hom}}

\newcommand{\WT}[1]{\widetilde{#1}}
\newcommand{\WH}[1]{\widehat{#1}}

\newcommand{\cF}{\mathcal{F}}

\newcommand{\cH}{\mathcal{H}}

\newcommand{\Z}{\mathbb{Z}}
\newcommand{\R}{\mathbb{R}}
\newcommand{\C}{\mathbb{C}}

\newcommand{\MF}{\mathrm{MF}}
\newcommand{\WF}{\mathcal{WF}}
\newcommand{\Mil}{M}

\newcommand{\Jac}{\mathrm{Jac}}
\newcommand{\bL}{\mathbb{L}}
\newcommand{\Cone}{\mathrm{Cone}}

\begin{document}

\title[Berglund--H\"ubsch mirrors of invertible curve singularities via Floer theory]{Berglund--H\"ubsch mirrors of invertible curve singularities via Floer theory}

\author[Cho]{Cheol-Hyun Cho}
\address{Department of Mathematical Sciences, Research Institute in Mathematics\\ Seoul National University\\ Seoul \\ South Korea}
\email{chocheol@snu.ac.kr}
\author[Choa]{Dongwook Choa}
\address{Institute for Basic Science Center for Geometry and Physics \\ Nam-gu\\Pohang\\South Korea}
\email{dwchoa@ibs.re.kr}
\author[Jeong]{Wonbo Jeong}
\address{Department of Mathematical Sciences, Research Institute in Mathematics\\ Seoul National University\\ Gwanak-gu\\Seoul \\ South Korea}
\email{wonbo.jeong@gmail.com}
 
 \begin{abstract}
 We find a Floer theoretic approach to obtain the transpose polynomial $W^T$ of an invertible curve singularity $W$. This gives an intrinsic construction of the mirror transpose polynomial and enables us to define a canonical $\AI$ functor that takes Lagrangians in the Milnor fiber of W and converts them into matrix factorizations of $W^T$. We find Lagrangians in the Milnor fiber of $W$ that are mirror to the indecomposable matrix factorizations of $W^T$ when $W^T$ is ADE singularity and discover that Auslander-Reiten exact sequences can be realized as surgery exact triangles of Lagrangians in the mirror.

There are two primary steps in the Floer theoretic method for obtaining a transposition polynomial: To get a Lagrangian $L$ and corresponding disc potential function $W_L$, we first determine the quotient $X$ by the maximal symmetry group for the Milnor fiber. 
Second, we define a class $\Gamma$ of symplectic cohomology of $X$ based on the monodromy of the singularity $W$. Another disc counting function, $g$, is defined by the closed-open image of $\Gamma$ on $L$. We demonstrate that restricting to the hypersurface $g = 0$ transforms the disc potential function $W_L$ into the transpose polynomial W T. This second step is the mirror of taking the cone of quantum cap action by the monodromy class $\Gamma$.
\footnotemark  
\end{abstract}

\maketitle
\stepcounter{footnote}
\footnotetext{
This work has been extracted from version 1 of the preprint \cite{CCJ20}(arXiv:2010.09198v1), except for Section 8.
After significantly expanding, its first half became version 2 of the preprint (arXiv:2010.09198v2) in which we provided a general construction of the $\AI$-category associated with the log Fano or log Calabi-Yau type Landau-Ginzburg orbifold.}
 \tableofcontents


\section{Introduction}
Consider the following polynomials, which are called invertible curve singularities.
Here $p,q \in \mathbb{N}$ ($p,q \geq 2$).
\begin{table}[h!]
  \begin{center}
    \label{tab:table1}
    \begin{tabular}{c|c|c} 
      \textbf{Type} & $W$ & $W^T$ \\       \hline
   \textit{Fermat} $F_{p,q}$ &$x^p+y^q$ & $x^p+y^q$     \\  
        \hline
   \textit{Chain} $C_{p,q}$ & $x^p+xy^q$ & $x^py+y^q$     \\  
        \hline
   \textit{Loop} $L_{p,q}$ & $x^py+xy^q$ & $x^py+xy^q$     \\  
     \end{tabular}
  \end{center}
\end{table}

For a given polynomial $W$ in the above, exponents of monomials form an invertible $(2\times2)$ matrix $A$,
and the transpose polynomial $W^T$ is then defined by the transpose matrix $A^T$. 
In this paper, we propose a symplectic geometric approach to obtain the transpose polynomial $W^T$ from the geometry of $W$.

This is a part of the so-called Berglund--Hubsch (homological) mirror symmetry conjecture between invertible polynomials. Its exact form requires some explanation and inevitably involves the symmetry groups of these polynomials.
First, an invertible polynomial is given by
 \[W(x_1, \ldots, x_n)=\sum_{i=1}^n \prod_{j=1}^n x_j^{a_{ij}}\] 
where the matrix of exponents $A=(a_{ij})$ is an $n\times n$ invertible matrix whose entries are nonnegative integers. Kreuzer--Sharke \cite{KrSk} classified invertible singularities and demonstrated that, up to a change in variables, they are given by a Thom--Sebastiani sum of Fermat, Chain, and Loop type polynomials. 

The symmetry group of $W$ is the next crucial piece of information.
The \emph{maximal (abelian) symmetry group} of $W$ is a finite abelian group defined by
\[G_{W} := \left\{ (\lambda_1, \ldots, \lambda_n) \in (\mathbb{C}^{*})^n | \;W(\lambda_{1}x_{1},\ldots, \lambda_{n}x_{n}) = W \right\}.\]
   
Berglund, H\"ubsch, and Henningson \cite{BH93} \cite{BH95} proposed the following mirror pairs.
\begin{defn}
Let $W$ be an invertible polynomial with an exponent matrix $A$, and let $W^T$ be the transpose polynomial whose exponent matrix is $A^T$.

Let $G$ be a subgroup of $G_W$. Define the transpose subgroup $G^T$ by
	$$G^{T}: = \Hom(G_{W}/G,\mathbb{C}^{*}) \subset \mathrm{Hom}(G_{W}, \mathbb{C}^{*}) \simeq G_{W^T}. $$
The following are called Berglund--H\"ubsch(--Henningson) dual pairs of each other.
$$\left(W, G\right)  \Longleftrightarrow \left(W^T, G^T\right)$$
\end{defn}
In particular, two extreme cases arise when $G$ or $G^T$ are trivial.
$$W  \Longleftrightarrow  \left(W^T, G_{W^T}\right), \hspace{.5cm} (W,G_W)  \Longleftrightarrow W^T.$$

\begin{conjecture}[Berglund--H\"ubsch homological mirror symmetry]
\label{Conj: BHHMS}
There exists a derived equivalence between Fukaya category associated with $(W, G)$ and $G^T$-equivariant matrix factorization category $\mathcal{MF}(W^T, G^T)$.
\end{conjecture}
\subsection{Pertinent works}
When $G$ is a trivial subgroup of $G_W$, the associated Fukaya category of $W$ is the Fukaya-Seidel category ${FS}(W)$ \cite{S08}. It is arguably one of the most deeply studied objects in the community since then. Conjecture \ref{Conj: BHHMS} also has been extensively studied. To name a few, we refer to \cite{Sei01} for type $A_n$, \cite{AKO08} for Del-Pezzo type, \cite{FU11} for Brieskorn-Pham cases, \cite{FU13} for type $D_n$, \cite{KST07, T10} for the quiver algebra approach, \cite{HS19, Habermann22} for invertible curve cases. In these works, homological mirror symmetry is proved by ingenious computation of both sides. 

In the related context, homological mirror symmetry for the \emph{Milnor fiber} of $W$ was also studied. See \cite{LU20, LU18} for a precise statement and the proof for simple singularities via the deformation technique. Later \cite{Gam24} generalizes the result using perverse schobers. 

But for a non-trivial subgroup of $G_W$, Fukaya Seidel category associated with $(W, G)$ is not fully defined because equivariant Morsification is not possible in general (see  \cite{Hab22} for the special case of admissible subgroups for invertible curve singularities that allows equivariant Morsification.) In \cite{CCJ20}, we have given a definition for the case that $W$ is log Fano or log Calabi--Yau type based on a different approach. Unfortunately, invertible curve singularities are log general type
except $F_{2,2}$, and hence the full construction of \cite{CCJ20} does not apply.  Recall that a weighted homogeneous polynomial of weight $(w_{1}, \dots, x_{n}; h)$ is called log Fano if $\sum_{i=1}^{n} w_{i} - h > 0$ and log Calabi-Yau if $\sum_{i=1}^{n} w_{i} - h = 0$. 

\subsection{Results of the paper}
Let us explain our approach, which is somewhat orthogonal to the above literature.
We plan to construct the mirror dual pair $(W^T, G^T)$ from the symplectic geometry of $(W,G)$.

We work with the Milnor fiber $M_W$ of $W$, together with a monodromy automorphism $\rho:M_W \to M_W$,
instead of considering a  Morsification of $W$. It is well-known that $M_W$ has the structure of a Liouville manifold (hence an exact symplectic manifold) and $\rho$ can be chosen to be an exact symplectomorphism. 
By the symplectic geometry of $W$, we mean these two data $M_W$ and $\rho$. 

 Next, we take the quotient of the Milnor fiber $M_W$ by the maximal symmetry group $G_W$. 
 We find that when the symplectic side has the maximal symmetry $G = G_W$, the geometry of the quotient
simplifies dramatically such that it allows mirror construction.  Namely, the quotient orbifold $[M_W/G_W]$ is always a partial orbifold compactification of a pair of pants. This enables us to find its mirror. The general subgroup cases follow from the maximal case by incorporating the dual group action.

We have two main steps corresponding to these two data $M_W$ and $\rho$.
The first step is to construct the mirror disc potential $W_\bL$ of the orbifold quotient $[M_W/G_W]$ of the Milnor fiber.
The second step is to define a symplectic cohomology class $[\Gamma]$ of $[M_W/G_W]$ from the monodromy $\rho$,
and use its closed-open image to define another disc potential function $g$. 
Then, the restriction of  $W_\bL$ to
the hypersurface defined by $g=0$ becomes the desired transpose polynomial $W^T$.

  \begin{thm}
  \label{thm: summary}
  Given an invertible curve singularity $W$ with maximal symmetry group $G_W$, its mirror $W^T$ can be obtained as follows.
\begin{enumerate}  
\item There exist an weakly unobstructed immersed Lagrangian submanifold $\bL$ in  the quotient orbifold $[M_W/G_W]$
whose  disc potential function is
   \[
  W_\bL = \left.
  \begin{cases}
   x^p+y^q+xyz,   & \text{for } \;\;F_{p,q}  \\
   y^q+xyz,  & \text{for } \;\;C_{p,q} \\
    xyz , & \text{for } \;\;L_{p,q}
  \end{cases}
   \right. 
\] 
\item There exists a Hamiltonian orbit $\Gamma_W$ on the orbifold $[M_W/G_W]$ whose image under Kodaira-Spencer map (decorated with bounding cochains) is given by  $g(x,y,z) \cdot 1_\bL$ where
  \begin{equation}
 g(x,y,z) =
  \begin{cases} \label{eqn:g}
   z   &  \text{for } \;\;F_{p,q}  \\
   z-x^{p-1}  & \text{for } \;\;C_{p,q} \\
    z-x^{p-1} -y^{q-1} & \text{for } \;\;L_{p,q}
  \end{cases}
  \end{equation}
  \item Transpose polynomial $W^T$ can be obtained by restricting $W_\bL$ on the hypersurface $g=0$:
   $$W_\bL\vert_{g=0}(x,y)=W^T(x,y).$$
\end{enumerate}
 \end{thm}
 
The above construction allows us to construct the following explicit functorial correspondences as well.

 \begin{thm}[Theorem \ref{thm:um}] \label{thm:intro1}
 Let $W$ be an invertible curve singularity.
 \begin{enumerate}
 \item  We have a localized mirror $\AI$-functor which gives a derived equivalence
$$\mathcal{F}^\bL:  \mathcal{WF}([M_W/G_W]) \to \mathcal{MF}(W_\bL)$$
 \item  The restriction of $W_\bL$ to the hypersurface $\{g=0\}$ defines a corresponding functor between their matrix factorization categories, and when composed with $\mathcal{F}^\bL$, it defines an $\AI$-functor
 $$\mathcal{G}: \mathcal{WF}([M_W/G_W]) \to \mathcal{MF}(W^T)$$
\end{enumerate}
\end{thm}
The functor $\mathcal{F}^\bL$ realizes homological mirror symmetry between the orbifold $[M_W/G_W]$ and the Landau-Ginzburg model $W_\bL$. The functor $\mathcal{G}$ does not give an equivalence, and
for this, we need a symplectic counterpart, which is a mirror to the restriction to the hypersurface $\{g=0\}$.
See Theorem \ref{thm:c1} and Proposition \ref{prop:isos} below.

As an application, we find a relation between Auslander--Reiten theory \cite{Auslander1975} of matrix factorizations of $W^T$ and symplectic geometry of $W$.
Recall that the category $ \mathcal{MF}(W^T)$   is equivalent to a category of \textit{maximal Cohen--Macaulay modules} of a Cohen--Macaulay ring $R:=\C[x_1,\cdots,x_n]/(W^T)$ \cite{E80}.  It has been studied intensively in the 80s with tremendous success, including the classification problem of indecomposable MCM modules.  Greuel--Kn\"orrer and Kn\"orrer showed that for ADE (simple) singularities, there are only finitely many indecomposable MCM modules \cite {GK}, \cite{KN87}. The Auslander--Reiten quiver records indecomposable objects as well as irreducible morphisms between them and the cases of ADE curve singularities were computed by Dieterich--Wiedemann \cite{DW} and Yoshino \cite{Yo}.
 We investigate the geometry behind such Auslander--Reiten quiver.

  \begin{thm}[Theorem \ref{thm:ADEAR}]
When  $W^T$ is ADE curve singularity, we find explicit Lagrangians in $[M_{W}/G_W]$ which corresponds to indecomposable matrix factorizations in the AR quiver under the $\AI$-functor $\mathcal{G}$. Moreover, Auslander--Reiten almost split exact sequences are realized as Lagrangian surgery exact sequences.
\end{thm}
Buchweitz--Greuel--Schreyer proved the converse statement of Kn\"orrer by constructing infinitely many non-isomorphic MCM modules for non-simple cases \cite{BGS87}. It would be very interesting to find a geometric reason for such phenomena, which we leave for future investigation.

 \begin{figure}[h]
\begin{subfigure}{0.43\textwidth}
\includegraphics[scale=0.5]{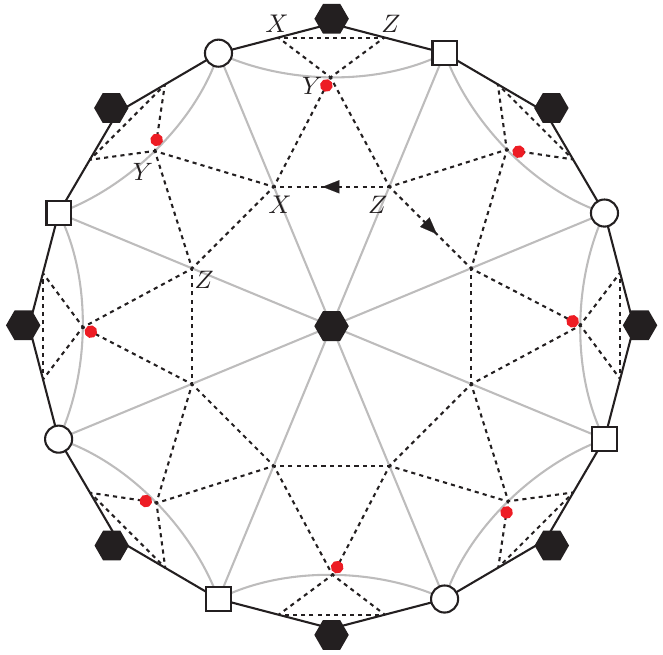}
\centering
\end{subfigure}
\begin{subfigure}{0.43\textwidth}
\includegraphics[scale=0.5]{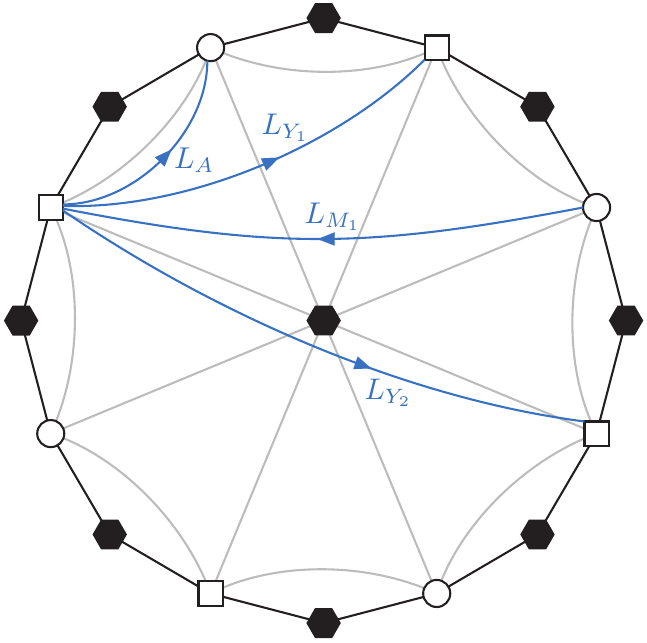}
\centering
\end{subfigure}
\centering
\caption{\label{fig:dncell} Milnor fiber for $D_5^T=C_{2,4}$ and Lagrangians for indecomposable MF's}
\end{figure}

From now on, let us explain our construction in more detail.
\subsection{Description of the Milnor fiber}
Let $M_{W} := W^{-1}(1)$ be the Milnor fiber of $W$. When $W$ is a curve singularity, its Milnor fiber is given by a punctured Riemann surface.   We compactify $M_W$ to $\WH{M}_W$ by adding points for punctures and extend $G_W$ action on it.
We determine the topology of $M_W$. More importantly, we find that the quotient of $M_W$ by $G_W$ is an orbifold partial compactification of a pair of pants. 
\begin{prop}\label{prop:abc}
For invertible curve singularities, the genus $g$ and the number $k$ of punctures of the Milnor fiber $M_W$ are given as follows. Also the quotient $[\WH{M}_W/G_W]$ is
an orbifold projective line $ \mathbb{P}^1_{a,b,c}$: 
\begin{align*}
\textrm{(Fermat)} \;\;\; g &= \frac{\mu_F +1 -d}{2}, & k&=d,& (a,b,c) &= \left(p,q,\textstyle{\dfrac{pq}{d}}\right) \;\textrm{for}\; d= \gcd(p,q). \\
\textrm{(Chain)} \;\;\; g &= \frac{\mu_C -d}{2}, & k&=d+1,& (a,b,c) &= \left(pq,q,\textstyle{\dfrac{pq}{d}}\right) \;\textrm{for}\; d= \gcd(p-1,q). \\
\textrm{(Loop)} \;\;\; g &= \frac{\mu_L -1 -d}{2}, & k&=d+2,& (a,b,c) &= \left(pq-1,pq-1,\textstyle{\dfrac{pq-1}{d}}\right) \;\textrm{for}\; d= \gcd(p-1,q-1). \\[-5mm]
\end{align*}
Here, $c$ vertex for Fermat, $a,c$-vertices for Chain, $a,b,c$-vertices for Loop are punctures of $[M_W/G_W]$.
\end{prop}

In fact, we describe an equivariant tessellation of the Milnor fiber explicitly by  investigating the monodromies near branch points as follows:
 The equator of $\mathbb P_{a,b,c}^1$ contains three orbifold points and divides the orbi-sphere into two cells.
From the orbifold covering $\WH{M}_W \to \mathbb P_{a,b,c}^1$ and considering lifts of these two cells, we obtain a tessellation of Milnor fibers of invertible curve singularities. 
Let us give a combinatorial description of the tessellation of $\WH{M}_W $ as well as $G_W$-action on it.

Consider a $2m$-gon whose boundary edges are labelled by  $a_{1}, \ldots, a_{2m}$ ordered and oriented in a counterclockwise way.
We say edges are identified as $\pm(2p-1)$ pattern if $a_{2k}$  and  $(a_{2k+2p-1})^{op}$ are identified, and
$a_{2k-1}$ and $(a_{2k-2p})^{op}$ are identified for any $k$. Here, indices are modulo $2m$, and $a^{op}$ is the orientation reversal of the edge. Note that even and odd numbered edges play different roles.
See Figure \ref{fig3} (B) for $16$-gon identified as $\pm 7$ pattern.

\begin{thm}\label{thm:tess}
The compactified Milnor fiber $\WH{M}_W$ and a tessellation on it are explicitly described as follows.  
\begin{enumerate}
\item(Fermat) $\WH{M}_{F_{p,q}}$ is given by  $(2pq-2p)$-gon with edges identified as $\pm(2q-1)$ pattern.
An odd-numbered edge corresponds to an oriented path from $a$-vertex to $c$-vertex in the quotient.
\item (Chain) $\WH{M}_{C_{p,q}}$ is given by   $(2pq)$-gon with edges identified as $\pm(2p-1)$ pattern.
An odd-numbered edge corresponds to an oriented path from $b$-vertex to $c$-vertex in the quotient.
\item (Loop) $\WH{M}_{L_{p,q}}$ is given by $2(pq-1)$-gon with edges identified as $\pm(2p-1)$ pattern.
An odd-numbered edge corresponds to an oriented path from $b$-vertex to $c$-vertex in the quotient.
\end{enumerate}
\end{thm}

\subsection{Homological Mirror symmetry for the Milnor fiber quotient}
Recall that given a Lagrangian submanifold $L$ of a symplectic manifold $M$, we can define its disc potential function $W_L$ as a function on the Maurer-Cartan space (deformation space) of $L$
(see \cite{CO}, \cite{FOOO}). This is, by now a standard method of constructing a mirror Landau-Ginzburg model in simple cases.

Also, recall that the Milnor fiber quotient $[M_W/G_W]$ is a partial orbifold compactification of a pair of pants. The mirror symmetry of full orbifold compactification of a pair of pants is now well understood.
Seidel  considered an immersed Lagrangian submanifold $\bL$ on it, called Seidel Lagrangian,
which was used to prove the mirror symmetry of the genus two curve (see \cite{Se} and \cite{Efimov}).

The first author with Hong and Lau \cite{CHL} showed that for $\mathbb{P}_{a,b,c}^1$, Seidel Lagrangian $\bL$ is unobstructed and has a disc potential function $W_\bL$ such that the canonically defined $\AI$-functor, called localized mirror functor from Fukaya category of  $\mathbb{P}_{a,b,c}^1$ to the matrix factorization category of $W_\bL$ becomes an equivalence, realizing the homological mirror symmetry. 
(see also \cite{ACHL} for the mirror symmetry between big quantum cohomology and Jacobian rings).

The method of \cite{CHL} extends easily in our case of partial orbifold compactification of a pair of pants.
The potential $W_\bL$ in Theorem \ref{thm: summary} (1) is obtained as a disc potential function of the Seidel Lagrangian $\bL$ in $[M_W/G_W]$ (see Figure \ref{fig:fabc}). Furthermore, we have a localized mirror $\AI$-functor, and
hence the homological mirror symmetry for quotients of Milnor fibers in Theorem \ref{thm:intro1} (1)
can be proved  by showing that
a split-generator of the wrapped Fukaya category of the quotient is mapped to a split-generator of the matrix factorization category of a mirror potential under the localized mirror functor $\mathcal{F}^\bL$.

\begin{figure}[h]
\includegraphics[scale=1]{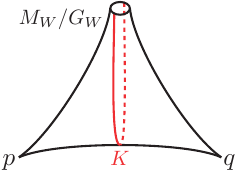}
\centering
\caption{\label{fig:splitgen} Split generator $K$ in Fermat cases $F_{p,q}$}
\end{figure}

\subsection{Floer theory for the monodromy of an isolated singularity}
Note that the Milnor fiber itself cannot encode all the data of a given singularity. The other information that we use is a monodromy of singularity (instead of a Morsification and associated vanishing cycles).
An invertible singularity has an isolated singularity at the origin, and a geometric monodromy is defined by using a symplectic parallel transport on fibers as we travel around the value 0. There are two natural choices for a geometric monodromy on the boundary $\partial M_W$. One may choose it to be the identity map near $\partial M_W$. When $W$ is a weighted homogeneous polynomial (including invertible singularities), there is a circle action on the link of the singularity given by the weight action.

Now, let us explain how we define another disc potential function $g(x,y,z)$ in Theorem \ref{thm: summary} (2) from a geometric monodromy.  For this, let us recall a construction from our previous work \cite{CCJ20}.
The boundary $\partial M_W$, which can be regarded as a link of the singularity $W$, has a $S^1$ action whose orbits are Reeb orbits of the boundary contact structure. This  $\partial M_W$ family of Reeb orbits defines a symplectic cohomology class (in Morse-Bott setting) of the quotient orbifold $[M_W/G_W]$. We denote it by $\Gamma_W$ and call it the monodromy class (see Subsection \ref{subsec:Gamma}).

In singularity, there is a variation operator that relates homology cycles relative to the boundary with compact homology cycles. In \cite{CCJ20}, we considered a quantum cap action (see \cite{auroux07}, \cite{Ga}) by $\Gamma_W$,
\begin{equation}\label{eq:capi}
\cap \Gamma_{W} :  \mathcal{WF}([M_W/G_W]) \to  \mathcal{WF}([M_W/G_W])
\end{equation}
as a quantum analogue of the variation operator. Furthermore, allowing arbitrary many insertions of $\Gamma_W$ on popsicle discs, we have defined a new Fukaya $\AI$-category associated with $(W,G_W)$
when $W$ is log Fano or log Calabi-Yau type.  Invertible curve singularities are log general type, but
we will see that we still have an $A_3$-algebra structure in Section 8.

In our case at hand, we will consider the closed-open map image of the monodromy class $\Gamma_W$
on $\bL$. More precisely, we can decorate the boundary of the closed-open disc with Maurer-Cartan elements of the Seidel Lagrangian $\bL$ to define so-called the Kodaira--Spence map (see \cite{ACHL} for more details).
Its image becomes a multiple of a unit, and its coefficient is the polynomial $g(x,y,z)$ in Theorem \ref{thm: summary} (2).
  \begin{thm}[Theorem \ref{thm:ksg}] \label{thm:intro2}
  The  Kodaira--Spencer map 
  \[\mathrm{KS}^\bold b: SH^\bullet\big([M_W/G_W]\big)  \to H^\bullet(CF(\bL, \bL)\otimes \C [x,y,z], m_1^{\bb,\bb}) \]
  sends $\Gamma_{W}$ to $g(x,y,z) \cdot 1_\bL$.
  \end{thm}
      \begin{figure}[h]
    \centering
    \includegraphics[scale=0.5, trim=0 550 0 20, clip]{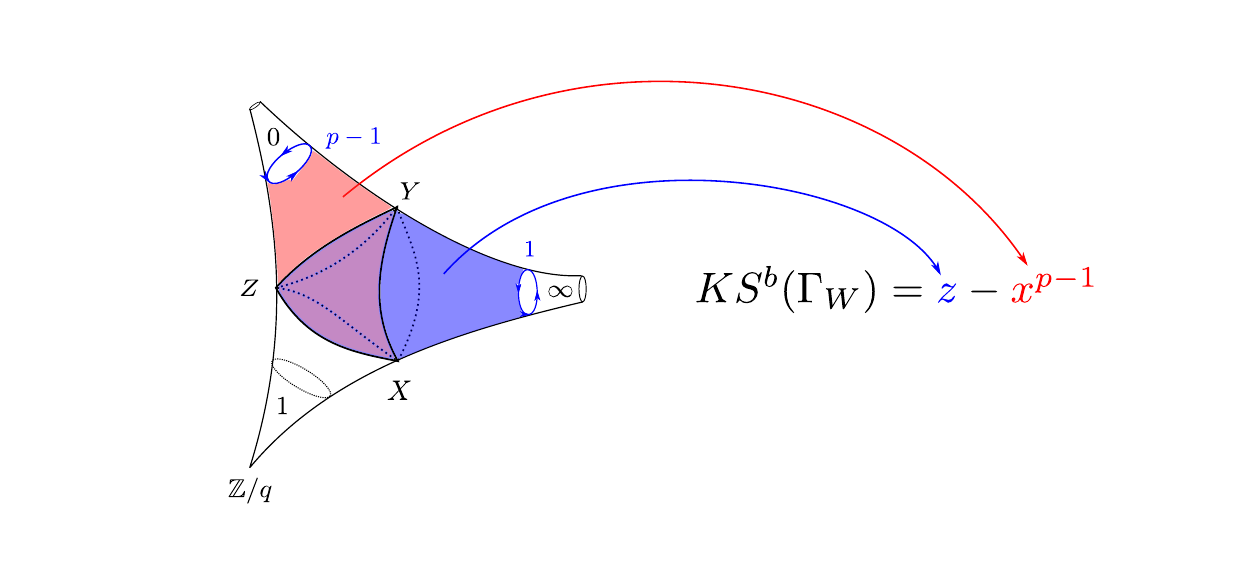}
    \caption{\label{fig:KS}$\mathrm{KS}^\bold b(\Gamma_W)$ for a chain type singularity}
    \end{figure}
Surprisingly, the transpose polynomial $W^T$ and $W_\bL$ are related by 
$$W^T =W_\bL- xy \cdot g(x,y,z).$$
  
\subsection{Quantum cap action of $\Gamma_W$ is mirror to the hypersurface restriction}
From the previous identity, we can obtain $W^T$ by restricting $W_\bL$ on the hypersurface $\{g=0\}$.
We show that the mirror counterpart of this restriction is taking the cone of the quantum cap action by $\Gamma_W$. 
First, $\cap \Gamma_W$-action in \ref{eq:capi} is a map of $\AI$-bimodules over $\WF([M_W/G_W])$.
Similarly, multiplication by $g(x,y,z)$ defines an $\AI$-module map between   $\mathcal{MF}(W_\bL)$ over itself. Localized mirror functor $\mathcal{F}^\bL$ intertwines them up to homotopy.
  \begin{thm}\label{thm:c1}
  There is a homotopy commuting diagram of $\AI$-bimodules:
  \[\begin{tikzcd}
    \WF([M_W/G_W]) \arrow[d, "\cF^\bL"] \arrow[r, "\cap \Gamma_W"] & \WF([M_W/G_W]) \arrow[d, "\cF^\bL"] \arrow[r]& \cF(W, G_W) \arrow[d, "\widetilde{\cF}^\bL"] \arrow[r]& \phantom a \\
    \mathcal{MF}(W_\bL) \arrow[r, " \cdot g"] & \mathcal{MF}(W_\bL) \arrow[r] & \mathcal{MF}(W_\bL)|_g \arrow[r] & .
    \end{tikzcd}\]
    Here, the third one in each row denotes the mapping cone of the first morphism, and all vertical lines are quasi-isomorphisms. Furthermore, $\mathcal{MF}(W_\bL)|_g \simeq \mathcal{MF}(W^T)$.
  \end{thm}
Hence, this theorem proves an $\AI$-bimodule version of Berglund--H\"ubsch homological mirror symmetry for invertible polynomials of two variables.

\subsection{Homology category for $(W,G)$}
In \cite{CCJ20}, this $\AI$-bimodule  $\cF(W, G_W)$ in fact has a new $\AI$-category structure, whose $\AI$-maps are defined by popsicle maps with $\Gamma_W$ insertions, when $W$ is log Fano or log Calabi-Yau. 
Invertible curve singularities  are log general type, and hence there may exist non-trivial obstructions for $\AI$-structure.  We show that the construction of \cite{CCJ20} defines at least an $A_3$-structure. Here $A_{3}$-structure means that we have three operators $M_{1}$, $M_{2}$ and $M_{3}$ of the usual $\AI$-category setup so that they satisfy the first three $\AI$-identities. This implies that $M_1$ gives a differential, $M_2$ gives a product with Leibniz rule, and $M_3$ gives the homotopy for the associativity of the $M_2$ product on cohomology.   In particular, we can define at least the homology category $\mathcal{H}(W,G)$ for invertible curve singularities.

\begin{thm}[Theorem \ref{thm:A3}]
For given invertible curve singularity $W$ except three cases $F_{3,3}$, $C_{3,2}$, and $L_{2,2}$, there exists an $A_{3}$-category structure on $\cF(W, G_W)$. In particular, its homology category $\cH(W,G_W)$ is defined. 
\end{thm}

Now, we can check whether  Berglund--H\"ubsch homological mirror symmetry holds on the level of homology categories. 
Especially, we can choose a split-generator $K$ of $\WF([M_W/G_W])$ as in Figure \ref{fig:splitgen} for any invertible curve singularity and consider it as the object of $\cH(W, G_{W})$. The functor $\mathcal{G}$ in Theorem \ref{thm:intro1} defines a corresponding matrix factorization $\mathcal{G}(K)$ of $W^T$.
 Then, we have the following quasi-isomorphism of morphism spaces.
\begin{prop}[Theorem \ref{thm:qiso}]\label{prop:isos}
The following isomorphism holds for a given split-generator $K$:
$$\Hom_{\cH(W, G_W)}(K,K) \cong \Hom_{\mathcal{MF}(W^T)}(\mathcal{G}(K),\mathcal{G}(K)).$$
\end{prop}
We conjecture that $\cF(W, G)$ becomes indeed an $\AI$-category, and the above quasi-isomorphism extends to the $\AI$-equivalence of categories. 


\subsection{Structure of the paper}
In Section \ref{sec:eqtop}, we investigate an equivariant topology of Milnor fibers of invertible curve singularities and give a combinatorial description of
the equivariant tessellations of the Milnor fibers with respect to the maximal diagonal symmetry group $G_W$-action. In particular, it is shown that the quotient of Milnor fiber $M_W$ by $G_W$-action is an orbifold sphere with three special points which are either orbifold points are punctures. In Section \ref{sec:eqFloer}, we recall the setup of relevant Floer theory  and localized mirror functor.
We prove homological mirror symmetry for quotients of Milnor fibers in Section \ref{sec:HMSMilnor}.
This is before we take monodromy into consideration.

In Section \ref{sec:KS}, we briefly recall some constructions given in \cite{CCJ20}.
After, we investigate how the HMS for Milnor fiber intertwines monodromy information in Section \ref{sec:BHMS}.
For this purpose, we show that the cap action by $\Gamma_W$ in the symplectic side is homotopic to the multiplication by a polynomial $g(x,y,z)$ in the matrix factorizations under localized mirror functor.
Next, we find in Section \ref{sec:AR} explicit Lagrangians and surgery exact triangles which correspond to indecomposable matrix factorizations and Auslander--Reiten almost split exact sequences.
Lastly, we construct an $A_{3}$-category structure on new Fukaya category $\cF(W, G_W)$ in Section \ref{sec:A3}. We see that its cohomology category is closely related to the matrix factorization category of $W^{T}$.

In Appendix \ref{sec:ac}  provided by Osamu Iyama, it is proved that any short exact sequence between indecomposable matrix factorizations that involves same modules as  one of the AR exact sequences
should be isomorphic to it. 


\subsection{Acknowledgement}
We would like to thank  Osamu Iyama for the discussions on matrix factorizaton and providing us the
appendix. We would like to thank Otto van Koert, Hanwool Bae for the discussion on symplectic cohomology theory on orbifolds, and popsicle compactifications. We would like to thank Atsushi Takahashi, Philsang Yoo, Kaoru Ono for helpful discussions. The second author was partially supported by the T.J.Park science fellowship grant. 

 \section{Equivariant topology of Milnor fibers for invertible curve singularities}\label{sec:eqtop}

Let $W$ be an invertible polynomial of two variables. In this section, we first describe the topology of its Milnor fiber $M_W=W^{-1}(1)$ and their maximal symmetry group $G_W$. We show in Proposition \ref{prop:abc} that the quotient $[M_W/G_W]$ is homeomorphic to an orbifold sphere with
three special points, which are either orbifold points or (orbifold) punctures. For later calculations, we give an explicit description of the tessellation on $M_W$ induced from the associated orbifold covering $M_W \to [M_W/G_W]$.


\subsection{Topology of Milnor fiber}
Recall that Milnor fiber is homotopy equivalent to the bouquet of $\mu$-circles where $\mu$ is the Milnor number of the singularity. 
\begin{lemma}
The weights (up to $\gcd$) and Milnor numbers of curve singularities are as follows.
\begin{enumerate}
\item Weights of $F_{p,q}$ are $(q,p;pq)$. Its Milnor number is $\mu_F = (p-1)(q-1)$.
\item Weights of $C_{p,q}$ are  $(q,p-1;pq)$. Its Milnor number is  $\mu_C = pq - p+1$.
\item Weights of $L_{p,q}$ are  $(q-1,p-1;pq-1)$. Its Milnor number is $\mu_L = pq$.
\end{enumerate}
\end{lemma}
\begin{proof}
Milnor numbers can be easily computed by the following formula.
 \begin{thm}[\cite{MO}]
 Let $f(x_1,\ldots,x_{n})$ be the weighted homogeneous polynomial of
 weights $(w_1,\ldots,w_{n},h)$. Then, it has an isolated singularity at the origin whose Milnor number is given by
$$\mu = \left(\frac{h}{w_1} -1\right) \cdots \left(\frac{h}{w_{n}} -1\right).$$
\end{thm}
\end{proof}

The maximal symmetry group $G_{W}$ has another description:
$$G_{W} = \left\{ (\lambda_{1}, \lambda_{2}) \in \mathbb{C}^{*} | \;W(\lambda_{1}x_{1}, \lambda_{2}x_{2}) = W(x_{1},x_{2})\right\}.$$
Then, it is easy to check the following.

\begin{lemma}
$G_{F_{p,q}} \simeq \mathbb{Z} / p  \oplus \mathbb{Z} / q, \hskip 0.2cm G_{C_{p,q}} \simeq \mathbb{Z} / pq $ and $G_{L_{p,q}} \simeq \mathbb{Z} / (pq-1)$.
\end{lemma}
\begin{proof}
$G_{F_{p,q}} = \left \{ \left ( \text{exp} \left ( \frac{2 k \pi i}{p} \right ),  \text{exp} \left ( \frac{2 l \pi i}{q} \right ) \right ) \Big | \, 0 \leq k \leq p-1, 0 \leq l \leq q-1 \right \}$. We can take generators of $G_{C_{p,q}}$  and $G_{L_{p,q}}$ as 
$(\xi^{q}, \xi^{-1}) $ and $(\eta^{q}, \eta^{-1})$ respectively 
for $\xi = \text{exp} \left ( \frac{2 \pi i}{pq} \right ), \eta =   \text{exp} \left (  \frac{2 \pi i}{pq-1} \right )$.

\end{proof}

Now, we give a proof of Proposition \ref{prop:abc}, which determines the topological type of $M_W$ and $[M_W/G_W]$.
\begin{proof}
It is well-known that the number of punctures is the same as the number of irreducible factors of $W$.
Recall that for sufficiently small $r$ and $0< \epsilon \ll r$, the link $W^{-1}(0)\cap S^{2n-1}_r$ and
$W^{-1}(\epsilon)\cap S^{2n-1}_r$ are diffeomorphic, and each factor of $W$
gives a boundary component of $W^{-1}(0) \cap B^{2n}_{r}$. For Fermat type, 
$x^p+y^q$ factors into $d$ factors for $d=\gcd(p,q)$.
For chain type, since $x^p+xy^q= x(x^{p-1}+y^q)$, $C_{p,q}$ has  $d+1$ factors
with $d =\gcd(p-1,q)$. For loop type, since $x^py+xy^q=xy(x^{p-1}+y^{q-1})$,
$L_{p,q}$ has  $d+2$ factors
with $d =\gcd(p-1,q-1)$.

To compute the genus, note that  $M_W$ is obtained by removing $k$ points from  $\WH{M}_W$.
Hence, Euler characteristic is $\mathcal{E}(M_W) = \mathcal{E}(\WH{M}_W) - k$.
But $M_W$ has the homotopy type of bouquet of $\mu$-circles for the Milnor number $\mu$,
and its Euler characteristic $\mathcal{E}(M_W) = 1 -\mu$.
Therefore, the equality $2-2g -k = 1-\mu$ holds, and the genus of $\WH{M}_W$ (and hence $M_W$) is $g= (\mu+1-k)/2$.

Now, to get the quotient orbifold $[\WH{M}_W/G_W]$, we first observe that there are exactly three orbits of $G_W$ with non-trivial stabilizer in $\WH{M}_W$ and show that the quotient has genus zero using
orbifold Euler characteristic. We will use the fact that $\mathcal{E}(\WH{M}_W)/|G_W|$
equal the orbifold Euler characteristic of $[\WH{M}_W/G_W]$.

Let us consider the Fermat case. Orbits of $[(0,1)]$ and $[(1,0)]$ give two singular orbits 
of $\Z/p \oplus \Z/q$-action on $M_W$. They have stabilizers $\Z/p, \Z/q$ respectively.
For $d=\gcd(p,q)$, since $\Z/p \oplus \Z/q$ acts transitively on the irreducible components, it also acts transitively on $d$ punctures (or compactification points in $\WH{M}_W$).
So the quotient has three orbifold points $(a,b,c)=(\Z/p,\Z/q, \Z/(pq/d))$.
To see that the quotient is $ \mathbb{P}^1_{a,b,c}$, 
$$\mathcal{E}(\WH{M}_W) = 2-2g = k+1-\mu = d +1 - (p-1)(q-1)$$
Note that it equals $|G| \cdot \mathcal{E}_{orb}( \mathbb{P}^1_{a,b,c})$
which is
$$ (pq) \cdot \left(\frac{1}{a} + \frac{1}{b} + \frac{1}{c} -1\right) = pq \cdot \left(\frac{1}{p} +\frac{1}{q} + \frac{d}{pq} -1\right)$$
This proves the claim for the Fermat case.

The other cases are similar.
For the chain case, the orbit of $[(1,0)]$ has stabilizer $\Z/q$. The other two orbifold points come from punctures. Note that $C_{p,q}$ is a product of $x$ and $x^{p-1}+y^q$. It is easy to see that
$G_W$ action preserves each branches $x=0$ as well as $x^{p-1}+y^q=0$. Hence, the puncture corresponding to the branch $x=0$ has the full group $G_f$ as a stabilizer, and
the other $d$ punctures (for the factors of $ x^{p-1}+y^q=0$ with $d =\gcd(p-1,q)$) are acted by $G_f$ in a transitive way. Therefore, the orbifold point has stabilizer $\Z/(pq/d)$.
For the loop type, $M_W$ has no fixed point of $G_W$-action, and the punctures for
factors $x,y, x^{p-1}+y^{q-1}$ form three orbits with stabilizer $\Z/(pq-1), \Z/(pq-1), \Z/((pq-1)/d)$.
This finishes the proof.
\end{proof}


\subsection{Orbifold covering}\label{subsec: orbicovering}
We observed that $G_W$ acts on the Milnor fiber $M_W$ to produce the regular orbifold covering
$\WH M_W \to \mathbb P_{a,b,c}^1$.
Given a Riemann surface, there may be two non-equivalent group actions with isomorphic
quotient space (see Broughton \cite{Br} for example). Hence, to determine the $G_W$-action on $M_W$
explicitly,  we find an explicit group homomorphism 
 \begin{equation}\label{eq:orb1}
\phi: \pi_1^{\mathrm{orb}}\left( \mathbb{P}^1_{a,b,c}\right)  \to G_W.
\end{equation}
For the kernel $\Gamma =\textrm{Ker}(\phi)$, $\WH{M}_W$ is an orbifold covering of  $\mathbb{P}^1_{a,b,c}$   corresponding to the kernel $\Gamma$ with deck transformation group $G_W$.

We use the following presentation of the orbifold fundamental group of $\mathbb{P}^1_{a,b,c}$
\begin{equation}\label{pi1orb}
 \pi_1^{\mathrm{orb}}\left( \mathbb{P}^1_{a,b,c}\right)=\left\langle\gamma_1,\gamma_2,\gamma_3 \mid
\gamma_1^a = \gamma_2^b = \gamma_3^c = \gamma_1\gamma_2\gamma_3=1\right\rangle. 
\end{equation}
Here, $\gamma_1$ is a small loop going counter-clockwise around $0\in \mathbb P^1$, $\gamma_2$ is for $1\in \mathbb P^1$ and $\gamma_3$ is for $\infty \in \mathbb P^1$. Later on, this presentation will serve as an additional grading on a Floer theory.
\begin{prop}\label{prop:homo}
The homomorphism \eqref{eq:orb1} is given as follows.
\begin{enumerate}
\item  (Fermat) For the covering $M_{F_{p,q}} \to \mathbb P^1_{p,q,\frac{pq}{\gcd(p,q)}}$, we have $$\phi(\gamma_1) = (1,0), \;  \phi(\gamma_2) =(0,1), \;  \phi(\gamma_3) = (-1,-1) \in \Z/p \times \Z/q$$
identified with $G_W$ by $(k,l) \to \left(e^\frac{2\pi k i}{p}, e^\frac{2\pi l i}{q}\right)$.
\item (Chain) \; For the covering $M_{C_{p,q}} \to \mathbb P^1_{pq,q,\frac{pq}{\gcd(p-1,q)}}$, we have
 $$\phi(\gamma_1) = 1, \; \phi(\gamma_2) = -p,\;  \phi(\gamma_3) = p-1 \in \Z/pq$$
identified with $G_W$ by $k \to \left(e^\frac{2\pi k i}{p}, e^\frac{-2\pi k i}{pq}\right)$.
\item (Loop) \;\; For the covering $M_{L_{p,q}} \to \mathbb P^1_{pq-1,pq-1,\frac{pq-1}{\gcd(p-1,q-1)}}$, we have
$$\phi(\gamma_1) = 1,\; \phi(\gamma_2) =-p,\;  \phi(\gamma_3) = p-1 \in \Z/(pq-1)$$
identified with $G_W$ by $k \to \left(e^\frac{2\pi kq i}{pq-1}, e^\frac{-2\pi k i}{pq-1}\right)$.
\end{enumerate}

\end{prop}
Let us give the proof in each case separately.
\subsubsection{Fermat type $F_{p,q}$}
$M_{F_{p,q}}$ is a locus of an equation $x^p+y^q=1$. We regard them as a Riemann surface of a multivalued function 
\[y= (1-x^p)^{\frac{1}{q}},  \]
with $q$ branch  points  $x_k=e^\frac{2k\pi i}{q}$ (for $k=0,\ldots, q-1$).
We connect each branch points $x_k$  with $\infty$ by a ray $\{r e^\frac{2k\pi i}{q} \mid r \geq 1\}$. With these branch cuts, $M_{F_{p,q}}$ is a $q$ sheeted covering of a complex plane $\C$.

A fundamental domain of the quotient is the following "pizza" shape domain. 
\begin{equation}\label{fdf}
\left\{x=re^\theta \; \bigg| \; 0 \leq r, \hskip 0.1cm 0\leq\theta\leq\frac{2\pi }{p}\right\}
\end{equation}
There are three distinguished paths $\gamma_i:[0,1] \to \C$.
\begin{itemize}
\item $\gamma_1(t) = \epsilon\cdot e^\frac{2\pi i t}{p}, \hskip 0.2cm (0<\epsilon \ll 1)$, a small path around the origin.
\item $\gamma_2(t) = 1+ \epsilon \cdot e^{2\pi it}, \hskip 0.2cm (0<\epsilon \ll1)$, a small circle around the branch points.
\item $\gamma_3(t) = Re^\frac{-2\pi it}{p}, \hskip0.2cm (R \gg 1)$, a boundary circle with opposite orientation. 
\end{itemize}
These are orbifold loops that correspond to generators of $\pi_1^{\mathrm{orb}}\left(\mathbb P^1_{p,q,\frac{pq}{\gcd(p,q)}}\right)$ in \eqref{pi1orb}.

As we realize $F_{p,q}$ as a $q$ sheeted covering of $\C$, we label those sheets from $1$ to $q$ so that the crossing branch cuts increase the label number by $+1$. Each sheet has $p$ copies of the fundamental domain. We put the label $i_j$ on the following copy;
\[\left\{x=re^\theta \; \bigg| \; 0 \leq r, \hskip 0.1cm \frac{2(j-1)\pi}{p}\leq\theta\leq\frac{2j\pi }{p}\right\} \subset \textrm{$i$-th sheet}.\]
In this setup, we can write down the $\phi$ from a representation of the fundamental group to the group of permutation of the set of labels $\{i_j \; | \; 1\leq i\leq q, \; 1\leq j\leq p\}.$ 

\begin{align*}
\phi:\pi_1^{\mathrm{orb}}\bigg(\mathbb P^1_{p,q,\frac{pq}{\gcd(p,q)}}\bigg) &\to S_{pq}\\
\gamma_1 &\mapsto (1_1, 1_2, \ldots, 1_p)(2_1, 2_2, \ldots, 2_p)\cdots(q_1, q_2, \ldots, q_p)\\
\gamma_2 &\mapsto (1_1, 2_1, \ldots, q_1)(1_2, 2_2, \ldots, q_2)\cdots(1_p, 2_p, \ldots, q_p)\\
\gamma_3 &\mapsto (\gamma_1\circ\gamma_2)^{-1}
\end{align*} 
The image of this representation is isomorphic to $\Z/p\times \Z/q$, generated by $\gamma_1$ and $\gamma_2$. It is compatible with the diagonal symmetry group action in the following way:
\begin{itemize}
\item $\gamma_1$ is a rotation of each sheets by $\frac{2\pi}{p}$. It corresponds to a diagonal action $x\to e^\frac{2\pi i}{p}\cdot x, \hskip 0.2cm y\to y$. 
\item $\gamma_2$ is a rotation of each sheets by $\frac{2\pi}{q}$ so it corresponds to a diagonal action $x\to x, \hskip 0.2cm y\to e^\frac{2\pi i}{q}\cdot y$.
\item $\gamma_3$ corresponds to a diagonal action $x\to e^{-\frac{2\pi i}{p}}\cdot x, \hskip 0.2cm y\to e^{-\frac{2\pi i}{q}}\cdot y$.
\end{itemize}

\subsubsection{Chain type $C_{p,q}$}
$M_{C_{p,q}}$ is a locus of an equation $x^p+xy^q=1$. We regard them as a Riemann surface of a multivalued meromorphic function 
\[y=   \bigg(\frac{1-x^p}{x}\bigg)^{\frac{1}{q}}. \]
This function has $q$ zero branch points $x_k=e^\frac{2k\pi i}{q}$ 
and a single pole branch point $x=0$.

We connect each branch point with $\infty$ by rays as before. Also, we overlap a ray from a pole $x=0$ and a ray from a zero $x=1$. Because they are coming out of different sources, they cancel each other on the overlap. With this choice of branch cuts, $M_{C_{p,q}}$ is a $q$ sheeted covering of $\C^*$. 

A fundamental domain of the quotient $M_{C_{p,q}}/G_{C_{p,q}}$ 
can be taken as the same domain \eqref{fdf} with the puncture at the origin, and orbifold loops $\gamma_1, \gamma_2, \gamma_3$ are the same as in the Fermat case. 

Due to the branch cut along the line segment $[0,1]$ on the real axis, the monodromy representation is different from the Fermat cases. 
It is not hard to see that we get the following symmetric group representation
\begin{align*}
\phi:\pi_1^{\mathrm{orb}}\bigg(\mathbb P^1_{pq,q,\frac{pq}{\gcd(p-1,q)}}\bigg) &\to S_{pq}\\
\gamma_1 &\mapsto (1_1, 1_2, \ldots, 1_p, q_1, q_2, \ldots,3_p, 2_1, 2_2, \ldots, 2_p)\\
\gamma_2 &\mapsto (1_1, 2_1, \ldots, q_1)(1_2, 2_2, \ldots, q_2)\cdots(1_p, 2_p, \ldots, q_p)\\
\gamma_3 &\mapsto (\gamma_1\circ\gamma_2)^{-1}
\end{align*} 
Unlike the Fermat case, $\phi(\gamma_1)$ generates $\phi(\gamma_2)$ by the relation $\phi(\gamma_2) = \phi(\gamma_1)^{-p}.$ Hence the image of $\phi$ is generated by $\gamma_1$, and isomorphic to $\Z/pq$. note that $\phi(\gamma_1)$ rotates each sheet by $\frac{2\pi}{p}$, and change the label of a sheet by $+1$ after you apply it by $-p$ times. Therefore, 
\begin{itemize}  
  \item $\gamma_1$ corresponds to $x\to e^\frac{2\pi i}{p} \cdot x, \hskip 0.2cm y \to e^\frac{-2\pi i}{pq} \cdot y$.
 \end{itemize}

\subsubsection{Loop type $L_{p,q}$}
$M_{L_{p,q}}$ is a locus of an equation $x^py+xy^q=1$. As we cannot realize $L_{p,q}$ as a Riemann surface of a single function, we work with the following parametrization by $z \in \C$
\[x=   \bigg(\frac{z^q}{1-z}\bigg)^{\frac{1}{pq-1}}, \hskip 0.2cm y=  \bigg(\frac{(1-z)^p}{z}\bigg)^\frac{1}{pq-1},\]
with two branch points  $z=0,1$. 
 We connect each two with $\infty$ by half lines and let the one from the origin overlap the one from $z=1$. Then $M_{L_{p,q}}$ is a $pq-1$ sheeted covering of a $z$-plane $\C \setminus \{0,1\}$ as follows. 

A fundamental domain is the whole $z$-plane minus two points $z=0, 1$. The three distinguished paths are
\begin{itemize}
\item $\gamma_1(t) = \epsilon\cdot e^{2\pi i t}, \hskip 0.2cm (0<\epsilon \ll1)$, a small circle around $z=0$.
\item $\gamma_2(t) = 1+ \epsilon \cdot e^{2\pi i t}, \hskip 0.2cm (0<\epsilon \ll 1)$, a small circle around $z=1$.
\item $\gamma_3(t) = Re^{-2\pi it}, \hskip0.2cm (R \gg 1)$, a boundary circle with opposite orientation. 
\end{itemize}
Let us compute the monodromy representation. Whenever we cross a branch cut inside $z$-plane, we change a covering sheet for $x$ and $y$ both. Each of them has $pq-1$ possibilities, so there are $(pq-1)\times(pq-1)$ different sheets. Let's label them by $(i,j), \hskip 0.1cm i,j = 1, \ldots, pq-1$.  But we don't need all of them because  $\pi_1$-orbit of $(1,1)$ consists of only $pq-1$ sheets among them. The monodromy representation of $\pi_1^{\mathrm{orb}}$ is written as  
\begin{align*}
\phi:\pi_1^{\mathrm{orb}}\bigg(\mathbb P^1_{pq-1,pq-1,\frac{pq-1}{\gcd(p-1,q-1)}}\bigg) &\to S_{pq-1}\times S_{pq-1}\\
\gamma_1 &\mapsto (q,-1): (a,b) \to (a+q,b-1)\\
\gamma_2 &\mapsto (-1,p) : (a,b) \to (a-1, b+p)\\
\gamma_3 &\mapsto (\gamma_1\circ\gamma_2)^{-1}
\end{align*} 
Since $\phi(\gamma_2) = \phi(\gamma_1)^{-p}$, the image of this representation is generated by $\gamma_1$ and isomorphic to $\Z/pq-1$. Furthermore, it is easy to see that 
\begin{itemize}
\item $\gamma_1$ corresponds to a diagonal action $x\to e^\frac{2\pi q i}{pq-1}\cdot x, \hskip 0.2cm y\to e^\frac{-2\pi i}{pq-1}y$.
\end{itemize}

 
\subsection{Equivariant tessellation of Milnor fibers}
Now, we give a proof of Theorem \ref{thm:tess}, which describes the equivariant tessellation of the Milnor fibers.
\begin{proof}
Recall that we have $[\WH{M}_W/ G_W]= \mathbb{P}^1_{a,b,c}$ from Proposition \ref{prop:abc}.
Let $H$ be the universal cover of  $\WH{M}_f$ or equivalently that of $\mathbb{P}^1_{a,b,c}$. We have
$\pi_1^{\mathrm{orb}}(\mathbb P_{a,b,c})$ action on $H$. Let $F$ be a fundamental domain in $H$ for this action as in the Figure \ref{fig:f}
where the angle is measured in $S^2$, $\R^2$ or $\mathbb{H}$ depending on the universal cover.
Here $x_1,x_2,x_3$ project down to $a,b,c$ orbifold points. 
We have the full cone angle at two of them, but the cone angle for the other is divided in half.
Also, we will use Proposition \ref{prop:homo} which describes the relation between generators $\gamma_1,\gamma_2,\gamma_3$ of
$\pi_1^{\mathrm{orb}}(\mathbb P_{a,b,c})$ and $G_f$.

\begin{figure}[h]
\includegraphics[scale=0.6]{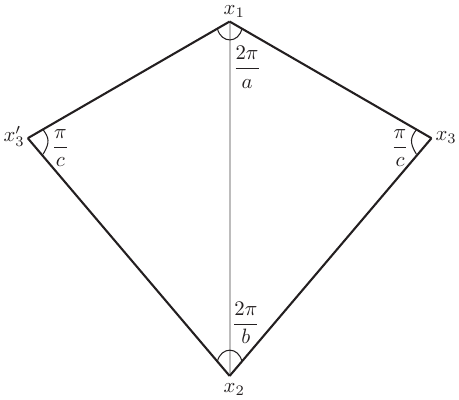}
\centering
\caption{\label{fig:f} Fundamental domain of $\mathbb{P}^{1}_{a,b,c}$ in $\mathbb{H}$}
\end{figure}

For the Fermat case, consider $\gamma_1,\gamma_2 \in \pi_1^{\mathrm{orb}}(\mathbb P_{a,b,c})$
and collect the following $p \times q$ copies of $F$ to define a polygon
$$P:=\left \{ \gamma_{2}^{i} \gamma_{1}^{j} F \, \bigg| \, 0 \leq i \leq p-1, 0 \leq j \leq q-1 \right\}.$$
 First, $P$ is a fundamental domain of $\WH{M}_W$ since $G_W= \Z/p \times \Z/q$ and
$\phi(\gamma_1) = (1,0)$, $\phi(\gamma_2) = (0,1)$ are the generators of $G_W$ by Proposition \ref{prop:homo}.
 
  Also, one can check that $P$ is a $(2pq-2p)$-gon in the following way. 
First, $\{\gamma_{2}^{j} F\}$ for $j=0,\ldots, q-1$ can be glued counter-clockwise way around the vertex $x_2$ of Figure \ref{fig:f}
to form a $2q$-gon, say $Q$. Then, by applying $\{\gamma_1^i\}$ for $i=0,1,\ldots,p-1$ to $Q$, we get $p$-copies of $Q$ glued around the vertex $x_1$
to form a $(2pq-2p)$-gon and this is exactly $P$. Because $2$ edges of $Q$ meeting at the vertex $x_1$ become interior edges, the number of boundary edges decreases by $2p$ from $2pq$.
See Figure \ref{fig:2} (A) for the case of $\WH{M}_{F_{5,2}}$, where $Q$ is given by the union of $F$ and $\gamma_2 F$ and
$P$ is the $10$-gon.

\begin{figure}[h]
\begin{subfigure}{0.43\textwidth}
\includegraphics[scale=0.62]{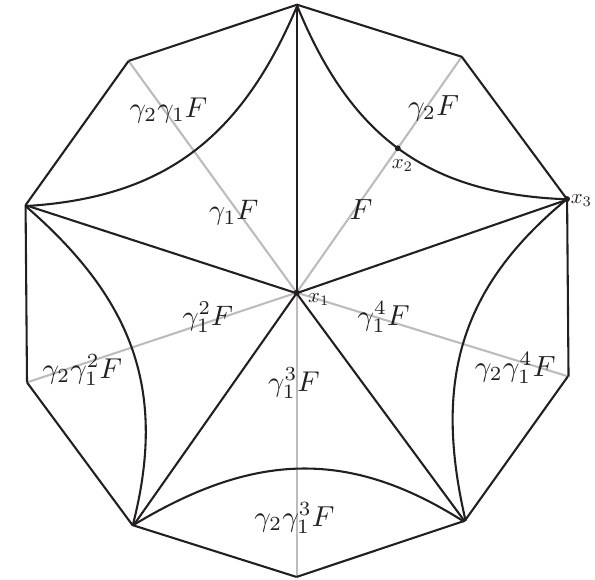}
\centering
\caption{$\WH{M}_{F_{5,2}}$ case}
\end{subfigure}
\begin{subfigure}{0.43\textwidth}
\includegraphics[scale=0.62]{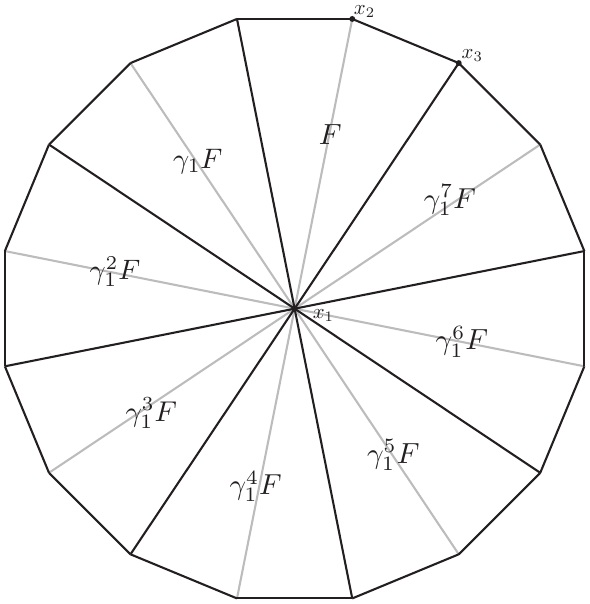}
\centering
\caption{$\WH{M}_{C_{4,2}}$ case}
\end{subfigure}
\centering
\caption{\label{fig:2} Tessellations of $A_{4}$ and $D_{5}$ singularities}
\end{figure}

To find the boundary identification, we consider additional tiles next to $P$.
To see how two boundary edges from $\gamma_{2}^{i} \gamma_{1}^{j} F$ are identified to the remaining edges, we consider the additional rotation action around vertex for $x_1$. 
Consider two tiles $\gamma_{1} \gamma_{2}^{i} \gamma_{1}^{j} F$, $\gamma_{1}^{-1} \gamma_{2}^{i} \gamma_{1}^{j} F$ which are not in $P$. Each tile shares one boundary edge with $\gamma_{2}^{i} \gamma_{1}^{j} F$. Note that $\gamma_{1} \gamma_{2}^{i} \gamma_{1}^{j} F$ and $\gamma_{1}^{-1} \gamma_{2}^{i} \gamma_{1}^{j} F$ are identified with $\gamma_{2}^{i} \gamma_{1}^{j+1} F$ and $\gamma_{2}^{i} \gamma_{1}^{j-1} F$ respectively because of $\phi(\gamma_{1} \gamma_{2}^{i} \gamma_{1}^{j}) = \phi(\gamma_{2}^{i} \gamma_{1}^{j+1})$, $\phi(\gamma_{1}^{-1} \gamma_{2}^{i} \gamma_{1}^{j}) = \phi(\gamma_{2}^{i} \gamma_{1}^{j-1})$. This observation implies a boundary edge $e = (\gamma_{2}^{i} \gamma_{1}^{j})F \cap (\gamma_{1} \gamma_{2}^{i} \gamma_{1}^{j})F$ should be identified with corresponding edge of $\gamma_{2}^{i} \gamma_{1}^{j+1} F$. Similarly, boundary edge $e^{\prime} = (\gamma_{2}^{i} \gamma_{1}^{j})F \cap (\gamma_{-1} \gamma_{2}^{i} \gamma_{1}^{j})F$ is identified with corresponding edge of $\gamma_{2}^{i} \gamma_{1}^{j-1} F$. One can check  that it gives $\pm (2q-1)$ identification on $\partial P$. This proves the proposition for the Fermat case.

\begin{figure}[h] 
\begin{subfigure}{0.4\textwidth}
\includegraphics[scale=0.65]{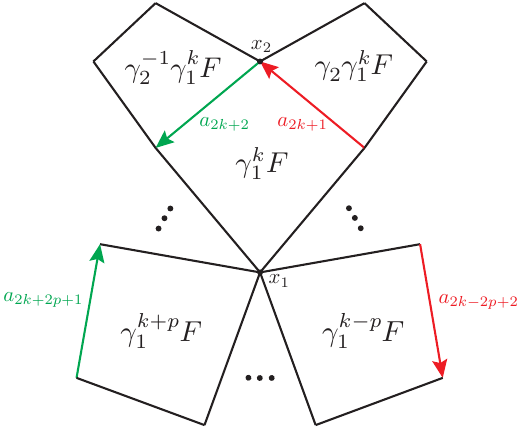}
\centering
\caption{Boundary identification}
\end{subfigure}
\begin{subfigure}{0.4\textwidth}
\includegraphics[scale=0.5]{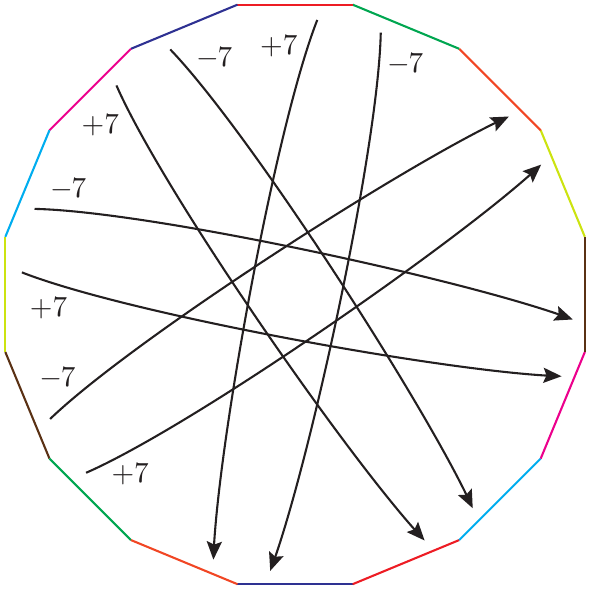}
\centering
\caption{$\pm 7$ pattern}
\end{subfigure}
\centering
\caption{\label{fig3}Chain cases}
\end{figure}

Next, let us discuss the chain type. Recall that we have $G_{C_{p,q}} = \Z/pq$, and $\phi(\gamma_1)=1 \in \Z/pq$ is the generator.  We take the following $pq$-copies of $F$ to define a $2pq$-gon:
$$P:=  \left \{ \gamma_{1}^{i} F \, \mid \, 0 \leq i \leq pq-1 \right \},$$
which is a fundamental domain for the Milnor fiber.
Consider $\gamma_2^{-1} \gamma_1^kF$ and $\gamma_2 \gamma_1^k F$. Since $\phi(\gamma_2) = -p$, we have $\phi(\gamma_2^{-1} \gamma_1^k) = \phi(\gamma_1^{p+k})$. Therefore, $\gamma_2^{-1} \gamma_1^kF$ can be identified with $
\gamma_1^{p+k}F$ as a tile in the Milnor fiber. Hence $x_2x_3$ edge of $\gamma_1^{k}F$ should be identified with $x_2x_3'$ edge of $\gamma_1^{k+p}F$. 
See Figure \ref{fig3} (A). In terms of edges of $P$, this is $+(2p-1)$ identification. From the same argument for $\gamma_2 \gamma_1^k F$, we find that $x_2x_3'$ edge of $\gamma_1^{k}F$ should be
identified with $-(2p-1)$ pattern. This proves the chain case.

For the loop type, we can proceed similarly as in the chain case. We take
$$P:=\left\{ \gamma_{1}^{i} F \, | \, 0 \leq i \leq pq-2 \right\}$$
and we get the same identification as in the chain case from Proposition \ref{prop:homo}.
 \end{proof}

\begin{remark}
A description of $M_W$ is not unique. For example, we will use another shape of the domain in $F_{3,4}$, $F_{3,5}$ and $C_{3,3}$ case (see Section \ref{sec:AR}). A gluing rule for $C_{3,3}$ is given as in the Figure \ref{fig:E7gluing}.
\begin{figure}[h]
\includegraphics[scale=0.5]{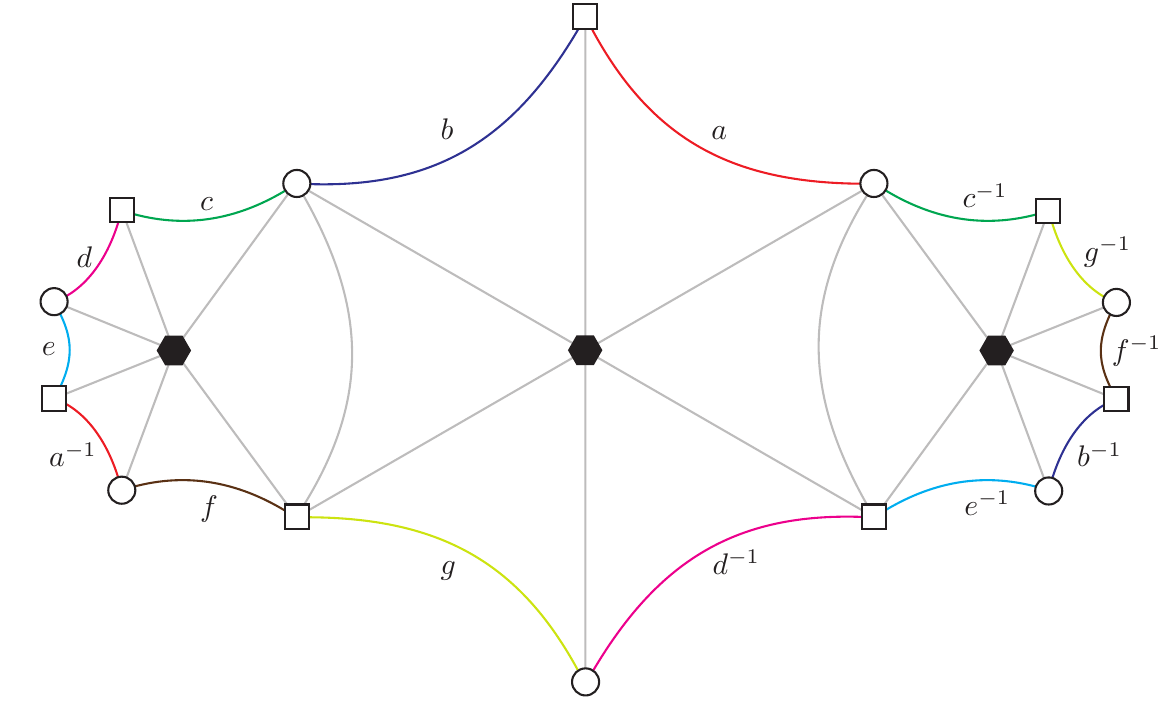}
\centering
\caption{Fundamental domain in $C_{3,3}$ case and gluing}
\label{fig:E7gluing}
\end{figure}
\end{remark}
 
\begin{remark}
We remark that $\WH{M}_W$ is a sphere for $F_{2,2}$, a torus 
for  $F_{3,2}, F_{4,2}, C_{2,2}, C_{3,2}, L_{2,2}$ and higher genus surface for the rest of the cases.
For the last case, the universal cover can be taken as the hyperbolic plane $\mathbb{H}$, and there exists a Fuchsian group $\Gamma$ 
such that $\mathbb{H} / \Gamma \simeq \WH{M}_W$. Furthermore,
 a finite group $G$ acts on $M_W$ if and only if there exist a Fuchsian group $\Gamma'$ and a surjective homomorphism $\phi : \Gamma' \to G_W$ with kernel $\Gamma$ such that $M \simeq \mathbb{H} / \Gamma$ and $M / G \simeq \mathbb{H} / \Gamma'$. It is exactly a homomorphism given in Proposition \ref{prop:homo}. We refer readers to \cite{Katok92} about Fuchsian group and related facts. 
\end{remark}  

From the universal cover of $\WH{M}_W$, we replace compactified points by punctures and get a cover of $M_W$. We will use this cover of $M_W$ in the next section.


 \section{Floer theory for Milnor fiber quotients and localized mirror functor}\label{sec:eqFloer}


  \subsection{$\boldsymbol{\Omega}$- and $\boldsymbol{H_1}$-grading}
We consider additional gradings on $\mathcal{WF}([M_W/G_W])$ following \cite{Se} (with small variation) in addition to $\Z/2$-grading.
We use a holomorphic volume form $\Omega$ on $\mathbb P^1$ with poles of order one at $0,1 \in \mathbb P^1$. This choice provides a trivialization of a tangent bundle $T_{[\Mil_W/G]}$ away from $0,1$ and $\infty$. Time-1 orbits of $H$ are always disjoint from $0, 1, \infty\in \mathbb P^1$ after adding a small time-dependent perturbation term. Therefore, each Hamiltonian orbit still carries an honest cohomological Conley-Zehnder index. We use this integer as a degree.
  
   We can put a grading on Lagrangians and Hamiltonian chords between them in a similar fashion. For a Lagrangian $L$ which is oriented and away from $0, 1, \infty \in \mathbb P^1$, we get a phase map $\overline \phi_L : L \to S^1$ defined by
$$     \overline \phi_L(x) =\frac{\Omega(X)}{|\Omega(X)|} $$
     where $X$ is a nonvanishing vector field on $TL$ pointing positive direction.
   
   \begin{defn}\label{Omega-grading}
   \textit{An $\Omega$-grading} on $L$ is a choice of lift  $\phi_L: L \to \R$
    of a phase map $\overline \phi_{L}$. An \textit{$\Omega$-graded Lagrangian} $L$ is a Lagrangian submanifold with a specific choice of $\Omega$-grading $\phi_L$.
   \end{defn}
   For any time-1 Hamiltonian chord $a\in\chi(L_0, L_1)$ between graded Lagrangian submanifolds, there is the unique homotopy class of Lagrangian path from $T_{L_0, a(0)}$ to $T_{L_1, a(1)}$ compatible to the grading. The absolute Maslov index $\mu_M(a)$ is now well-defined, and we use it as a degree of $a$. 
   
   A discrepancy occurs when we consider a moduli space of discs. A standard index formula starts to read an intersection number of holomorphic maps and pole divisor of $\Omega$. Let 
   \[\overline {\mathcal M}_{m; n,1; [u]}(\gamma_1, \ldots, \gamma_m; a_1, \ldots, a_n, a_0)\]
   be a sub-moduli space of $\overline {\mathcal M}_{m; n,1; [u]}(\gamma_1, \ldots, \gamma_m; a_1, \ldots, a_n, a_0)$ whose relative class is $[u]$. Then, a standard index formula is given by 
   \begin{align*}
   \mathrm{dim}_\R\overline {\mathcal M}_{m; n,1; [u]}(\gamma_1, \ldots, \gamma_m; a_1, \ldots, a_n, a_0) &= (2m+n-2)+\deg a_0 -\sum_{i=1}^n\deg a_i -\sum_{j=1}^n \deg \gamma_j \\
   &+ 2\left(\deg(u, 0) +\deg(u, 1)\right) 
   \end{align*}
   The dimensions of our moduli spaces may differ in even numbers. It breaks a $\Z$-grading into $\Z/2$-grading. 
   
   Meanwhile, there is a topological grading coming from an orbifold cohomology. Recall we use a notation $\gamma_1, \gamma_2, \gamma_3$ to denote a homotopy class of loops winding orbifold point $0, 1$ or $\infty$ respectively (see Figure \ref{fig:fabc} for orientations).  We denote by $H_1^{\mathrm{orb}}$
   the abelianization of $\pi_1^{\mathrm{orb}}$. We get
\begin{equation} \label{eqn:1sthlgy}
   H_1^{\mathrm{orb}} \left( [\Mil_W/G_W] \right) \simeq \left\{ 
   \begin{array}{lc}  \Z\langle\gamma_1, \gamma_2, \gamma_3\rangle/\{p\gamma_1=q\gamma_2=\gamma_1+\gamma_2+\gamma_3=0\} & \mathrm{(Fermat)} \\
    \Z\langle\gamma_1, \gamma_2, \gamma_3\rangle/\{q\gamma_2=\gamma_1+\gamma_2+\gamma_3=0\} & \mathrm{(Chain)}\\
     \Z\langle\gamma_1, \gamma_2, \gamma_3\rangle/\{\gamma_1+\gamma_2+\gamma_3=0\} & \mathrm{(Loop)}   
   \end{array}\right.
\end{equation}
   note that any symplectic cochains, including Morse critical points, can be considered as an element of $H_1^{\mathrm{orb}}$. Moreover, if the Lagrangian submanifold $L$ is simply connected, elements of $CW^\bullet(L, L)$ can also be labeled by $H_1^{\mathrm{orb}}$. Let us call it an \textit{$H_1$-grading}. The Floer theoretic operation uses pseudo-holomorphic curves whose homological boundary is a difference of the homology class of output and inputs. Therefore, we have
   \begin{lemma}\label{topological grading}
     A pseudo-holomorphic curve operation is homogeneous with respect to an $H_1$-grading. 
   \end{lemma}

\subsection{Seidel's Lagrangian $\boldsymbol{\bL}$.}

Since $[M_W/G_W]$ is an orbifold sphere with three special points, we consider an immersed circle, called Seidel Lagrangian $\bL$
and its $\AI$-algebra following Seidel (see Figure \ref{fig:fabc} and \cite{Se}).
$\bL$ is oriented so that the edges of the front triangle in Figure \ref{fig:fabc} are oriented counter-clockwise. We briefly recall the algebra structure of $CF^\bullet(\bL, \bL)$.  It has three immersed generators $X,Y,Z$ of odd degree,  $\bar{X}=Y\wedge Z, \bar{Y}=Z\wedge X, \bar{Z}=X\wedge Y$ of even degree. 

\begin{figure}[h]
\includegraphics[scale=0.6]{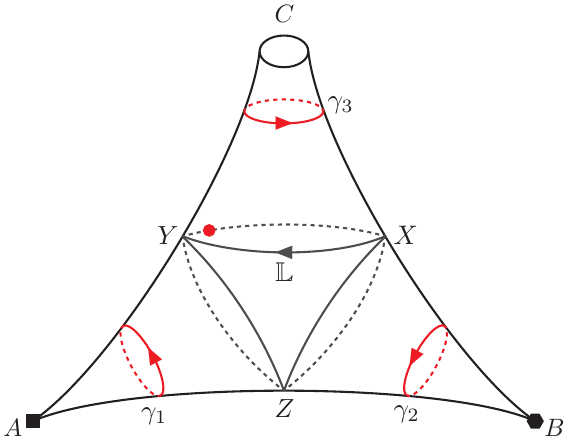}
\centering
\caption{\label{fig:fabc}  Orbifold sphere $\mathbb{P}^{1}_{a,b,c}$ in the Fermat case with one puncture $C$}
\end{figure}
\begin{prop}\label{grading computation}
An $\Omega$ and $H_1$-grading of $CF^\bullet(\bL, \bL)$ is given by the following table (see \ref{Omega-grading} and \ref{topological grading}).

\begin{center}
\begin{tabular}{ccccccccc}
&$1_\bL$ & $X$ & $Y$ & $Z$ & $\bar{X}$ & $\bar{Y}$ & $\bar{Z}$ & $[pt]=X\wedge Y \wedge Z$ \\
\hline
\hline
$\Omega$-grading & $0$ & $1$ & $1$ & $-1$ & $0$ & $0$ & $2$ & $1$ \\
$H_1$-grading & $0$ & $\gamma_1$ & $\gamma_2$ & $\gamma_3$ & $-\gamma_1$ & $-\gamma_2$ & $-\gamma_3$ & $0$\\
\hline
\hline
\end{tabular}
\end{center}

\end{prop}

\begin{proof} (See \cite{Se}, \cite{Sheridan11}) $H_1$-grading is still well-defined even though $\bL$ is not simply-connected. Because the class of $\bL$ is null-homotopic, any path inside $\bL$ starting from an immersed point to its another lift defines the unique homotopy class of loop in $H_1$. An $\Omega$-grading can be computed as a Maslov index of that class. Because $\Omega$ has  order $1$ poles at $0$ and $1$, Maslov index of $X$ and $Y$ is $1$. The index of $Z$ differs by $2$ because our $\Omega$ has no poles nor zeros at $\infty$.
\end{proof}
In the surface case, we will use the following Seidel's convention \cite{Se} of signs for $\AI$-operation $\{m_k\}$,
which are defined as counting convex polygons. Each polygon contributes $\pm 1$ to the coefficient of output and the sign is determined as follow. Let $L_{0}, \ldots, L_{k}$ be Lagrangian submanifolds which intersect transversally and $x_{i}$ be an intersection point of $L_{i-1}$ and $L_{i}$ for $0 \leq i \leq k$ (modulo $k+1$). Suppose that there is a polygon $P$ bounded by $L_{0}, \ldots, L_{k}$ which has $k$ inputs $x_{1}, \ldots, x_{k}$ and one output $x_{0}$. For each $1 \leq i \leq k-1$, if the orientation of $L_{i}$ does not match to the orientation of $\partial P$, there is a sign $(-1)^{|x_{i}|}$. If the orientation of $L_{k}$ does not match to the orientation of $\partial P$, it gives $(-1)^{(|x_{k}| + |x_{0}|)}$. If $L_{i}$ is equipped with a non-trivial spin structure, we represent it by a red dot on $\bL$. Each time when an edge of a polygon contains the red dot, the sign is multiplied by $(-1)$. Multiplying all these sign contributions determines the sign of polygon.
 
Using the reflection symmetry (take $\bL$ to be invariant under the involution), we can follow \cite{CHL} to prove that $\bL$ is weakly unobstructed
and compute its potential function  $W_\bL$.
    \begin{lemma}
    Equip $\mathbb L$ with a nontrivial spin structure (marked as red dot in Figure \ref{fig:fabc}). Then 
      \begin{enumerate}
      \item $\mathbb L$ is weakly unobstructed. 
      \item $\bb=xX+yY+zZ$ is a weak bounding cochain with potential $W_\bL$ where 
\[
  W_\bL = \left.
  \begin{cases}
   x^p+y^q+xyz,   & \text{for } \;\;F_{p,q}  \\
   y^q+xyz,  & \text{for } \;\;C_{p,q} \\
    xyz , & \text{for } \;\;L_{p,q}
  \end{cases}
   \right. 
\] 
      \end{enumerate}  
    \end{lemma}
    \begin{remark}
Note that $W_\bL$ is independent of $p$ for the chain, $p,q$ for the loop case.
This is not a contradiction because the quotient space $[M_W/G_W]$ is also independent of those indices as well.
\end{remark}
    \begin{remark}
    Since Milnor fibers are exact and $\bL$ is an exact Lagrangian (since it is homologically trivial), there
exist a change of coordinate such that $W$ does not have any area $T$-coefficient. Therefore, 
we will omit them in the paper.
\end{remark}

\begin{proof}
Weakly unobstructedness can be proved exactly the same way as Theorem 7.5 of \cite{CHL}.
To compute $W_\bL$, we fix a generic point and count all polygons whose corners are given by $X,Y,Z$'s.
Because of punctures, there are finitely many polygons contributing to $W_\bL$. Also, we are only counting 
smooth discs, which have lifts to the Milnor fiber. So we can count them on the cover.

In the Fermat case, recall that the Milnor fiber can be obtained by first taking $2p$-gon and taking $q$-copies of these $2p$-gon's by rotation around the $\Z/q$ fixed point. 
Then, we have one $p$-gon and one $q$-gon and $XYZ$-triangle passing through a generic point.
See Figure \ref{fig:po} (A).
Therefore, we have 
$$W_\bL = x^p + y^q +xyz$$

\begin{figure}[h]
\begin{subfigure}{0.43\textwidth}
\includegraphics[scale=0.5]{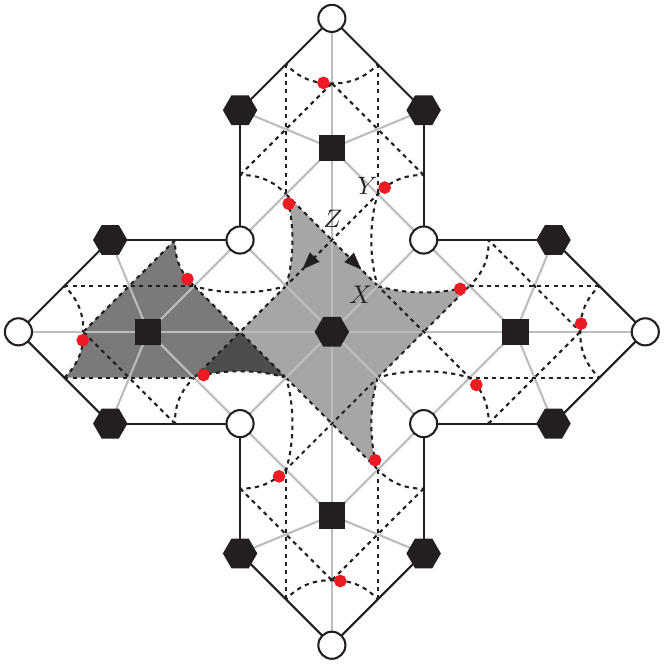}
\centering
\caption{Fermat case $F_{3,4}$}
\end{subfigure}
\begin{subfigure}{0.43\textwidth}
\includegraphics[scale=0.5]{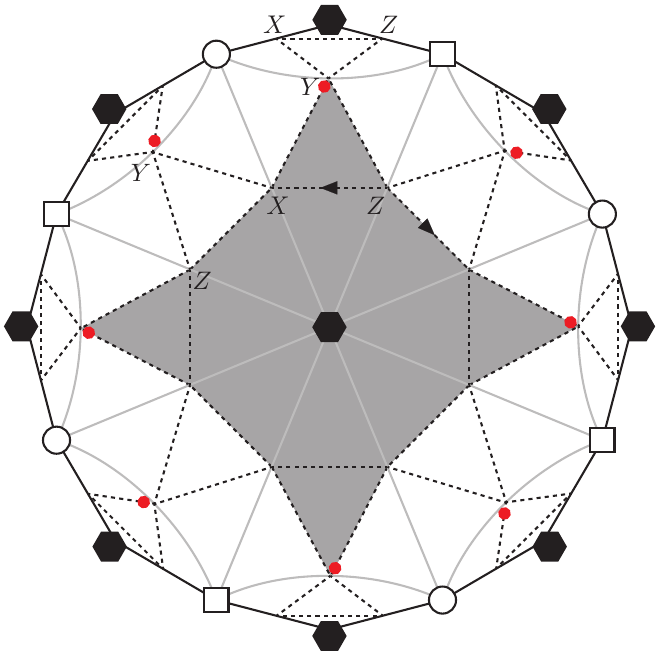}
\centering
\caption{Chain case $C_{2,4}$}
\end{subfigure}
\centering
\caption{Holomorphic polygons for the potentials}
\label{fig:po}
\end{figure}

In the chain case, its Milnor fiber is given by a $2pq$-gon with $A$-puncture at the center and
$B$-vertex and $C$-puncture as vertices of $2pq$-gon (with $\Z/pq$-action around $A$).
To see the discs, it is more convenient to describe the Milnor fiber with the orbifold point $B$ at the center. Note that there is only $\Z/q$-action around $B$. We can easily find a $q$-gon for the potential.
Hence, we obtain
$$W_\bL = y^q +xyz$$

In the loop case, all vertices are punctures and the only nontrivial disc is $XYZ$-triangle. Hence we have $$W_\bL= xyz$$
\end{proof}


\subsection{Localized mirror functor to Matrix factorization category}
For weakly unobstructed $\bL$,  localized mirror functor formalism \cite{CHL} provides
a canonical $\AI$-functor from Fukaya category of  $[M_W/G_W]$ to the matrix factorization
category of $W_\bL$. We use the version that appeared in \cite{CHLnc}.
\begin{defn}
Let $W^{\mathbb{L}}$ be the disc potential of $\mathbb{L}$. The localized mirror functor $\mathcal{F}^{\mathbb{L}} : \mathrm{Fuk}(X) \to \mathcal{MF}(W_{\mathbb{L}})$ is defined as follows.
\begin{itemize}
\item For given Lagrangian $L$, $\mathcal{F}^{\mathbb{L}}(L) := (CF(L,\mathbb{L}), -m_{1}^{0,\bb}) =: M_{L}$.
\item Higher component
$$\mathcal{F}^{\mathbb{L}}_{k} : CF(L_{1},L_{2}) \otimes \cdots \otimes CF(L_{k},L_{k+1}) \to \mathcal{MF}(M_{L_{1}},M_{L_{k+1}})$$
is given by
$$\mathcal{F}^{\mathbb{L}}_{k}(a_{1}, \ldots,a_{k}) := \sum_{l=0}^{\infty} m_{k+1+l}(a_{1},\ldots,a_{k},\bullet,\underbrace{\bb, \ldots, \bb}_{l}).$$
Here the input $\bullet$ is an element in $M_{L_{k+1}} = CF(L_{k+1},\mathbb{L})$.
\end{itemize}
\end{defn}

\begin{thm}[ \cite{CHL}]\label{thm:lmf}
$\mathcal{F}^{\mathbb{L}}$ defines an $A_{\infty}$-functor, which is cohomologically injective on $\bL$.
\end{thm}

\begin{example}
Consider a polynomial $x^{2} + xy^{4}$ and its Milnor fiber ($D_5^T$ singularity). Its fundamental domain and lifts of Seidel Lagrangian are as in Figure \ref{fig:exam} with $W_\bL = y^{4} + xyz$. We calculate $\mathcal{F}^{\mathbb{L}}(L) = (CF(L,\mathbb{L}), -m_{1}^{0,\bb})$ for the $L$ in the Figure. It is generated by two intersection points $o_{1}$ and $e_{1}$.

\begin{figure}[h]
\begin{subfigure}{0.33\textwidth}
\includegraphics[scale=0.4]{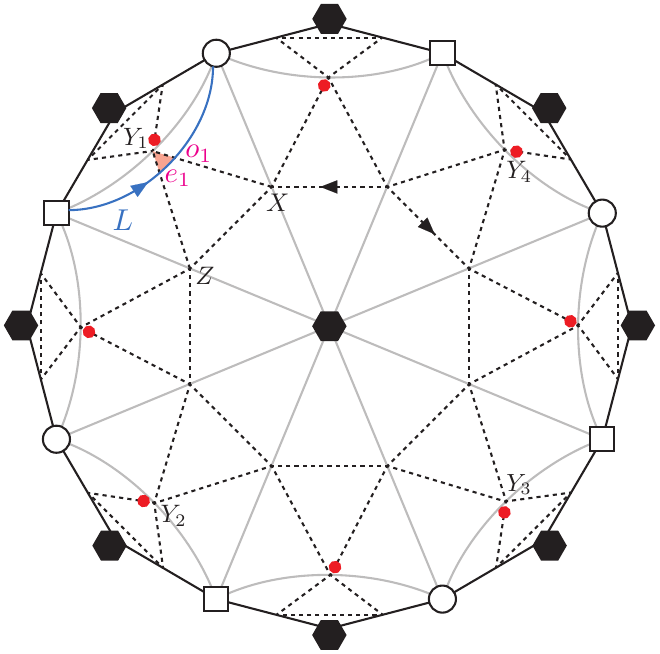}
\centering
\end{subfigure}
\begin{subfigure}{0.33\textwidth}
\includegraphics[scale=0.4]{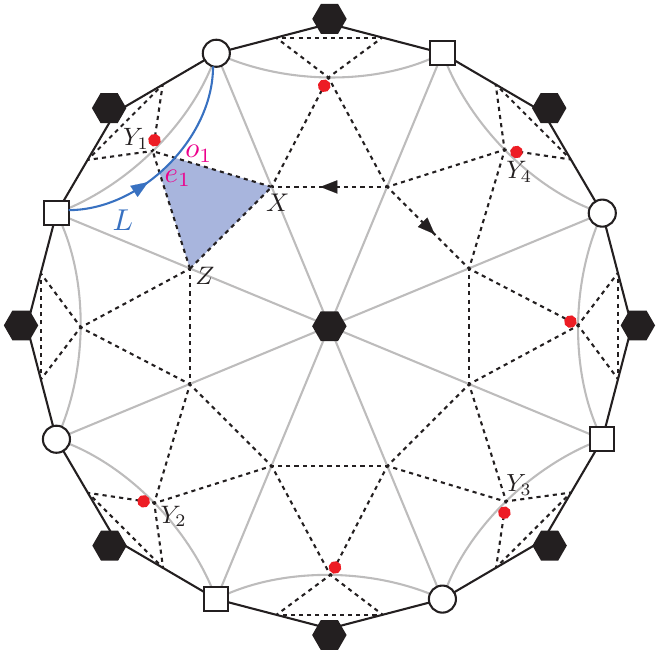}
\centering
\end{subfigure}
\begin{subfigure}{0.33\textwidth}
\includegraphics[scale=0.4]{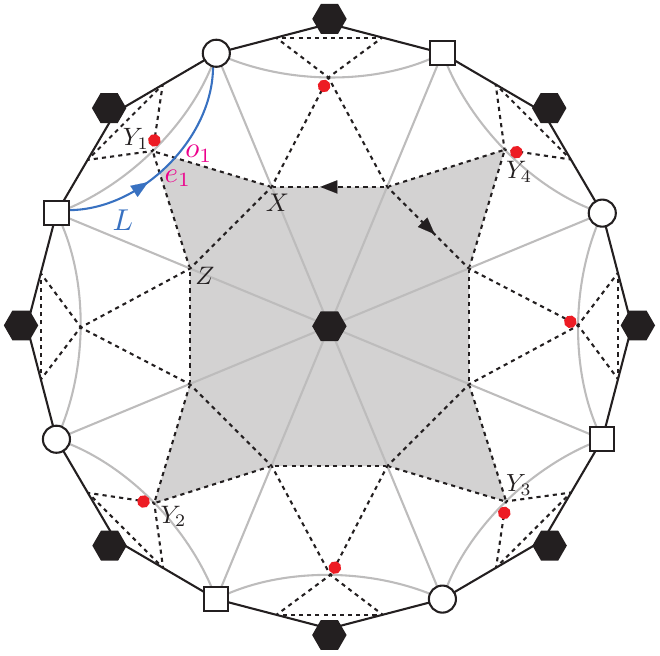}
\centering
\end{subfigure}
\centering
\caption{Polygons contributing to $m_{1}^{0,\bb}$}
\label{fig:exam}
\end{figure}

There is one triangle $o_{1} Y_{1} e_{1}$ which contributes to $m_{1}^{0,\bb}(o_{1})$ and it gives $y$. There are two polygon $e_{1} Z X o_{1}$ and $e_{1} Y_{2} Y_{3} Y_{4} o_{1}$. They contribute to $m_{1}^{0,\bb}(e_{1})$ and give $y^{3} + xz$. For all three polygons, since the orientations of discs and those of Lagrangians are matched, all signs of discs are $+$. Therefore, $(CF(L,\mathbb{L}), -m_{1}^{0,\bb})$ is a (matrix) factorization
$$-y \cdot (-y^{3} - xz)$$
of the potential $y^{4} + xyz$. Later, we will consider a restriction $z=x$, which turns $W_\bL$ into $D_5$ singularity $y^4+yx^2$.
\end{example}
Another important example is when $L$ is $\bL$ itself. $\mathcal{F}^\bL(\bL)$ can be computed following \cite{CHL}.
\begin{thm}(c.f. Theorem 7.6 \cite{CHL})
 \[M_{\bL} \simeq k^\textrm{stab}_{W_\bL}\]
where $k^\textrm{stab}_{W_\bL}$ is a matrix factorization of stabilized residue field at the origin. 
\end{thm}
\begin{remark}
The singularity of $W_\bL$ is not isolated, and in fact contains some of the coordinate axes. 
Hence $k^\textrm{stab}_{W_\bL}$ alone does not generate $\mathcal{MF}(W_\bL)$.
In the next section, we will find non-compact Lagrangians $L$ whose associated matrix factorization $M_L$
corresponds to the desired coordinate axes.
\end{remark}

The following lemma tells us that homotopic Lagrangians correspond to quasi-isomorphic matrix factorizations,
which will be useful in our computations.
\begin{lemma}
Suppose two non-compact Lagrangian $L_{1}$ and $L_{2}$ connect the same punctures and represent the same homotopy class with respect to their boundary in the cover of Milnor fiber. Then they are isomorphic.
\end{lemma}
\begin{proof}
Assume that there is no intersection between them. We have two even wrapped generators $\alpha \in \Hom(L_{2},L_{1})$ and $\beta \in \Hom(L_{1},L_{2})$. They satisfy $m_{1}(\alpha) = m_{1}(\beta) = 0$, $m_{2}(\beta, \alpha) = e_{1}$ and $m_{2}(\alpha, \beta) = e_{2}$. Hence $\alpha$ and $\beta$ give an isomorphism between them.

\begin{figure}[h]
\includegraphics[scale=0.7]{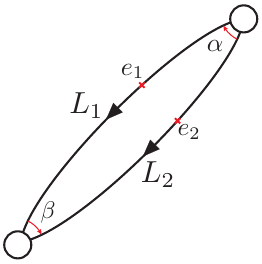}
\centering
\caption{Isomorphism between two Lagrangians}
\label{fig:Lagisom}
\end{figure}

In general case when they intersect arbitrarily, we choose an open set containing two Lagrangians and another Lagrangian $L^{\prime}$ in the complement. Then there is a sequence of isomorphisms $L_{1} \simeq L^{\prime} \simeq L_{2}$.
\end{proof}


\section{Homological mirror symmetry for Milnor fibers (without monodromy action)}\label{sec:HMSMilnor}
In this section, we consider homological mirror symmetry for Milnor fiber as a symplectic manifold.
We will find that $G_W$-equivariant mirror of $M_W$ is a Landau-Ginzburg model $W_\bL$.
By applying Theorem \ref{thm:lmf} to the wrapped Fukaya category of $[M_W/G_W]$, 
we obtain an $\AI$-functor which gives a derived equivalence.
\begin{thm}\label{thm:um}
We have an $\AI$-functor
$$\mathcal{F}^\bL:  \mathcal{WF}([M_W/G_W]) \to \mathcal{MF}(W_\bL)$$
where $W_\bL$ for Fermat $F_{p,q}$, chain $C_{p,q}$ and loop $L_{p,q}$ cases are given as
$$W_\bL = x^p+y^q+xyz, \;\;\; y^q+xyz, \;\;\; xyz$$
This functor is fully faithful and gives a derived equivalence between two categories.
\end{thm}
\begin{remark}
$W_\bL$ is related to the transposed potential $W^T$ as follows.
If we set $$g(x,y,z) =
  \begin{cases}
   z,   & \text{for } \;\;F_{p,q}  \\
   z-x^{p-1},  & \text{for } \;\;C_{p,q} \\
    z-x^{p-1} -y^{q-1} , & \text{for } \;\;L_{p,q}
  \end{cases}$$
then we have
$$W_\bL  = W^T(x,y) +xyg.$$
\end{remark}
As we will explain later, if we add monodromy information and take our newly defined $\AI$-category, the mirror  will be obtained by setting $g=0$, hence 
we obtain the matrix factorization of $W^T(x,y)$.

We prove the above theorem in the rest of the section. Although we treat each case separately, the underlying strategies are basically the same.


\subsection{Fermat cases}
Recall that  $G_W=\Z /p \times \Z/q $ is the maximal diagonal symmetry of $W=x^p+y^q$ and
the quotient space $[M_{F_{p,q}}/G_W]$ has a single puncture $C$ of orbifold order $\frac{pq}{\gcd(p,q)}$.
Then, for a preimage $\WT{C}$ in $M_{F_{p,q}}$, we connect $\WT{C}$ and $(1,0) \cdot \WT{C}$ by the shortest path  $\WT{L}_1$ as in the Figure \ref{fig:f4}, which we take as a non-compact Lagrangian.  
We denote by $L$ the embedded Lagrangian in $[M_{F_{p,q}}/G_W]$ given by its projection.
Denote by  $\WT{L}$ the set of lifts of $L$ in $M_{F_{p,q}}$, which is exactly $G_W \cdot \WT{L}_1$.

\begin{figure}[h]
\begin{subfigure}{0.43\textwidth}
\includegraphics[scale=0.62]{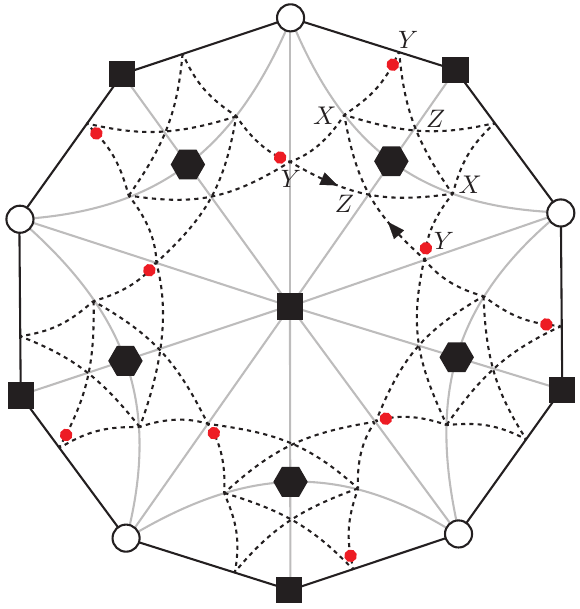}
\centering
\end{subfigure}
\begin{subfigure}{0.43\textwidth}
\includegraphics[scale=0.62]{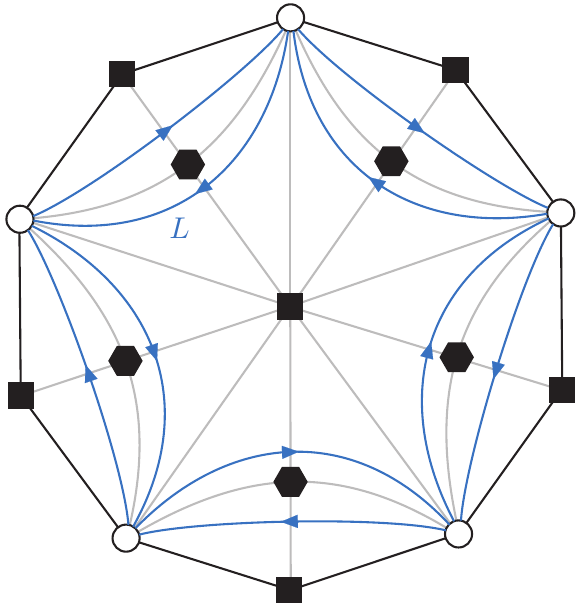}
\centering
\end{subfigure}
\centering
\caption{\label{fig:f4} Milnor fiber of $F_{4,2}$ and a choice of Lagrangian $L$ and its lifts}
\end{figure}

To prove the theorem in Fermat case, we show that $G_W$ copies of $L$ split-generate $\mathcal{WF}(M_{F_{p,q}})$. Also, we compute the mirror matrix factorization $\mathcal{F}^\bL(L)$,
and prove that the functor is fully faithful.
Finally, we show that $\mathcal{MF}(W_\bL)$ is split generated by $\mathcal{F}^\bL(L)$, and this proves the theorem \ref{thm:um}.

Recall that $CW^\bullet(L,L)$ is defined as  
$(CW^{\bullet}(\WT{L},\WT{L}))^{G_W}$.
 \begin{lemma}\label{lem:fermatww}
    Wrapped Floer complex $CW^\bullet(L,L)$ satisfies the following:
      \begin{enumerate}
        \item as a vector space,
          \[CW^\bullet(L,L) \simeq T(a,b)/\mathcal R_{F_{p,q}}\] 
          Here, $T(a,b)$ is a tensor algebra generated by two alphabets $a,b$.  The ideal $\mathcal R_{F_{p,q}}$ is defined as
          \[\mathcal R_{F_{p,q}} = <a\otimes a=\delta_{2,p}, b\otimes b=\delta_{2,q}>\]
          $\Z/2$-gradings of $a,b$ are odd and this induces $\Z/2$-grading on $T(a,b)/\mathcal R_{F_{p,q}}$.
        \item $m_1$ vanishes and $m_2$ coincides with the tensor product. 
        \item $m_k(a,\ldots,a)$ is zero for $1\leq k<p$ and it is equal to $1$ for $k=p$. Likewise, $m_k(b,\ldots, b)$ is zero for $1\leq k<q$ and equal to $1$ for $k=q$.
      \end{enumerate}
      
Its $\Omega$ and $H_1$-grading is given by the following table.
      \begin{center}
\begin{tabular}{ccccccccc}
&$1_L$ & $a$ & $b$ \\
\hline
\hline
$\Omega$-grading & $0$ & $1$ & $1$ \\
$H_1$-grading & $0$ & $-\frac{\gamma_3}{2}$ & $-\frac{\gamma_3}{2}$ \\
\hline
\hline
\end{tabular}
\end{center}
    \end{lemma}
    \begin{proof}
 We can choose $\WT{L}_1$ so that ${G_W}$-orbits of $\WT{L}_1$ are disjoint. Therefore  $CW^\bullet(L,L)$ consists only of Hamiltonian chords at infinity.  Among such chords, we choose the following two generators.
    \begin{itemize}
      \item $a$, the shortest chord $\in CW^\bullet(\WT{L}_1,(1,0)\cdot \WT{L}_1), \hskip 0.2cm (1,0)\in \Z/p \times \Z/q$.
      \item $b$, the shortest chord $\in CW^\bullet(\WT{L}_1,(0,1)\cdot \WT{L}_1), \hskip 0.2cm (0,1)\in \Z/p \times \Z/q$. 
    \end{itemize}
For example, take a rotation of $\WT{L}_1$ around the $\Z/p$ fixed point and there is the unique wrapped generator $a$ between these two branches.
  By abuse of notation, we also denote by $a,b$ the generators in $(CW^{\bullet}(\WT{L},\WT{L}))^{G_W}$   given by the sum of $G_W$-copies of the above generators.
     
    We can also concatenate them to create new Hamiltonian chords ($m_2$-products near the puncture), denoted by $\{a, b, ab, ba, aba, bab, \ldots \}$. One can check that $m_1$ vanishes. Note that if we consider $m_2$-operation near the puncture,  $m_2(a,a)$ and $m_2(b,b)$ vanish for $p, q \geq 3$ as they are not composable. If $p=2$ or $q=2$, we could have an $m_2$-product coming from a global holomorphic polygon which contributes $m_2(a,a)$ or $m_2(b,b)$ respectively. In general, there are two global $J$-holomorphic polygons with all of its corners are of word length $1$. They are $p$-gon and $q$-gon and come from lifts of upper/lower hemisphere of $(M_{F_{p,q}}/{G_W})\setminus L)$.  Their corners are Hamiltonian chords $a$ or $b$ at infinity. They cannot contribute to $m_{p-1}$ or $m_{q-1}$, only contribute to $m_p$ or $m_q$ respectively. The boundaries of these polygons are whole ${G_W}$-orbits of $L$ so they represent the unit element of $CW^{ \bullet}(L,L)$. The computation of grading is entirely analogous to Proposition \ref{grading computation}. We leave it as an exercise. 
    \end{proof}
    
 \begin{lemma}
 $\WT{L}$ split-generates the wrapped Fukaya category of $M_{F_{p,q}}$.
 \end{lemma}
    \begin{proof}
 We proceed as in the work of Heather Lee \cite{Lee}.
To avoid confusion, let us denote by $\tilde{a}$ the sum over $G_W$ orbit of $a$ in this proof. 
 From Abouzaid's generating criterion, it is enough to show that the following open-closed
 map hits the unit.
 \[\mathcal{OC}:CC_\bullet(CW^\bullet(\WT{L},\WT{L}))\to SH^\bullet(M_{F_{p,q}})).\]
We take the following Hochschild cycle 
\[ \frac{{\tilde{a}}^{\otimes p}}{p} - \frac{{\tilde{b}}^{\otimes q}}{q} \in CC_\bullet(CW^\bullet(\WT{L},\WT{L}))\] 
It is not hard to see that $\WT{L}$ provides a tessellation of $M_{F_{p,q}}$, which consists of $q$ distinct $p$-gons and $p$ distinct $q$-gons. 
We first check that it is ${G_W}$-equivariant Hochschild cycle. From Lemma \ref{lem:fermatww}, it is enough to check $m_p, m_q$ operations respectively.
 $$\partial_{Hoch}( {\tilde{a}}^{\otimes p}/p-\tilde{b}^{\otimes q}/q) = 
   (m_p(\tilde{a},\ldots,\tilde{a})-m_q(\tilde{b}, \ldots, \tilde{b}))= 1_{\WT{L}}-1_{\WT{L}} =0.$$

      On the other hand, the image of the open-closed map of this Hochschild cycle 
    is  a cocycle represented by the bounded area of $M_{F_{p,q}}$ covering each region with weight one.
      Note that the orientation of the boundary Lagrangians of $p$-gon and $q$-gon are opposite to each other, and thus  $p$-gons and $q$-gons in the image add up despite the negative sign in the expression  $-\tilde{b}^{\otimes q}/q$.
    \end{proof}

  Let us discuss the mirror matrix factorization. Using localized mirror functor, we can explicitly compute
  the mirror matrix factorization. Since $W_\bL$ has non-isolated singularity (singularity along $z$-axis), we need to be a little bit careful in the discussion.
 By counting appropriate polygons from the picture with a sign, we can prove the following
 lemma, whose proof is left as an exercise. Let $S = \C[x,y,z]$.
    \begin{lemma}
    The localized mirror functor $\cF^{\mathbb L}: \mathcal{WF}([M_{F_{p,q}}/G_{F_{p,q}}]) \to \mathcal{MF}(W_\bL)$ sends $L$ to a matrix factorization $M_L= (\begin{tikzcd}[ampersand replacement=\&]S^{\oplus 2}\arrow[r, "\delta_0", shift left =1] \& S^{\oplus 2}\arrow[l, "\delta_1", shift left =1]\end{tikzcd})$ where
      \begin{equation*}
      \delta_0 = 
        \begin{pmatrix}
        x&y\\
        -y^{q-1}&x^{p-1}+yz
        \end{pmatrix}, \;\;
        \delta_1=
        \begin{pmatrix}
        x^{p-1}+yz&-y\\
        y^{q-1}&x
        \end{pmatrix}
      \end{equation*}  
    \end{lemma}
    \begin{remark}\label{rmk:residuefield}
    If we set $z=0$, this matrix factorization become a compact generator of $\mathcal{MF}(W^T)$
    corresponding to skyscraper sheaf at the singular point.
    \end{remark}
 
 \begin{cor}
The matrix factorization $M_L$ is of Koszul type. Namely, we have an isomorphism
       \[M_L \cong \big( S[\theta_x, \theta_y] , \partial_K + \partial_K' \big).\]
Here, $\theta_x, \theta_y$ are odd degree generators (hence anti-commute) and
      \[       \partial_K = x\cdot\iota_{\theta_x}+y\cdot\iota_{\theta_y}, \;\;\; 
 \partial_K' := W_x \theta_x\wedge \cdot + W_y \theta_y\wedge \cdot \]
          where $W_x = (x^{p-1}+yz), \hskip 0.2cm W_y = y^{q-1}.$
 \end{cor}
 \begin{remark}
 The following Koszul complex has cohomology $\C[z]$ concentrated on the right end.
  \[K(x,y) := 0 \to S(\theta_1 \wedge \theta_2) \xrightarrow{\partial_K} S\theta_1 \oplus S\theta_2 \xrightarrow{\partial_K} S \to 0 \]
  
  \end{remark}

Following Dyckerhoff \cite{Dy}, we compute its endomorphism algebra $\mathrm{End}(M_L)$.
    \begin{lemma} $\mathrm{End}_{\MF}(M_L)$ is quasi-isomorphic to a DG algebra of polynomial differential operators 
      \begin{gather*}
      \Hom_{\MF}(M_L, M_L) \simeq \bigg( S[\partial_{\theta_x},\partial_{\theta_y}, (\theta_ x\wedge), (\theta_ y\wedge)], D \bigg)\\
      D(\partial_{\theta_x})=W_x, \;D(\partial_{\theta_y})=W_y, \;D(\theta_x\wedge)= x, \;D(\theta_y\wedge) = y.
      \end{gather*}
    Its cohomology is 
      \begin{gather*}
      H^\bullet(\Hom_{\MF}(M_L, M_L)) \simeq \C[z][\Gamma_x, \Gamma_y]\\
      \Gamma_x = [\partial_{\theta_x}-x^{p-2}(\theta_x\wedge) -z(\theta_y\wedge)], \;\Gamma_y = [\partial_{\theta_y} - y^{q-2}(\theta_y\wedge)].
      \end{gather*}
    \end{lemma}  
    \begin{proof}
    The first part of the lemma is obvious because morphisms of a matrix factorization of Koszul type are those of exterior algebras. It is easy to check that the differential satisfies given equations. For example, 
    \begin{equation*}
    D(\partial_{\theta_x}) = [\partial_K+\partial'_K, \partial_{\theta_x}] = [\partial'_K, \partial_{\theta_x}] = [(W_x\theta_x\wedge), \iota_{\theta_x}]= W_x.
    \end{equation*}
    To each differential operator, we can assign an order of its symbols. It provides a decreasing filtration $\{F^i\}$ on the complex. The first page of the spectral sequence associated to the filtration is a dual Koszul complex associated to a regular sequence $(x,y)$;
      \[E_1 = H^\bullet \bigg(K^\vee(x,y)\otimes_\C \C[\partial_{\theta_x}, \partial_{\theta_y}], \partial_K^\vee\otimes 1 \bigg)\simeq \C[z][\partial_{\theta_x}, \partial_{\theta_y}]\]
    In particular we know that the cohomology algebra is a $\C[z]$- modules of rank less or equal to $4$. On the other hand, the cycles generated by $\Gamma_x, \Gamma_y$ in the lemma have already provided four $\C[z]$-linear independent element. Therefore the spectral sequence degenerates at $E_1$ page. This finishes the proof. 
    \end{proof}
    We can show that our mirror functor is fully-faithful. 
    \begin{lemma}
      The first-order part of the mirror functor is 
      \begin{gather*}
      \cF^{\mathbb L}_1: CW^\bullet(L,L) \to \Hom_{\mathcal{MF}}(M_L, M_L)\\
      a \to \Gamma_x, \;b \to \Gamma_y.
      \end{gather*}
      It is a quasi-isomorphism. Therefore $\cF^{\mathbb L}$ embeds ${\mathcal{WF}}(M_{F_{p,q}})$ as a full subcategory of $\mathcal{MF}(W_\bL)$.
    \end{lemma}
    \begin{proof}
 From the Figure \ref{fig:f4}, we see that $\cF^\mathbb L_1$ sends $a$ to $\Gamma_x$ and $b$ to $\Gamma_y$. Moreover,   
        \[[\Gamma_x, \Gamma_y] = [-z(\theta_y\wedge), \partial_{\theta_y}] = z.\]
      Therefore $ab+ba$ hits $z$ and $\cF^\mathbb L_1$ is surjective.
      
      note that $CW^\bullet(L,L)$ and $H^\bullet(\Hom_{\mathcal{MF}}(M_L, M_L))$ are filtered by 
        \[F^k := (ab+ba)^k \cdot CW^\bullet(L,L), \hskip 0.2cm G^l := z^l \cdot H^\bullet(\Hom_{\mathcal{MF}}(M_L, M_L))\] 
      It is easy to check that $\cF^\mathbb L_1$ is a filtered map with respect to $F^\bullet$ and $G^\bullet$. 
      
      The graded piece $F^0/F^1$ is a $4$ dimensional vector space spanned by four words $<1,a,b,ab>$. This is because
        \[aba = (ab+ba)\cdot a - \delta_{2,p} b, \hskip 0.2cm bab= (ab+ba)\cdot b-\delta_{2,q}a.\]
      An element $ab+ba$ is in the center of the algebra. Therefore
      \[F^k/F^{k+1} \simeq (ab+ba)^k\cdot F^0/F^1 = (ab+ba)^k\cdot<1,a,b,ab>.\] 
      By a similar reason, we have 
      \[G^k/G^{k+1}\simeq z^k \cdot G^0/G^1 = z^k \cdot <1, \Gamma_x, \Gamma_y, (\Gamma_x\circ\Gamma_y)>\] 
      The induced morphism of associated graded $Gr \cF^\mathbb L_1$ is an isomorphism of vector spaces at every level. By the comparison theorem, so is $\cF^\mathbb L_1$. 
    \end{proof}
    \begin{cor}
      $\cF^{\mathbb L}:\WF([M_{F_{p,q}}/G_{F_{p,q}}])\to \mathcal{MF}(W^T+xyz)$ is a quasi-equivalence. 
    \end{cor}
    \begin{proof}
      It is enough to show that $M_L$ and $M_\mathbb L$ generate $\MF(W^T+xyz)$. Orlov's equivalence 
      \[\MF(W^T+xyz) \simeq D_{sg}(W^T+xyz)\]
      \[\left(\begin{tikzcd}M^1\arrow[r, "\phi", shift left =1]& M^0 \arrow[l, "\psi", shift left =1]\end{tikzcd}\right) \mapsto \mathrm{coker}(\psi) \]
      sends $M_\mathbb L$ to a skyscraper sheaf $\mathcal O_o$ at the origin and $M_L$ to a structure sheaf $\mathcal O_z$ of $z$-axis. These are two irreducible components of a critical locus of $W^T+xyz$. Therefore they generates $\MF(W^T+xyz)$ (see \cite{Stev13}).
    \end{proof}

\subsection{Chain cases}
 The polynomial $W =C_{p,q}= x^p +xy^q$ has  maximal symmetry group $G_W=\Z/pq$.
We proceed as in the Fermat case.
Denote by $\xi$ the following generator of $G_W$:
  \[x\to e^\frac{2\pi i}{p}\cdot x, \hskip 0.2cm y \to e^\frac{-2\pi i}{pq} \cdot y.\]
Recall that  the quotient space $[M_{C_{p,q}}/G_W]$ has one orbifold point of order $q$ and two punctures of order $pq$ and $\frac{pq}{\gcd(p-1,q)}$, respectively. 
Let us call them as $B_1, B_2$ respectively. The orbifold action near $B_1$ is generated by $\xi$ while the action near $B_2$ is generated by $\xi^{p-1}$ by Proposition \ref{prop:homo}. 

\begin{figure}[h]
\includegraphics[scale=0.50]{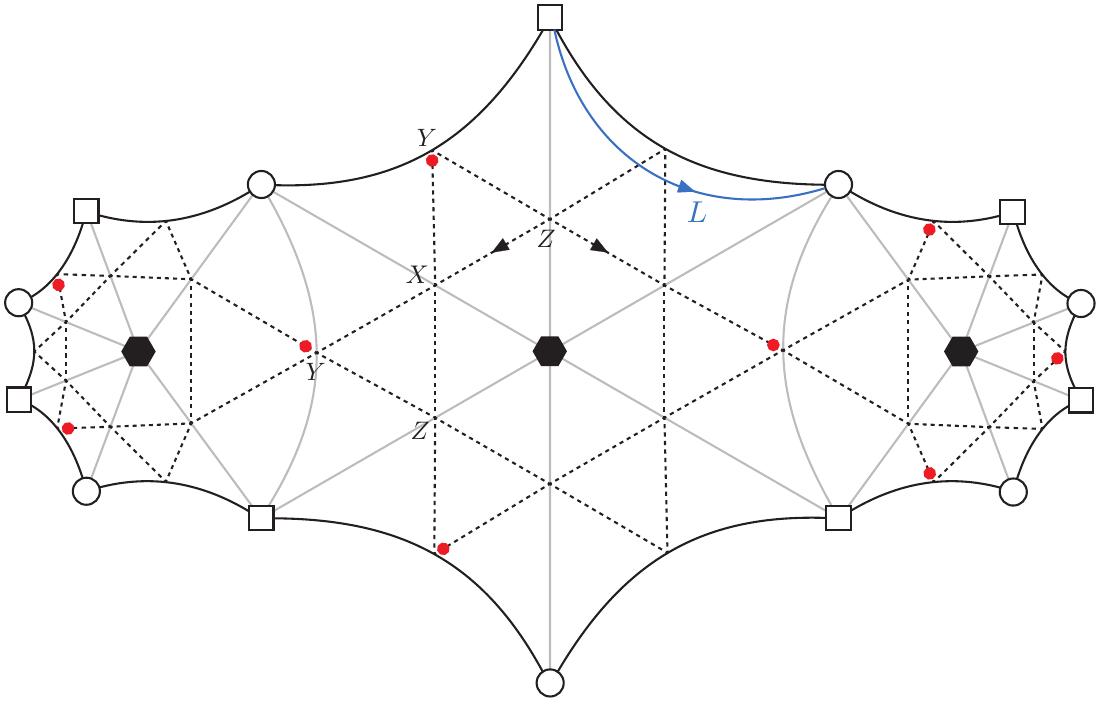}
\centering
\caption{\label{fig:C33} Milnor fiber of $E_7=C_{3,3}$ and a choice of Lagrangian $L$}
\end{figure}

 We take a Lagrangian $L$ connecting $B_1$ and $B_2$ in $\mathbb{P}^1_{pq, q, \frac{pq}{\gcd(p-1,q)}}$ (we may take the part of the equator between $B_1$ and $B_2$), and denote by $\WT{L}$ the sum of all lifts of $L$ in the Milnor fiber.  
    \begin{lemma}
    The wrapped Floer complex $CW^{\bullet}(L,L)$ satisfies the following:
      \begin{enumerate}
        \item as a vector space,
          \[CW^{\bullet}(L,L) \simeq \C[a,b]/(ab=0)\] 
          Here, $a,b$ are even variables. 
        \item $m_1$ vanishes and $m_2$ coincides with a polynomial multiplication. 
        \item $m_{k}(a,b,a,b\ldots)=0$ for $1\leq k \leq 2q-1$ and $m_{2q}(a,b, \ldots, a, b)=1$. Likewise, $m_{k}(b,a,b,a\ldots)=0$ for $1\leq k \leq 2q-1$ and $m_{2q}(b,a, \ldots, b,a)=1$.
      \end{enumerate}
      Its $\Omega$ and $H_1$-grading is given by the following table.
      \begin{center}
      \begin{tabular}{ccccccccc}
      &$1_L$ & $a$ & $b$ \\
      \hline
      \hline
      $\Omega$-grading & $0$ & $0$ & $2$ \\
      $H_1$-grading & $0$ & $-\gamma_1$ & $-\gamma_3$ \\
      \hline
      \hline
      \end{tabular}
      \end{center}
    \end{lemma}
    \begin{proof}
   Branches of $\WT{L}$ do not intersect with each other in the interior. Therefore $CW^{\bullet}(L,L)$ consists of Hamiltonian chords at infinity near $B_1$ or $B_2$. Among them, we choose two generators between the nearest orbits. Namely, choose one lift $\WT{L}_1$ and take the wrapped generator
    \begin{itemize}
      \item $a$, the shortest chord $\in CW^\bullet(\WT{L}_1, \xi^{-1} \cdot \WT{L}_1)$ near $B_1$
      \item $b$, the shortest chord $\in CW^\bullet(\WT{L}_1, \xi^{1-p} \cdot \WT{L}_1)$ near $B_2$.
    \end{itemize}
Here, $a$ (resp. $b$) is nothing but the chord between $\WT{L}_1$ and its clockwise rotation at $B_1$ (resp. $B_2$). 
Namely, recall that $\xi, \xi^{p-1}$ correspond to $\gamma_1, \gamma_3$ of the orbifold fundamental group in the Proposition \ref{prop:homo}.  And $\gamma_1^{-1}, \gamma_3^{-1}$ are the minimal clockwise
rotations in the uniformizing neighborhood of orbifold points. Therefore $\xi \cdot \WT{L}_1$ is obtained by clockwise rotation of $\WT{L}_1$(centered at $B_1$) sending $B_2$-vertex to the nearest $B_2$-vertex. The same holds for $ \xi^{p-1} \cdot \WT{L}_1$ switching the role of $B_1$ and $B_2$.

    We can also concatenate them to create new Hamiltonian chords, namely $a^2, a^3, \ldots , b^2, b^3, \ldots$. We cannot concatenate different words as their heads and tails are different from each other.
     The rest of the argument is similar to the Fermat case.  $m_1$ vanishes because there are no $J$-holomorphic strip between them. Concatenating two chords corresponds to $m_2$ operation concentrated near the punctures. The first global $J$-holomorphic polygon contributes to a non-trivial $A_\infty$ operation is a $2q$-gon. It is a lift of an orbifold bigon $(M_{C_{p,q}}/G_W) \setminus L$. Its corners consists of $q$ many $a$ and $b$ alternating each other.  
    \end{proof}
    \begin{lemma}
   $\WT{L}$ generates the wrapped Fukaya category of $M_{C_{p,q}}$.
    \end{lemma}
    \begin{proof}
 We proceed as in the Fermat case. Milnor fiber $M_{C_{p,q}}$ is tessellated by $p$ copies of $2q$-gons that are considered in the previous lemma. In Figure \ref{fig:C33}, this is given by 3 copies of hexagons.
 To show that open-closed map hits the unit, we take the following Hochschild cycle;
    \[\frac{1}{q}(\tilde{a}\otimes \tilde{b})^{\otimes q}\in CC_\bullet(CW^\bullet (\WT{L},\WT{L}))\]
    It is indeed a cycle because
    \begin{equation*}
    \partial_{Hoch}\left(\frac{1}{q}(\tilde{a}\otimes \tilde{b})^{\otimes q}\right) =    m_{2q}(\tilde{a},\tilde{b},\ldots, \tilde{a},\tilde{b})-  m_{2q}(\tilde{b},\tilde{a},\ldots, \tilde{b},\tilde{a}) =  1_{\WT{L}} - 1_{\WT{L}} =0.
    \end{equation*}
     On the other hand, the open-closed image of this Hochschild cycle is
      a cocycle represented by the bounded area of $M_{C_{p,q}}$ covering each region with weight one.
     \end{proof}
 
If we solve the weak Maurer-Cartan equation for $\bL$, we get the potential $W_\bL = y^q+xyz$,
which can be also written as $W^T + xyg$ with $g(x,y,z)=z - x^{p-1}$.

   \begin{lemma}
    The localized mirror functor $\cF^{\mathbb L}: \mathcal{WF}([M_{C_{p,q}}/G_{C_{p,q}}]) \to \mathcal{MF}(W_\bL)$ sends $L$ to the matrix factorization $M_L= (\begin{tikzcd}[ampersand replacement=\&]S\arrow[r, "\delta_0", shift left =1] \& S\arrow[l, "\delta_1", shift left =1]\end{tikzcd})$ where
      \begin{equation*}
      \delta_0 =y, \;\;\delta_1 =y^{q-1}+xz.
      \end{equation*}  
    \end{lemma}
    \begin{proof} 
    This follows from the Figure \ref{fig:C33}.
    \end{proof}

    The matrix factorization $M_L$ is again of Koszul type. 
    One can check directly that 
      \[M_L = \big( S[\theta_y] , (y\cdot i_{\theta_y}+W_y \cdot \theta_y \wedge)\big), \hskip 0.3cm W_y = y^{q-1}+xz\]
    Using the same technique,   
    \begin{lemma} The self-hom space of $M_L$ is quasi-isomorphic to a DG algebra of polynomial differential operators 
      \begin{gather*}
      \Hom_{\MF}(M_L, M_L) \simeq \bigg( S[\partial_{\theta_y}, (\theta_ y \wedge)], D \bigg)\\
      D(\partial_{\theta_y})=W_y, \;D(\theta_y\wedge) = y.
      \end{gather*}
    Its cohomology is concentrated to an even degree and isomorphic to 
      \[H^\bullet(\Hom_{\MF}(M_L, M_L)) \simeq \C[x,z]/(xz=0)\]
    \end{lemma}  
    \begin{proof}
    The first part of the lemma is the same as Fermat case. The cohomology computation can be done in a similar way, but we found that it is much easier to do it by hand. This complex is isomorphic to the 2-periodic complex 
    \begin{equation*}
    \begin{tikzcd}[ampersand replacement=\&]S^{\oplus 2, even}\arrow[r, "D_0", shift left =1] \& S^{\oplus 2, odd}\arrow[l, "D_1", shift left =1]\end{tikzcd} \mbox{ where  } D_0 = \begin{pmatrix} y & -y \\ W_y&-W_y \end{pmatrix}, \;\; D_1 = \begin{pmatrix} y&-W_y \\ y&-W_y \end{pmatrix}.
    \end{equation*}
    Therefore, we have 
    \begin{gather*}
    \mathrm{ker}(D_0)/\mathrm{im} (D_1) = \Bigg\{\begin{pmatrix}a\\b\end{pmatrix}\in S^{\oplus 2, even} \Big| a=b \Bigg\}\Bigg/ S\cdot \begin{pmatrix}y\\y\end{pmatrix} \oplus S \cdot \begin{pmatrix} W_y \\ W_y \end{pmatrix} \simeq \C[x,y,z]/(y=W_y=0) \simeq \C[x,z]/(xz=0),\\
    \mathrm{ker}(D_1)/\mathrm{im}(D_0) = \Bigg\{\begin{pmatrix}a\\b\end{pmatrix}\in S^{\oplus 2, odd} \Big| W_y\cdot a+y\cdot b =0\Bigg\}\Bigg/ S\cdot \begin{pmatrix} y\\-W_y\end{pmatrix} \simeq 0.
    \end{gather*}
    The last equality holds because $(y,W_y)$ is a regular sequence of $S$. 
    \end{proof}
    Now we can show that our mirror functor is an equivalence.
    \begin{lemma}
      The first-order part of the mirror functor is given by 
      \begin{gather*}
      \cF^{\mathbb L}_1: CW^\bullet(L,L) \to \Hom_{\mathcal{MF}}(M_L, M_L)\\
      a \to x, \;b \to z.
      \end{gather*}
      It is a quasi-isomorphism. Moreover $\cF^{\mathbb L}: \mathcal{WF} ([M_{C_{p,q}}/G_{C_{p,q}}]) \to \mathcal{MF}(y^q+xyz)$ is a quasi-equivalence.  
    \end{lemma}
    \begin{proof}
    Similar to Fermat case.
    \end{proof}

\subsection{Loop cases}
The loop type polynomial $W=x_1^px_2+x_1x_2^q$ has  $G_W =\Z/pq-1$ as the maximal diagonal symmetry group. One notable difference of a loop type from the others is that the action of $G_W$ is free. The quotient $M_{L_{p,q}}/G_W$ is an honest three-punctured sphere. Its wrapped Fukaya category and its homological mirror symmetry was proved in \cite{AAEKO}. The result in this section can be essentially found therein, except that we use a localized mirror functor to define the explicit correspondences.
  
   Let us introduce more notation. For loop type, we use variables $x_i \hskip 0.1cm (i=1,2,3)$ instead of $x,y,z$.  Let $\xi$ denote the following generators of this group. 
    \[x_1\to e^\frac{2q\pi i}{pq-1}\cdot x_1, \hskip 0.2cm x_2\to e^\frac{-2\pi i}{pq-1}\cdot x_2\]
    Also recall three punctures are of order $pq-1, pq-1$ and $\frac{pq-1}{\gcd(p-1,q-1)}$. Let's denote them by $B_1, B_2$, and $B_3$ respectively. A cyclic orbifold action is generated by $\xi$ near $B_1$, by $\xi^{-p}$ near $B_2$ and by $\xi^{p-1}$ near $B_3$ by Proposition \ref{prop:homo}.
 As there are three punctures, we choose three shortest Lagrangians $L_i$ from $B_{i+1}$ to $B_{i+2}$ for $i=1,2,3 \mod 3$ which are part of the equator sphere passing through 3 punctures.
The following can be checked from  \cite{AAEKO}.
    \begin{lemma}
    The  wrapped Floer complexes $CW^{\bullet}(L_i,L_j)$ satisfies the following:
      \begin{enumerate}
        \item as a vector space,
          \[CW^{\bullet}(L_i,L_j) \simeq \left \{
          \begin{array}{ll}
          \C[a_{i+1}, b_{i+2}]/(a_{i+1}b_{i+2}=0) & i=j \\
          \C<a_i^n \cdot c_{i,j}\cdot b_j^m>, \hskip 0.2cm n,m\in \mathbb N & i\neq j\\
          \end{array}\right.\] 
          Here, $a_i,b_i$ are even an and $c_{i,j}$ are odd.
        \item $m_1$ vanishes and $m_2$ coincides with a polynomial multiplication and an obvious bimodule structure. 
        \item $m_3(c_{12}, c_{23}, c_{31})=1$
      \end{enumerate}
       Its $\Omega$ and $H_1$-grading is given by the following table.
      \begin{center}
      \begin{tabular}{ccccccccc}
      &$1_{L_i}$ & $a_i$ & $b_{i+1}$ &$c_{12}$&$c_{23}$&$c_{31}$&\\
      \hline
      \hline
      $\Omega$-grading & $0$ & $2\cdot \delta_{i,3}$ & $2\cdot \delta_{i,3}$ & $-1$ & $1$ & $1$ \\
      $H_1$-grading & $0$ & $-\gamma_i$ & $-\gamma_i$ & & not defined & \\
      \hline
      \hline
      \end{tabular}
      \end{center}
    \end{lemma}
Consider the direct sum of lifts $\WT{L}_i$ of $L_i$ in $M_{L_{p,q}}$. Then the following
is well-known.
    \begin{lemma}
    $\{\WT{L}_i \}_{i=1,2,3}$ split-generates the wrapped Fukaya category of $M_{L_{p,q}}$.
    \end{lemma}
    
Next, we move on to the mirror computation. For the Seidel Lagrangian $\bL$ in the quotient
$[M_{L_{p,q}}/G_W]$, the potential function can be computed as 
$$W_\bL = xyz.$$
Recall that this is related to $W^T$ as 
 $xyz = x^py +xy^q  + xy(z - x^{p-1}-y^{q-1}) = W^T+ xy\cdot g(x,y,z)$.
 
From the picture, it is easy to check the following.
    \begin{lemma}
    The localized mirror functor $\cF^{\mathbb L}: \mathcal{WF}([M_{L_{p,q}}/G_{L_{p,q}}]) \to \mathcal{MF}(x_1x_2x_3)$ sends $L_i$ to the matrix factorization $M_{L_i} = (\begin{tikzcd}[ampersand replacement=\&]S\arrow[r, "\delta_{i,0}", shift left =1] \& S\arrow[l, "\delta_{i,1}", shift left =1]\end{tikzcd})$ where
      \begin{equation*}
      \delta_0 =x_i, \;\; \delta_1 =\frac{x_1 x_2 x_3}{x_i}
      \end{equation*}
    \end{lemma}
As before, we can also write them as
      \[M_{L_i} = \big( S[\theta_{x_i}] , (x_i\cdot i_{\theta_{x_i}}+W_{x_i} \cdot \theta_{x_i} \wedge)\big), \hskip 0.3cm W_{x_i} = \frac{x_1x_2x_3}{x_i}.\] 
     
    For later purposes, we calculate $\hom$ complex by hand 
   (in \cite{AAEKO} it was proved using Orlov's equivalence ).
       \begin{lemma} The self-hom space of $M_{L_i}$ is quasi-isomorphic to a DG algebra of polynomial differential operators 
      \begin{gather*}
      \Hom_{\MF}(M_{L_i}, M_{L_i}) \simeq \bigg( S[\partial_{\theta_{x_i}}, (\theta_ {x_i} \wedge)], D \bigg)\\
      D(\partial_{\theta_x})=W_{x_i}, \;D(\theta_x\wedge) = x_i.
      \end{gather*}
    The cohomologies of Floer complexes are given as follows; 
      \[H^\bullet(\Hom_{\MF}(M_{L_i}, M_{L_i})) \simeq \left \{ 
        \begin{array}{ll}
          \C[x_1, x_2, x_3]/(x_i=W_{x_i}=0) & i=j\\
          \C[x_1, x_2, x_3]\cdot(\frac{x_1x_2x_3}{x_ix_j})/(x_i=x_j=0)& i\neq j
        \end{array} \right . 
      \]               
    \end{lemma}  
    \begin{proof}
    A computation of the self-Floer complex is almost identical to that of chain type. A complex of morphism $\Hom_\MF(M_{L_i}, M_{L_j})$ is isomorphic to 
    \begin{equation*}
    \begin{tikzcd}[ampersand replacement=\&]S^{\oplus 2, even}\arrow[r, "D_0", shift left =1] \& S^{\oplus 2, odd}\arrow[l, "D_1", shift left =1]\end{tikzcd} \mbox{ where  }
    D_0 = \begin{pmatrix} x _i& -x_j \\ W_{x_j}&-W_{x_i} \end{pmatrix},\; D_1 = \begin{pmatrix} W_{x_i} & -x_j \\ W_{x_j}& -x_i \end{pmatrix}.
    \end{equation*}
    Therefore, we have 
    \begin{align*}
    \mathrm{ker}(D_0)/\mathrm{im}(D_1) & = \left\{ 
      \begin{array}{ll} 
        \Bigg\{\begin{pmatrix}a\\b\end{pmatrix}\in S^{2, even} \Big| a=b \Bigg\}\Bigg/ S\cdot \begin{pmatrix}x_i \\ x_i\end{pmatrix} + S \cdot \begin{pmatrix} W_{x_i} \\ W_{x_i} \end{pmatrix} & (i=j)\\
        \Bigg\{\begin{pmatrix}a\\b\end{pmatrix}\in S^{2, even} \Big| x_i\cdot a=x_j\cdot b \Bigg\}\Bigg/ S\cdot \begin{pmatrix}x_j \\ x_i\end{pmatrix} + S \cdot \begin{pmatrix} W_{x_i} \\ W_{x_j} \end{pmatrix} & (i\neq j)
      \end{array} \right. \\
    &\simeq \left\{ 
      \begin{array}{ll}
        \C[x_1, x_2, x_3]/(x_i=W_{x_i}=0) & (i=j)\\
        0 & (i\neq j)
      \end{array} \right. \\  
    \mathrm{ker}(D_1)/\mathrm{im}(D_0) & = \left\{ 
      \begin{array}{ll} 
        \Bigg\{\begin{pmatrix}c\\d\end{pmatrix}\in S^{2, odd} \Big| W_{x_i}\cdot c=x_i\cdot d \Bigg\}\Bigg/ S\cdot \begin{pmatrix}x_i \\ W_{x_i}\end{pmatrix} & (i=j)\\
        \Bigg\{\begin{pmatrix}c\\d\end{pmatrix}\in S^{2, odd} \Big| W_{x_i} \cdot c=x_j\cdot d \Bigg\}\Bigg/ S\cdot \begin{pmatrix}x_i \\ W_{x_j} \end{pmatrix} + S \cdot \begin{pmatrix} x_j \\ W_{x_i} \end{pmatrix} & (i\neq j)
      \end{array} \right. \\
    &\simeq \left\{ 
      \begin{array}{ll}
        0  & (i=j)\\
        \C[x_1, x_2, x_3]\cdot (\frac{x_1 x_2 x_3}{x_i x_j})/(x_i=x_j=0) & (i\neq j)
      \end{array} \right.
    \end{align*}
    \end{proof}
    Now we can show that our mirror functor is an equivalence as before.
    \begin{lemma}
      The first-order part of the mirror functor is 
      \begin{gather*}
      \cF^{\mathbb L}_1: CW^{ \bullet}(L_i,L_j) \to \Hom_{\mathcal{MF}}(M_{L_i}, M_{L_j})\\
      a_{i} \to x_i,\;
      b_{i} \to x_i,\;
      c_{i,j} \to \frac{x_1x_2x_3}{x_ix_j}.
      \end{gather*}
      It is a quasi-isomorphism. Moreover, $\cF^{\mathbb L}: \mathcal{WF}([M_{C_{p,q}}/G_{L_{p,q}}]) \to \mathcal{MF}(x_1x_2x_3)$ is a quasi-equivalence.
    \end{lemma}


\section{Berglund--H\"ubsch mirror construction via Floer theory}\label{sec:KS}

In this section, we will prove Theorem \ref{thm:intro2}. First, we briefly recall a construction of monodromy Reeb orbit $\Gamma_{W}$ for $W$. Next, we compute $\Gamma_W$ explicitly and show that the localized mirror functor $\cF^\bL$ intertwines $\Gamma_W$ and $g(x,y,z)$.
   
\subsection{Construction of monodromy Reeb orbit $\Gamma_{W}$} \label{subsec:Gamma}

More details about construction can be found in \cite{CCJ20}. Let $W \in \mathbb{C}[x_{1},\dots,x_{n}]$ be a weighted homogeneous polynomial of weight $(w_{1},\dots,w_{n};h)$. Its monodromy transformation $\Phi_{W} : M_{W} \to M_{W}, x_{i} \mapsto e^{\frac{2\pi i w_{i}}{h}}x_{i}$ can be described as a time-1 Hamiltonian flow of a quadratic Hamiltonian $H = \frac{1}{2} \sum_{i=1}^{n} |x_{i}|^{2}$ (with respect to some rescaled symplectic form by weights). Its flow is given by
$$\phi_{H}^{s}(x_{1},\dots,x_{n}) = \left( e^{\frac{2\pi i w_{1}s}{h}}x_{1}, \dots, e^{\frac{2\pi i w_{n}s}{h}}x_{n} \right)$$
and we call it a monodromy flow.

Let us consider a link of $W$, $L_{W,\delta} \coloneqq W^{-1}(0) \cap S^{2n-1}_{\delta}$ for small $\delta >0$. Then, the monodromy flow coincides with a Reeb flow $Fl^{\mathcal{R}}_{t}$ on $L_{W,\delta}$. There exists a Reeb chord $\gamma : [0,1] \to L_{W,\delta}$ starting at any point $x = \gamma(0)$ and ending at its monodromy image $\Phi_{W}(x) = \gamma(1)$. In the quotient of link $\left[ L_{W,\delta} / G_{W} \right]$, those Reeb chords become Reeb orbits. Note that a set of time-1 Reeb orbits in $\left[ L_{W,\delta} / G_{W} \right]$ is identified $\left[ L_{W,\delta} / G_{W} \right]$ itself.

In \cite{Sei00}, it was shown that one can identify symplectically the link $L_{W,\delta}$ with $W^{-1}(0) \cap S^{2n-1}_{\delta}$. Hence, one can also regard $\left[ L_{W,\delta} / G_{W} \right]$ as the contact type boundary of $\left[ M_{W} / G_{W} \right]$. Roughly, $\Gamma_{W}$ is the Reeb orbit corresponding to the fundamental class of $\left[ L_{W,\delta} / G_{W} \right]$.

More precisely, we need an $G_{W}$-equivariant Morse function $h$ on $L_{W,\delta}$ to perturbations of Morse-Bott components. A construction of such $h$ (with nice properties) is given in \cite[Lemma 9.1]{CCJ20}. Using this function $h$, we can define an equivariant time-dependent perturbation on a small normal neighborhood $\nu\left(L_{W,\delta} / G_{W} \right) \cong \left[ L_{W,\delta} / G_{W} \right] \times (-\epsilon,\epsilon)$ by
\begin{align*}
\overline{h} : \nu\left(L_{W,\delta} / G_{W} \right) \times S^{1} &\to \mathbb{R} \\
(x,s,t) &\mapsto h(Fl^{\mathcal{R}}_{t}(x))\rho(s)
\end{align*}
where $\epsilon$ is a small positive number and $\rho$ is a nonnegative small bump function supported in $(-\epsilon,\epsilon)$.

We consider a local Floer complex $CF^{\bullet}_{loc}\left(L_{W,\delta} / G_{W}, H_{S^{1}} \right)$ which is generated by Hamiltonian orbits of $H_{S^{1}} = H + \overline{h}$. A differential is defined by counting pseudo-holomorphic cylinders contained in $\nu\left(L_{W,\delta} / G_{W} \right)$. Its cohomology is denoted by $HF^{\bullet}_{loc}\left(L_{W,\delta} / G_{W}, H_{S^{1}} \right)$. It is known that as a $\mathbb{Z}/2$-graded vector space, there is an isomorphism (cf. \cite{KvK16})
$$HF^{\bullet}_{loc}\left(L_{W,\delta} / G_{W}, H_{S^{1}} \right) \cong H^{\bullet}_{\mathrm{Morse}}\left(L_{W,\delta} / G_{W}, h; \mathbb{Z} \right).$$

\begin{defn}
The monodromy Reeb orbit $\Gamma_{W} \in CF^{\bullet}_{loc}\left(L_{W,\delta} / G_{W}, H_{S^{1}} \right)$ is a cochain which corresponds to a fundamental class $\left[ L_{W,\delta} / G_{W} \right] \in H^{\bullet}_{\mathrm{Morse}}\left(L_{W,\delta} / G_{W}, h; \mathbb{Z} \right)$.
\end{defn}

Moreover, it was showed that $\Gamma_{W}$ is a cocycle in \cite[Lemma 9.7]{CCJ20}. Even though we did not define an orbifold symplectic cohomology theory in full generality, we regard $\Gamma_{W}$ as a symplectic cohomology class according to the above Floer theoretic discussion. Later, it will be used as an input of a closed-open map.

In our curve case, $\Gamma_W$ is represented by a union of loops around punctures. 
  \begin{prop}\label{prop:Gamma computation}
  A class $\Gamma_W$ is given by a sum of Hamiltonian orbits, geometrically represented by the following element of $\pi_1^{\mathrm{orb}}\left(\mathbb P^1_{a,b,c} \right)$:
  \begin{enumerate}
    \item \textbf{Fermat type} \bigg( $\simeq \mathbb P^1_{p,q,\frac{pq}{gcd(p,q)}} $\bigg): $\Gamma_W \leftrightarrow \gamma_3^{-1}$,
    \item \textbf{chain type} \bigg($\simeq \mathbb P^1_{pq,q, \frac{pq}{gcd(p-1,q)}}$\bigg): $\Gamma_W \leftrightarrow \gamma_1^{1-p}+\gamma_3^{-1}$,
    \item \textbf{loop type} \bigg($\simeq \mathbb P^1_{pq-1, pq-1, \frac{pq}{gcd(p-1, q-1)}}$\bigg): $\Gamma_W \leftrightarrow \gamma_1^{1-p}+\gamma_2^{1-q}+\gamma_3^{-1}$.  
  \end{enumerate}
  \end{prop}
  \begin{proof}
    Recall the covering homomorphisms $\phi: \pi_1^{\mathrm{orb}}\left(\mathbb P^1_{a,b,c} \right) \to G_W$ in Proposition \ref{prop:homo}. If a winding number of $\Gamma_W$ around a puncture corresponds to $\gamma_i$ is $m$, then the image $\phi(\gamma_i^m)$ is equal to the monodromy $g_W$.  It is enough to find a negative integer $k$ (because of the orientation) with the minimal absolute value among those who satisfying 
    \[\phi(\gamma_i ^k)=g_W \in G_W.\]
    \begin{enumerate}
    \item \textbf{Fermat type $x^p+y^q$} \\
    We have 
      \[g_W=\phi(\gamma_3^{-1}) = (1,1) \in \Z/p\times \Z/q.\]
    
    \item \textbf{Chain type $x^p+xy^q$}\\
    We have 
    \[g_W = \phi(\gamma_1^{1-p}) = \phi(\gamma_3^{-1}) = 1-p \in \Z/pq.\]
    
    \item \textbf{Loop type $x^py+xy^q$}\\
    We have 
    \[g_W = \phi(\gamma_1^{1-p})=\phi(\gamma_2^{1-q})=\phi(\gamma_3^{-1}) = 1-p \in \Z/pq-1.\] 
    \end{enumerate}
  \end{proof}
  

\subsection{$\boldsymbol{\Gamma_W}$ and the polynomial $\boldsymbol{g(x,y,z)}$}
We show that the closed-open map image of orbit $\Gamma_W$ (with boundary $\bb$-deformations) provides the polynomial $g(x,y,z)$.
For more precise formulation, let us first recall the notion of  \textit{Kodaira--Spencer map} defined by
Fukaya--Oh--Ohta--Ono \cite{FOOO_MS}.  For a Lagrangian torus $\bL$ inside a compact toric manifold $M$,  they constructed a map  $\mathrm{KS}: QH^\bullet (M) \to \Jac(W)$ from a closed-open map with boundary $\bL$. $T^n$-equivariant perturbation guarantees that the output is a multiple of the fundamental class $[\bL]$, and its coefficient gives an element $\Jac(W)$. In \cite{ACHL}, 
the case of $\mathbb{P}^1_{a,b,c}$ with $QH^\bullet_{\mathrm{orb}}(\mathbb{P}^1_{a,b,c})$ has been constructed, which we will use in our case.

Since our Milnor fiber is non-compact, we consider symplectic cohomology instead of quantum cohomology \cite{Se06}.
Following the general definition proposed in \cite{CL_ks}, it is natural to define Kodaira--Spencer map in our case as
 \[\mathrm{KS}^\bold b: SH^\bullet\big([M_W/G_W]\big)  \to H^\bullet(CF(\bL, \bL)\otimes \C [x,y,z], m_1^{\bb,\bb}) \]
 where we consider closed-open map with boundary on $(\bL,\bb)$. Here, the target of the map is the Koszul complex of $W_\bL$. We omit its definition
 since we would need a proper definition of orbifold symplectic cohomology. Instead, we just consider the image of $\Gamma_W$.
 \begin{defn}
  Kodaira--Spencer invariant of $\Gamma_W$ is defined as 
    \[\mathrm{KS}^\bold b (\Gamma_W) := \sum_l \mathrm{CO}_{\mathbb{L},l} (\Gamma_W)(\underbrace{\bb, \ldots, \bb}_{l}). \]
 \end{defn}
From \cite{ACHL} and \cite{CL_ks}, $\mathrm{KS}^\bold b (\Gamma_W)$ is a multiple of $[\bL]$, and its coefficient can be regarded as an element of $\Jac(W_\bL)$. Let us explain how to compute this invariant. To relate the insertion of orbifold quantum cohomology generators and symplectic cohomology generators, we will use the work of Tonkonog \cite{Tong}. We may assume that $\Gamma_W$ is
defined by autonomous Hamiltonian (see \cite{BO}). Then, the following is the main theorem of this paper.

  \begin{thm}\label{thm:ksg}
  The Kodaira--Spencer map sends $\Gamma_{W}$ to the polynomial $g(x,y,z)$ such that 
$$W_\bL(x,y,z) = W^T(x,y) +xy \cdot g(x,y,z).$$
More precisely, we have
$$\mathrm{KS}^\bold b(\Gamma_W)= 1_\bL \otimes g(x,y,z),$$
where
  $$g(x,y,z) =
  \begin{cases}
   z   & \text{for } \;\;F_{p,q},  \\
   z-x^{p-1}  & \text{for } \;\;C_{p,q}, \\
    z-x^{p-1} -y^{q-1} & \text{for } \;\;L_{p,q}.
  \end{cases}$$
  \end{thm}
  \begin{proof}
First, note that closed-open map preserves $\Omega$- and $H_1$-grading of $SH^\bullet$ and $CF^\bullet(\bL, \bL)$.
Next,  if we consider a compactification of $[M_W/G_W]$ into $\mathbb{P}^1_{a,b,c}$,  holomorphic discs that
contribute to the  closed-open map from $QH^\bullet(\mathbb{P}^1_{a,b,c}) \to \Jac(W_\bL)$ has been worked out in \cite{ACHL}.
Proceeding similarly, we can show that the image is a multiple of $[\bL]$ (from $\Z/2$-symmetry), and can be regarded as an element of $\Jac(W)$. From the grading consideration, the only possible contribution of $\mathrm{KS}^\bold b (\gamma_i^k)$ is a $k$-th power of a variable associated with the immersed corner of $\bL$ opposite to the puncture  (multiplied by the fundamental class $[\bL]$). 

We can see such a polygon in a picture explicitly as in Figure \ref{fig:KS}.
     \begin{enumerate}
      \item Fermat type $F_{p,q}$: $\mathrm{KS}^\bold b(\Gamma_W)=z\cdot 1.$
      \item Chain type $C_{p,q}$: $\mathrm{KS}^\bold b(\Gamma_W)=(z-x^{p-1})\cdot1.$ 
      \item Loop type $L_{p,q}$: $\mathrm{KS}^\bold b(\Gamma_W)=(z-x^{p-1}-y^{q-1})\cdot 1.$  
    \end{enumerate}

 These polynomials are exactly $g(x,y,z)$ we want.  More precisely, let $C$ be one of the punctures of $[M_W/G_W]$, and we replace a puncture $C$ (of branching order $N_C$) with a $\Z/N_C$ orbifold point.  Then $\gamma_C^k$ corresponds to the orbifold cohomology class $[k/N_C]$(coming from the twisted sector).
The above (signed) computation indeed follows from the counting of orbifold holomorphic polygons with $[k/N_C]$ insertion in the interior
according to Proposition \ref{prop:Gamma computation}.  The orbifold $\mathbb{P}^1_{a,b,c}$ has a Riemann sphere as an underlying space, and an orbifold holomorphic disc near the $[k/N_C]$ insertion is nothing but a holomorphic map from a disc to the Riemann sphere with tangency order $k$, which is Fredholm regular. Then we use Tonkonog's domain stretching technique to compare the orbifold and $\Gamma_W$ invariants. In our case, the hypersurface $\Sigma$ in \cite{Tong} is just a point (or points) playing the role of Donaldson hypersurface.
Technically speaking, Tonkonog used a variant of linear Hamiltonian, whereas we use quadratic Hamiltonian
but since we are only concerned about the orbit $\Gamma_W$, we can use the arguments by Seidel \cite{Se06} that the continuation map is compatible with a closed-open map to relate these two settings. 

Let us explain more details assuming the familiarity with \cite{Tong}. By clever choice of a sequence of Hamiltonians (called $S$-shaped ) for the domain stretching, the standard $J$-holomorphic discs break into parts which share Reeb orbits as same asymptotics.  Reeb orbits for $S$-shaped Hamiltonian are divided into types $I, II, III, IV_a, IV_b$,
depending on their positions with respect to Liouville collar. A key part of the proof is to show that only type $II$ Reeb orbit appears in the breaking.
In our case, if breaking occurs at type $I, IV_a, IV_b$ Reeb orbits (which are constant orbits), then collecting the parts from this constant orbit to $\Sigma$ 
we get a non-trivial sphere that maps to $X$. The starting polygon in $X$ does not intersect other vertices of $\mathbb{P}^1_{a,b,c}$, and
this intersection number with perturbed $J$-holomorphic curve is positive, so the sphere should not intersect other vertices. Therefore, such a sphere cannot exist.
This excludes these types of Reeb orbits as breaking orbits. The argument about type $III$ orbit using no escape lemma still applies to our case.
One can see that any disc bubble would increase the intersection with vertices of the orbisphere, hence do not occur. 
The rest of the argument is the same as the reference and we obtain the desired comparison results. We refer readers to \cite{Tong} for more details.
  \end{proof}


\section{Quantum cap action and hypersurface restriction}\label{sec:BHMS}

In this section, we prove that the mirror functor $\cF^\bL$ intertwines the quantum cap action $\cap \Gamma_W$ (see \cite[Definition 4.1]{CCJ20} for the definition) and the multiplication $\times g(x,y,z)$. 

Let us start with the simplest but essential part of the whole theorem. 
Observe that for any Lagrangian $L$, its mirror matrix factorization $M_L = \mathcal{F}^\bL(L)$ carries two canonical endomorphisms. 
One is just a multiplication $\times g(x,y,z):M_L \to M_L$, while the other one is given as follows:
  
  \begin{defn} An endomorphism $(\cap \Gamma_W)^\bold b$ of $M_L$ is defined by
\[(\cap \Gamma_W)^\bold b : M_L \to M_L, \hskip 0.3cm \alpha \to \sum_l(\cap \Gamma_W)(\alpha, \underbrace{\bb, \ldots, \bb}_l).\]
  \end{defn}
  It is easy to check that $(\cap \Gamma_W)^\bold b$ is a closed endomorphism. Note that this is a version of a localized mirror functor adapted to the quantum cap action. 
  
  \begin{prop} \label{prop:KS}
There is a homotopy  $H^\bL : M_L \to M_L$ such that
 $$(\cap \Gamma_W)^{\bb} - \times g(x,y,z) = [m_1^{0,\bb}, H^\bL].$$
  \end{prop}

  \begin{proof}
  Consider a moduli space similar to that of \cite[Proposition 4.2]{CCJ20} such that an interior marked point is allowed to be in the lower part of the disc along a vertical geodesic as in Figure \ref{fig:HtoKS}. We allow arbitrary insertions of $\bb$ in the lower part of the boundary of the disc.
 The figure illustrates two of the boundary components of $\mathcal{M}$, which correspond to
 $(\cap \Gamma)^{\bb}$ and $ \times g(x,y,z)$. For the latter, the disc bubble gives $\mathrm{KS}^\bold b(\Gamma_W)=1_\bL \otimes g(x,y,z)$,
 and hence the main disc component has to be a constant disc, and it corresponds to the multiplication by $g(x,y,z)$. 
 There are two additional boundary components which contribute to $[m_1^{0,\bb}, H^\bL]$. This proves the proposition.
  \end{proof}
  
  \begin{figure}[h]
  \centering
  \includegraphics[scale=0.6, trim=0 600 0 30, clip]{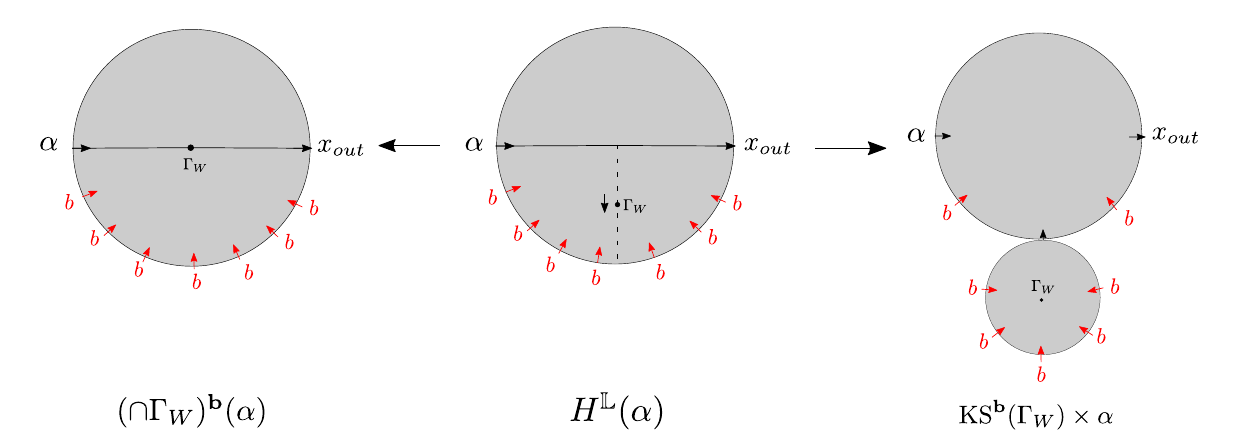}
  \caption{\label{fig:HtoKS} Deformed cap action and Kodaira--Spencer invariant $g(x,y,z)$}
  \end{figure}

Proposition \ref{prop:KS} can be generalized to a bimodule homomorphism. First, note that an $\AI$-functor $\cF^\bL$ makes $\mathcal{MF}(W^T+xyg)$  a $\mathcal{WF} -\mathcal{WF}$ bimodule with respect to which $\cF^\bL$ is a bimodule homomorphism. We view a multiplication $\times g$ as a bimodule homomorphism from $\mathcal{MF}$ to itself. 

  \begin{prop}\label{first square commutes}
  The following diagram of bimodule homomorphisms commutes up to homotopy. 
    \[
      \begin{tikzcd}
      \WF([M_W/G_W]) \arrow[r,"\cap \Gamma_W"] \arrow[d, "\cF^\bL"]& \WF([M_W/G_W]) \arrow[d, "\cF^\bL"]\\
      \mathcal {MF}(W^T+xyg) \arrow[r, "\times g"] &\mathcal {MF}(W^T+xyg)
      \end{tikzcd}
    \] 
  \end{prop}
  
\begin{proof}
In the course of the proof, we use the notation
\begin{align*} 
  m_{\mathrm{bimod}}(a_1, \ldots, a_{k_1}, \underline b, \ldots, a_{k_1+k_2+1}) = (-1)^{\diamond_{k_1}}m_{k_1+k_2+1}(a_1, \ldots, a_{k_1}, \underline b, \ldots, a_{k_1+k_2+1})\\
  \cF^\bL_{\mathrm{bimod}}(a_1, \ldots, a_{k_1}, \underline b, \ldots, a_{k_1+k_2+1}) = (-1)^{\diamond_{k_1}}\cF^\bL_{k_1+k_2+1}(a_1, \ldots, a_{k_1}, \underline b, \ldots, a_{k_1+k_2+1})
\end{align*}  
 to indicate that $m$ or $\cF^\bL$ are viewed as a bimodule structure and homomorphism between them. Here, $\diamond_j =\sum_{i=1}^j(\deg a_i-1)$. As an intermediate step, we introduce the following operations (See Figure \ref{fig:HandHtilde}):

\begin{enumerate}

  \item generalize $(\cap \Gamma_W)^\bold b$ as follows:
  \begin{align*}
    (\cap \Gamma_W)^\bold b : CW^\bullet(L_0, L_1)\otimes &\cdots \otimes CW^\bullet(L_{k-1}, L_k) \to \Hom (M_{L_0}, M_{L_k}),\\
    (\cap \Gamma_W)^\bold b_k (a_1, \ldots, a_k)(\alpha) &:= \sum_l(\cap \Gamma_W)_{k+l+1} (a_1, \ldots, a_k, \underline \alpha, \bb, \ldots, \bb).
  \end{align*}
  
  \item also generalizes $H^\bL$ by allowing $\mathcal{WF}$ inputs in the upper-boundary of the popsicle maps as in the Figure \ref{fig:HandHtilde}.
    
  \item define a pre-homomorphism of bimodules $\WT H^\bL$ as follows:
    \begin{align*}
    \WT H^\bL : \mathcal{WF}^{\otimes k_1} \otimes \underline{\mathcal{WF}}\otimes \mathcal{WF}^{\otimes k_2} &\to \mathcal{MF}(W^T+xyg),\\
    \WT H^\bL(a_1,\ldots, a_{k_1}, \underline b, \ldots, a_{k_1+k_2+1})(\alpha) &:= \sum_{l} (\cap \Gamma_W)_{k_1+k_2+2+l}( a_1, \ldots, \underline b, \ldots, a_{k_1+k_2+1}, \alpha, \underbrace{\bold b,\ldots, \bold b}_l)
    \end{align*}
  \end{enumerate}  
  
  The result follows from codimension one boundary configurations associated to (2) and (3), which are also illustrated in Figure \ref{fig:HandHtilde}.
  note that  $\mathrm{KS}^\bold b(\Gamma_W)$ and a disc potential is a multiple of unit. If degenerations involve such configurations, then it contributes to zero \textit{except} a single important case;  $k=0$ for (2). That is Proposition \ref{prop:KS} with Figure \ref{fig:HtoKS}. Otherwise, we can skip them. 
  
  Analyzing the case $k>0$ for (2), we get
  \begin{align*} 
  (\cap \Gamma_W)^\bold b (a_1, \ldots, a_k) & = \sum (-1)^{\diamond_j}H^\bL(a_1,\ldots, a_j, m_l(a_{j+1}, \ldots, a_{j+l}), \ldots, a_k)\\
  &+\sum m_2 \left(H^\bL(a_1,\ldots, a_j), \cF^\bL(a_{j+1}, \ldots, a_k) \right) \\
  &+ \sum m_2\left(\cF^\bL(a_1,\ldots, a_j), H^\bL(a_{j+1}, \ldots, a_k) \right).
  \end{align*}
  This relation says that $H^\bL_{k>0}$ provides a null-homotopy for $(\cap \Gamma_W)^\bold b_{k>0}$. Combined with Proposition \ref{prop:KS}, we conclude that $H^\bL$ provides a homotopy between $(\cap \Gamma_W)^\bold b$ and $\times g(x,y,z)$. Similarly,
  \small
    \begin{align*}
    \sum& m_2  \left( (\cap \Gamma_W)^\bold b (a_1, \ldots, a_j), \cF^\bL_{\mathrm{bimod}}(a_{j+1}, \ldots, \underline b, \ldots, a_{k_1+k_2+1})\right)\\
    + \sum &(-1)^{\diamond_j}\cF^\bL_{\mathrm{bimod}}  \left(a_1, \ldots, \underline{(\cap \Gamma_W)(a_{j+1}, \ldots, \underline b, \ldots, a_{j+l})}, \ldots, a_{k_1+k_2+1}\right) \\
    = \sum & (-1)^{\diamond_j} m_2\left( \cF^\bL(a_1, \ldots, a_j),  \WT H^\bL(a_{j+1}, \ldots, \underline b, \ldots, a_{k_1+k_2+1} ) \right) + \sum m_2\left( \WT H^\bL(a_{1}, \ldots, \underline b, \ldots, a_{j}) , \cF^\bL(a_{j+1}, \ldots, a_{k_1+k_2+1}) \right) \\
     + \sum &(-1)^{\diamond_j}\WT H^\bL(a_1, \ldots, m_{k} (a_{j+1}, \ldots, a_{j+k}), \ldots,  \underline b, \ldots, a_{k_1+k_2+1}) + \sum (-1)^{\diamond_j + \deg b}\WT H^\bL(a_1, \ldots , \underline b, \ldots, , m_{k} (a_{j+1}, \ldots, a_{j+k}), \ldots, a_{k_1+k_2+1})\\
     + \sum & (-1)^{\diamond_j} \WT H^\bL(a_1, \ldots, m_{\mathrm{bimod},k+1} (a_{j+1}, \ldots,  \underline b, \ldots, a_{j+k}), \ldots a_{k_1+k_2+1}).
   \end{align*}
   \normalsize
   
   This relation tells us that $\WT H^\bL$ provides a homotopy between $\cF^\bL_{\mathrm{bimod}} \circ (\cap \Gamma_W)$ and $(\cap \Gamma_W)^\bold b \circ \cF^\bL_{\mathrm{bimod}}$.  
   
  \begin{figure}[h]
  \centering
  \includegraphics[scale=0.75, trim=0 400 0 30, clip]{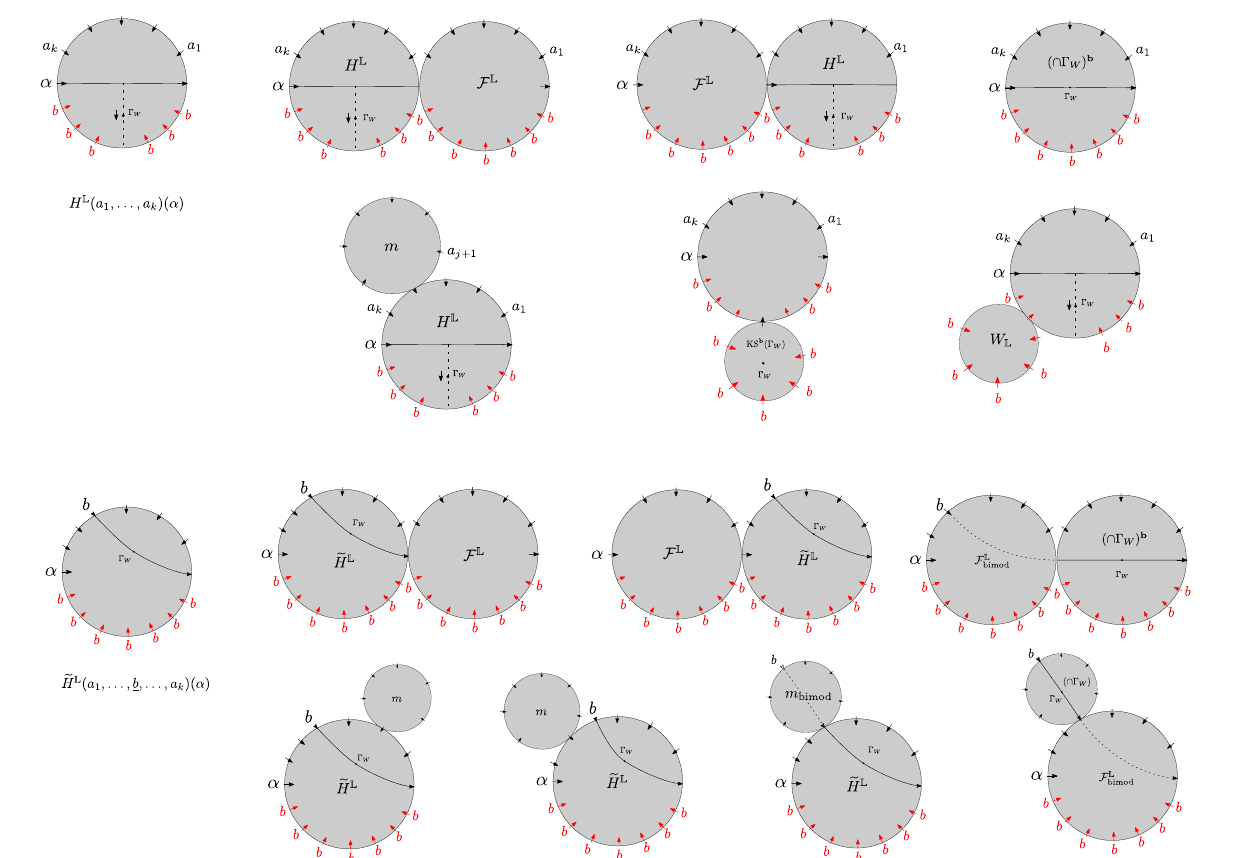}
  \caption{\label{fig:HandHtilde}$H^\bL$, $\WT H^\bL$ and their boundary configurations}
  \end{figure}

  Combine these two and define $K^\bL := \WT H^\bL + H^\bL\circ \cF^\bL_{\mathrm{bimod}}$. It provides a desired homotopy between $\cF^\bL_{\mathrm{bimod}} \circ (\cap \Gamma_W)$ and $(\times g)\circ \cF^\bL_{\mathrm{bimod}}$, which proves the theorem. 
  \end{proof}
  
From the algebraic non-sense, we obtain Theorem \ref{thm:c1} as a corollary of the Proposition \ref{first square commutes}.
(note that our polynomial $g$ is of the form $z-f(x,y)$ regardless of a type of $W$.  Therefore, we can apply \cite[Corollary 6.7]{CCJ20} to  $\mathcal{MF}(W^T+xyg)$.)
  
  \begin{cor} \label{cor:triangle}
  The following diagram commutes up to homotopy;
  \[\begin{tikzcd}
    \WF([M_W/G_W]) \arrow[d, "\cF^\bL"] \arrow[r, "\cap \Gamma_W"] & \WF([M_W/G_W]) \arrow[d, "\cF^\bL"] \arrow[r]& \cF(W, G_W) \arrow[d, "\widetilde{\cF}^\bL"] \arrow[r]& \phantom a \\
    \mathcal{MF}(W^T+xyg) \arrow[r, " \cdot g"] & \mathcal{MF}(W^T+xyg) \arrow[r] & \mathcal{MF}(W^T+xyg)|_g \arrow[r] & \phantom a 
    \end{tikzcd}\]
  Here the third one in each row denotes the mapping cone of the first morphism and we define the bimodule map $\widetilde\cF^\bL$  as $$\widetilde\cF^\bL( \ldots, \underline {a+\epsilon b}, \ldots) := (\cF^\bL(\ldots, \underline a, \ldots)+ K^\bL(\ldots, \underline b, \ldots)) +\epsilon \cF^\bL(\ldots, \underline b, \ldots).$$ Each row is a distinguished triangle of bimodules. All vertical lines induce quasi-isomorphisms. 
  \end{cor}
  
  Moreover, we also have the following lemma.
  
  \begin{lemma}\label{lem:es}
    The restriction $i^*: \MF(W^T+xyg) \to \MF(W^T)$ is essentially surjective. 
  \end{lemma}
  \begin{proof}
  $W^T$ is an isolated singularity. Therefore $\MF(W^T)$ is generated by the matrix factorization of stabilized residue field $k^{\mathrm{stab}}_{W^T}$ (see \cite{Dy}). It is enough to find matrix factorization $M \in \MF(W^T+xyg)$ such that $M|_{g=0} =k^{\mathrm{stab}}_{W^T}$. 
  \begin{align*}
    \textrm{Fermat type} \hskip 0.2cm &M = \begin{pmatrix} x&y\\ -y^{q-1} & x^{p-1}+yz\end{pmatrix} \times \begin{pmatrix} x^{p-1}+yz &-y\\ y^{q-1} & x\end{pmatrix}\\
    \textrm{Chain type} \hskip 0.2cm &M= \begin{pmatrix} x&y\\ -y^{q-1} & yz\end{pmatrix} \times \begin{pmatrix} yz &-y\\ y^{q-1} & x\end{pmatrix}\\
    \textrm{Loop type} \hskip 0.2cm  &M= \begin{pmatrix} x&y\\ 0 & yz\end{pmatrix} \times \begin{pmatrix} yz &-y\\ 0 & x\end{pmatrix}
  \end{align*}
  In fact, these matrix factorizations are $M_L = \cF^\bL(L)$ for explicit Lagrangians $L$ (for example, see Remark \ref{rmk:residuefield} for the Fermat case).
  \end{proof}
  
  Lemma \ref{lem:es} implies $\mathcal{MF}(W^T+xyg)|_g$ is another $\AI$-model for $\mathcal {MF}(W^T)$. This establishes an equivalence between $\cF(W, G_W)$ and $\mathcal {MF}(W^T)$ as $\AI$-bimodules over $\mathcal{WF}([M_W/G_W])$.


\section{Relation to Auslander--Reiten theory of Cohen--Macaulay modules} \label{sec:AR}
For  ADE singularity $W^T$, its matrix factorization category is of finite type, which means that there are only finitely many indecomposable ones.
Its Auslander--Reiten quiver has been described by Yoshino \cite{Yo}, and we find a corresponding Lagrangian Floer theory under
Berglund--H\"ubsch duality. More precisely, we find non-compact Lagrangians in the Milnor fiber of $W$ for each indecomposable matrix factorizations and realize all of Auslander--Reiten almost split exact sequences as Lagrangian surgery exact sequences.


\subsection{Auslander--Reiten theory}
Recall that a ring $R$ is Cohen--Macaulay (CM) if its Krull dimension equals its depth. For example, complete intersections give CM rings.  CM rings and their modules play a central role in the theory of commutative algebra. $R$ module $M$ whose depth (minimal length of projective resolution) equals the dimension of $R$ is called a maximal CM (in short MCM) module.

  Auslander and Reiten developed a classification theory of indecomposable objects and in particular defined an associated quiver, which is called Auslander--Reiten (AR) quiver $Q$. Vertices of $Q$ are indecomposable MCM modules and its arrows are given by irreducible morphisms (roughly the minimal morphisms that do not factor nontrivially). Also, there are dotted arrows, called AR translation $\tau$. Given an indecomposable module $M$ and its AR translation $\tau(M)$, there is
an associated AR exact sequence:
$$ 0 \to \tau(M) \to N \to M \to 0$$
where $N$ is the direct sum of MCM modules that are sources of arrows to $M$ (or
targets of arrows from $\tau(M)$). Therefore, one can read off all AR exact sequences whenever an AR quiver is given.
 
On the other hand, Eisenbud proved that MCM modules of $R$ correspond to $\Z/2$-graded matrix factorizations of the defining function via periodic resolution, and hence the above classification results can be translated into those of matrix factorizations \cite{E80}.
We refer readers to the excellent book by Yoshino \cite{Yo} for more details.

\subsection{Localized mirror functor and AR exact sequence}
In this section, we remind readers that we work with $\Z/2$-graded ($\AI$ or DG) categories.
Given an AR sequence,
\begin{equation}\label{eq:ar}
 0 \to \tau(M) \to N \to M \to 0
 \end{equation}
there exist a corresponding extension element $\alpha \in \mathrm{Ext}^1(M,\tau(M))$.
Given a $\Z/2$-graded matrix factorization $(d_{10} : P^{1} \to P^{0}, d_{01} : P^{0} \to P^{1})$ of $W$,
we may write $(d_{10}, d_{01})$ as a  pair of polynomial matrices $(\phi,\psi)$  satisfying  $\phi \cdot \psi = \psi \cdot \phi = W \cdot id$. 
We recall the following fact.
\begin{prop}[\cite{Yo} Proposition 3.11, Proposition 7.7]
For a matrix factorization $M=(\phi,\psi)$ of $W$, its AR translation $\tau(M)$ is given by $(\psi, \phi)$.
\end{prop}

We observe that a Lagrangian with opposite orientation corresponds to the Auslander translation.
\begin{lemma}\label{lem:lagtr}
If a Lagrangian $L$ maps to $M=(\phi,\psi)$ under localized mirror functor, then its orientation reversal $L[1]$ maps to a matrix factorization $\tau(M).$  
\end{lemma}
\begin{proof}
Localized mirror functor takes $L$ to a matrix factorization $P^{0} =  CF^0(L,\bL), P^1 = CF^1(L,\bL)$ with $d= - m_1^{0,\bb}$.
For the orientation reversal $L[1]$, $CF^\bullet(L,\bL)$ and $CF^\bullet(L[1],\bL)$ have the same set of generators, but with opposite $\Z/2$-grading, and
$m_1^{0,\bb}$ is modified accordingly (this is given by the same polygon). It remains to check the related signs. 
Since the orientation of $\bL$ does not change at $X,Y,Z$-corner of $\bL$, there are only 4 possible orientation choices for such a polygon,
and we can check each case.
 From the sign convention of Fukaya category of surface, the orientation of $L$ does not contribute to a sign. The only difference after orientation reversing is that the degrees of first input and output are interchanged. But the sum of those two degrees is still $1$, so the total sign of the polygon is not changed. This proves the claim.
\end{proof}

Now, we want to relate AR exact sequences with Lagrangian surgery exact sequences.
First, let us recall that a Lagrangian surgery can be related to a cone.
The following lemma is a modification of Lemma 5.4 of \cite{Ab08}.
\begin{figure}[h]
\begin{subfigure}{0.43\textwidth}
\includegraphics[scale=0.6]{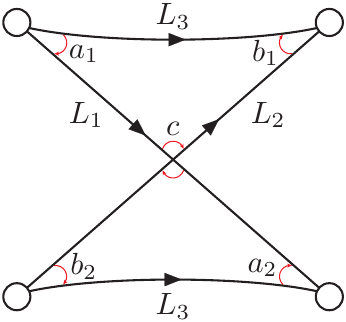}
\centering
\end{subfigure}
\begin{subfigure}{0.43\textwidth}
\includegraphics[scale=0.6]{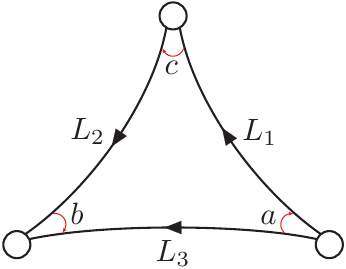}
\centering
\end{subfigure}
\centering

\caption{Lagrangian surgery at $c$}
\label{fig:conec}
\end{figure}

\begin{lemma}\label{lem:surc}
Let $L_1$ and $L_2$ be unobstructed non-compact curves which intersect transversally and
minimally at a single point $c$ in the interior so that $c \in CW^1(L_1,L_2)$, or $L_1$ and $L_2$ are disjoint 
but conical to the same puncture, with the wrapped generator $c  \in CW^1(L_1,L_2)$
as in Figure \ref{fig:conec}.
Then, $L_3$ that is obtained after Lagrangian surgery at $c$ is  isomorphic to the twisted complex $\mathrm{Cone}(c):= L_1 \stackrel{c}{\to} L_2$.
\end{lemma}
\begin{proof}
We find an explicit isomorphism as follows.
We define $ a \in CW^0(L_3, L_1), b \in CW^0(L_2, L_3)$ as in Figure \ref{fig:conec}.
In the first case, $a, b$ are taken to be the sum of two generators, $a = a_1 - a_2$ and $b = b_1 + b_2$.
Then we may regard $a,b$ as $a \in \Hom^0(L_3, \mathrm{Cone}(c)), b \in \Hom^0(\mathrm{Cone}(c), L_3)$. Then, it is enough to show that $m_2^{Tw}(a,b)=\be_{L_3}$
and $m_2^{Tw}(b,a)= \be_{\mathrm{Cone}(c)} = \be_{L_1} + \be_{L_2}$. By definition of $m_2^{Tw}$ for twisted complexes, we have $m_2^{Tw}(a,b) = m_3(a,c,b) = \be_{L_3}$, and
$m_2^{Tw}(b,a) = m_3(c,b,a) + m_3(b,a,c) = \be_{L_1} + \be_{L_2}$.
We can prove similarly when the orientations of $L_1$ and $L_2$ are reversed.
\end{proof}
\begin{lemma}\label{lem:surc2}
In the setting of the previous lemma, we have an exact triangle of 
matrix factorizations for $W_\bL$.
\begin{equation}\label{eq:arl}
 \big( CW^\bullet(L_1,\bL), -m_1^{0,\bb}\big) \stackrel{\mathcal{F}^\bL_1(a)}{\longrightarrow}  \big( CW^\bullet(L_3,\bL), -m_1^{0,\bb}\big)  \stackrel{\mathcal{F}^\bL_1(b)}{\longrightarrow}   \big( CW^\bullet(L_2, \bL), -m_1^{0,\bb}\big) \to   \end{equation}
We may set $g(x,y,z)=0$ in the above to obtain the exact triangle for $W^T$.
 
 \end{lemma}
 \begin{proof}
Those two triangles are just 
\[\cF^\bL(L_1) \to \cF^\bL(L_3) \to \cF^\bL(L_2) \to \hskip 0.2cm\textrm{and}\hskip 0.2cm (i^*\circ \cF^\bL)(L_1) \to (i^*\circ \cF^\bL)(L_3) \to (i^*\circ \cF^\bL)(L_2) \to.\]
The proposition follows from the fact that a localized mirror functor $\cF^\bL$ and the restriction $g(x,y,z)=0$ are both exact (see \cite[Proposition 6.5]{CCJ20}).
 \end{proof}

If an AR sequence \eqref{eq:ar} is the same as \eqref{eq:arl}, we will say
that AR sequence is realized as Lagrangian surgery exact sequence.
In this section, we prove the following.
\begin{thm}\label{thm:ADEAR}
For ADE curve singularity $W^T$, we find explicit Lagrangians in the Milnor fiber of $W$ corresponding to indecomposable matrix factorizations in the AR quiver.
Also, all AR sequences of matrix factorizations for ADE curve singularities $W^T$ can be realized as Lagrangian surgery exact sequence for $W$.
Namely, for any AR exact sequence \eqref{eq:ar}, we can find Lagrangians  $L_1,L_2,L_3$ in $[M_W/G_W]$ 
(with a Floer generator $c  \in CW^1(L_1, L_2)$) such that
\eqref{eq:arl} can be identified with  \eqref{eq:ar} in the cohomology category.
\end{thm}
\begin{remark}
Due to $\AI$-convention for matrix factorization category, the direction of a morphism $c$ in Floer theory is the opposite of
$\alpha \in \mathrm{Ext}^1(M,\tau(M))$.
\end{remark}
\begin{remark}
For an $\AI$-functor $\mathcal{F}^\bL|_{g=0}:\mathcal{WF}([M_W/G_W]) \to \mathcal{MF}(W^T)$, we will find several non-isomorphic Lagrangians
mapping to the same matrix factorization. This may look strange but this is because $\mathcal{F}^\bL|_{g=0}$ is not an equivalence.
\end{remark}

 In our definition of $\mathcal{MF}(W^T)$, trivial matrix factorization $(W^T,1)$ is homotopic to a zero object, but
 for AR theory in \cite{Yo}, $R=\C[x,y]/W^T$ (cokernel of the above factorization) is
 taken as an indecomposable MCM module. Hence  $\mathcal{MF}(W^T)$ corresponds to
 stable matrix factorization category in \cite{Yo}.
   
  Before going further, we recall the construction of a mapping cone in the matrix factorization (DG) category. Let $P = (P_{1}, P_{2})$, $P^{\prime} = (P_{1}^{\prime}, P_{2}^{\prime})$ be two matrix factorizations and $\phi = (\phi_{1}, \phi_{2}) : (P_{1}, P_{2}) \to (P_{1}^{\prime}, P_{2}^{\prime})$. We define a mapping cone of $\phi$, $\mathrm{Cone}(\phi)$, as a block matrix given by
\begin{equation*}
\begin{pmatrix}
P_{1}^{\prime}& \phi_{1}\\
0& P_{2}\\
\end{pmatrix}
\begin{pmatrix}
P_{2}^{\prime}& -\phi_{2}\\
0& P_{1}\\
\end{pmatrix}.
\end{equation*}
 
\subsection{$\boldsymbol{A_n}$-cases} 
For $A_n$-singularity given by $F_{n+1,2} = x^{n+1}+y^2$ with $R = \C[x,y]/F_{n+1,2}$,
the AR quiver can be described as follows.
AR quiver depends on the parity of $n$:  for even $n$,
$$\begin{tikzcd}
R \arrow[r,shift left] & M_{1} \arrow[l,shift left] \arrow[r,shift left] \arrow[out=240,in=300,dash,loop,swap,dashed] & M_{2} \arrow[l,shift left] \arrow[r,shift left] \arrow[out=240,in=300,dash,loop,swap,dashed] & \cdots \arrow[l,shift left] \arrow[r,shift left] & M_{\frac{n}{2}} \arrow[l,shift left] \arrow[out=250,in=290,dash,loop,swap,dashed] \arrow[l,loop right]
\end{tikzcd}$$
for odd $n$,
$$\begin{tikzcd}[row sep=small]
& & & & &N_{-} \arrow[dl,shift left]\\
R \arrow[r,shift left] & M_{1} \arrow[l,shift left] \arrow[r,shift left] \arrow[out=240,in=300,dash,loop,swap,dashed] & M_{2} \arrow[l,shift left] \arrow[r,shift left] \arrow[out=240,in=300,dash,loop,swap,dashed] & \cdots \arrow[l,shift left] \arrow[r,shift left] & M_{\frac{n-1}{2}} \arrow[l,shift left] \arrow[out=250,in=290,dash,loop,swap,dashed] \arrow[ur,shift left] \arrow[dr,shift left] \\
& & & & &N_{+} \arrow[ul,shift left] \arrow[uu,dash,dashed]
\end{tikzcd}$$

Here $M_k$ is the $2 \times 2$ matrix factorization given by 
\begin{equation*}
\begin{pmatrix}
y& x^{k}\\
x^{n+1-k}& -y\\
\end{pmatrix} \cdot
\begin{pmatrix}
y& x^{k}\\
x^{n+1-k}& -y\\
\end{pmatrix},
\end{equation*}
$N_-$  is the ($ 1 \times 1$ matrix) factorization
$$ (x^{\frac{n+1}{2}}-iy) \cdot (x^{\frac{n+1}{2}}+iy),$$
and $N_+$ is obtained from $N_-$ by switching two factors.
In fact, one can easily check
that  $2 \times 2$ matrix factorization $N_- \oplus N_+$ is isomorphic to $M_{\frac{n+1}{2}}$,
and this is what we obtain from a Lagrangian.

\begin{figure}[h]
\begin{subfigure}{0.43\textwidth}
\includegraphics[scale=0.55]{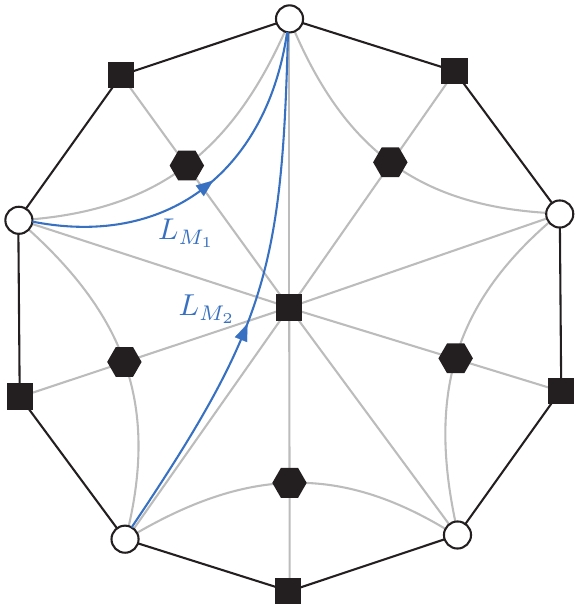}
\centering
\caption{\label{fig:anar} Lagrangians for $M_1, M_2$ in $A_4$-case}
\end{subfigure}
\begin{subfigure}{0.43\textwidth}
\includegraphics[scale=0.55]{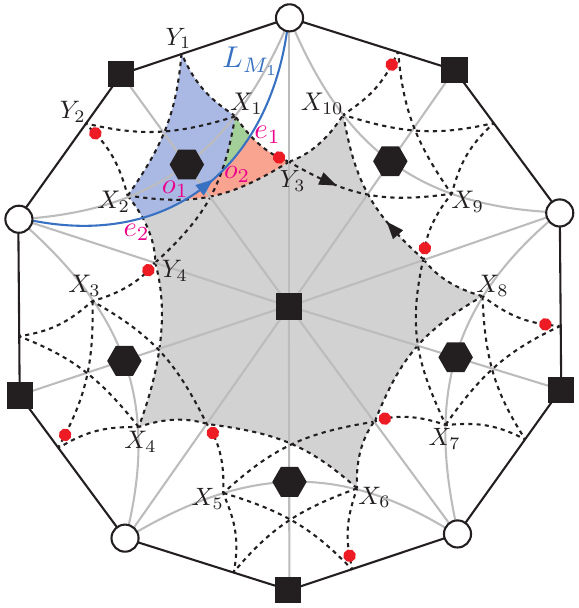}
\centering
\caption{Holomorphic polygons for $m^{0,\bb}_{1}$ calculation}
\label{fig:A4_disc}\end{subfigure}
\centering
\caption{\label{fig:famor}}
\end{figure}

We will find that as in Figure \ref{fig:famor} (A), $M_k$ corresponds to a non-compact Lagrangian connecting two punctures of the $2(n+1)$-gons. Namely, if we label the punctures as $v_{1}, \ldots, v_{n+1}$ counter-clockwise direction, then Lagrangian $L_{M_k}$ is given by the curve connecting
$v_1$ and $v_{1+k}$ (either project it to the quotient Milnor fiber or consider the $G_W$-copies in the Milnor fiber).

\begin{lemma}
$L_{M_k}$ maps to $M_k$ under $\mathcal{F}^\bL \mid_{z=0}$.
\end{lemma}
 
\begin{proof}
This is given by direct computation.
We give an explicit calculation for $M_{1}$ in $A_{4}$-case.
The Lagrangian $L_{M_1}$ intersects $\bL$ at 4 points, which we label as $o_{1}, o_{2}, e_{1},  e_{2}$.
$m_1^{0,\bb}$ counts holomorphic polygons connecting these intersections, while allowing $X,Y,Z$-corners in $\bL$.
In Figure \ref{fig:famor} (B), we illustrated some of the holomorphic polygons. 
Instead of drawing all of them, we write down the labels of the corners. For example, a grey polygon is denoted by $e_{2} X_{4} X_{6} X_{8} X_{10} o_{1}$. It contributes $x^{4}$ to the coefficient of $o_{1}$ in $m^{0,\bb}_{1}(e_{2})$.
Here is the list of all 8 polygons for $m^{0,\bb}_{1}$.
$$o_{1} Y_{3} e_{1},\; o_{1} X_{2} e_{2}, \;o_{2} X_{3} X_{5} X_{7} X_{9} e_{1}, \;o_{2} Y_{1} e_{2}, \;e_{1} Y_{2} o_{1}, \;e_{1} X_{1} o_{2}, \;e_{2} X_{4} X_{6} X_{8} X_{10} o_{1}, \;e_{2} Y_{4} o_{2}.$$
Signs can be computed accordingly (see the paragraph after Theorem \ref{thm:lmf})
and we obtain the following.
\begin{align*}
m^{0,\bb}_{1}(o_{1}) &= -y \cdot e_{1} + x \cdot e_{2} &
m^{0,\bb}_{1}(e_{1}) &= -y \cdot o_{1} + x \cdot o_{2} \\
m^{0,\bb}_{1}(o_{2}) &= x^{4} \cdot e_{1} + y \cdot e_{2} &
m^{0,\bb}_{1}(e_{2}) &= x^{4} \cdot o_{1} + y \cdot o_{2}
\end{align*}
This can be made into two $2\times 2$ matrices $M_{4}$, which is isomorphic to $M_1$ by a change of basis. The rest are similar and omitted.
\end{proof}

\begin{figure}[h]
\begin{subfigure}{0.43\textwidth}
\includegraphics[scale=0.5]{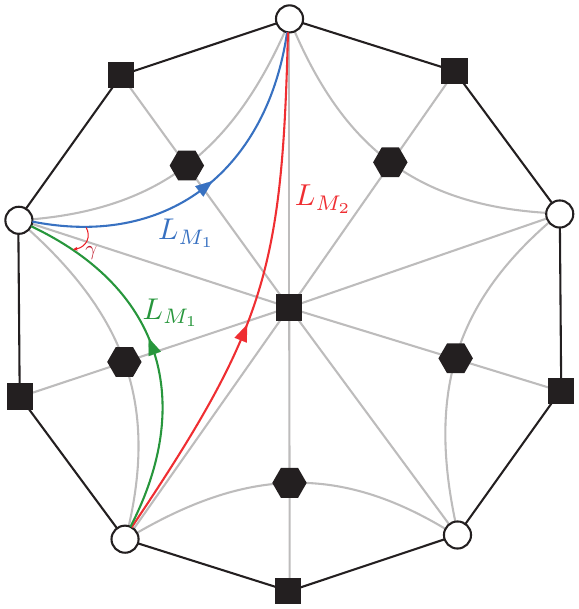}
\centering
\end{subfigure}
\begin{subfigure}{0.43\textwidth}
\includegraphics[scale=0.5]{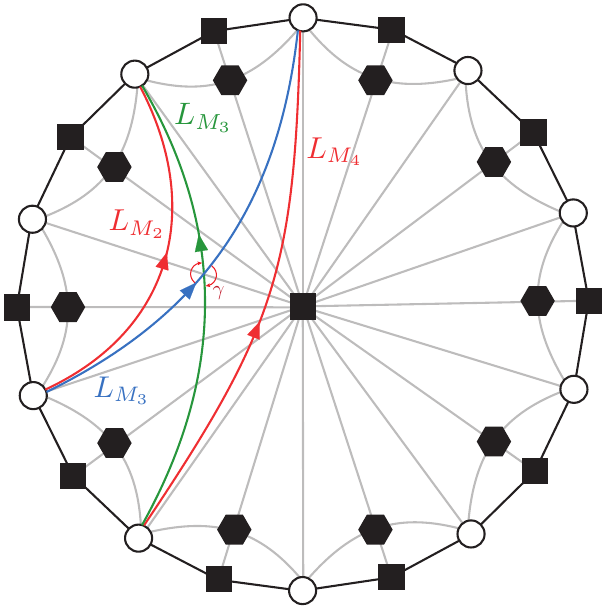}
\centering
\end{subfigure}
\centering
\caption{\label{fig:famor2}   }
\end{figure}

In Figure \ref{fig:famor2}, we take $\gamma$  as an odd morphism from $L_{M_1}$ to $L_{M_1}$,
and $\mathrm{Cone}(\gamma)$ becomes $L_{M_2}$. This realizes the AR sequence
$$0 \to \tau(M_1) \to M_2 \to M_1 \to 0$$
In general, the AR sequence of $A_n$ case is of the form 
$$0 \to  M_{k} \to M_{k-1} \oplus M_{k+1} \to M_{k} \to 0$$
In Chapter 9 of \cite{Yo}, a (half of) explicit morphism whose cone gives this AR sequence is specified:
\begin{equation}\label{eq:anmory}
\begin{pmatrix}
0& x^{k-1}\\
-x^{n-k}&0\\
\end{pmatrix}
\end{equation}

For $L_{M_{k}}$ with $k \geq 2$, we take an odd immersed generator $\gamma_k$
as follows. This is at the intersection between $L_{M_{k}}$
and the counter-clockwise rotation by $2\pi/(n+1)$ of $L_{M_{k}}$.
It is not hard to check that $\mathcal{F}^\bL_1(\gamma_k)$ equals the morphism \eqref{eq:anmory}
by counting suitable polygons.
For example, in the case of $A_4$ and $\gamma$ as in Figure \ref{fig:famor2}, we get $
- \begin{pmatrix}
0& x^{3} \\
-1&0\\
\end{pmatrix}$.
This is the morphism given in \eqref{eq:anmory} for $k=4, n=4$.

\begin{remark}
When we find any AR sequence, we do not calculate all morphisms in the sequence. It is enough to check whether Lagrangians in surgery sequence map to the corresponding indecomposable matrix factorizations in AR sequence. See Appendix \ref{sec:ac}.
\end{remark}

\subsection{$\boldsymbol{D_n}$-cases}
For $D_n$ singularity $x^{n-1} + xy^2$, AR quiver of indecomposable matrix factorizations
are given as follows. For odd $n$, we have

$$\begin{tikzcd}[column sep=scriptsize, row sep=small]
A \arrow[dddr] & & & & & & \\
& Y_{1} \arrow[ul] \arrow[dddl] \arrow[r] & M_{1} \arrow[ddl] \arrow[r] & Y_{2} \arrow[ddl] \arrow[r] & \cdots \arrow[ddl] \arrow[r] & M_{\frac{n-3}{2}} \arrow[ddl] \arrow[dr,shift left]& \\
B \arrow[uu,dash,dashed] \arrow[ur] & & & & \cdots & & X_{\frac{n-1}{2}} \arrow[ul,shift left] \arrow[dl,shift left] \arrow[l,loop right,dash,dashed] \\
& X_{1} \arrow[ul] \arrow[r] \arrow[uu,dash,dashed] & N_{1} \arrow[uul] \arrow[r] \arrow[uu,dash,dashed] & X_{2} \arrow[uul] \arrow[r] \arrow[uu,dash,dashed] & \cdots \arrow[uul] \arrow[r] & N_{\frac{n-3}{2}} \arrow[uul] \arrow[uu,dash,dashed] \arrow[ur,shift left] & \\
R \arrow[ur] & & & & & & \\
\end{tikzcd}$$

Here, $A$ (resp. $B$) is a $1\times 1$ factorization $(x, x^{n-1}+y^2)$ (resp. $( x^{n-1}+y^2, y)$).
\begin{equation*}
\phi_j = \begin{pmatrix}
y& x^{j}\\
x^{n-j-2}& -y\\
\end{pmatrix}, \psi_j=
\begin{pmatrix}
xy& x^{j+1}\\
x^{n-j-1}& -xy\\
\end{pmatrix}
\end{equation*}
$M_j$ (resp. $N_j$) is a $2\times 2$ matrix factorization given by $(\phi_j,\psi_j)$ (resp. $(\psi_j,\phi_j)$).
\begin{equation}\label{eq:Dn2}
\xi_j = \begin{pmatrix}
y& x^{j}\\
x^{n-j-1}& -xy\\
\end{pmatrix}, \eta_j=
\begin{pmatrix}
xy& x^{j}\\
x^{n-j-1}& -y\\
\end{pmatrix}
\end{equation}
$X_j$ (resp. $Y_j$) is a $2\times 2$ matrix factorization given by $(\xi_j,\eta_j)$ (resp. $(\eta_j,\xi_j)$).
For even $n$,
$$\begin{tikzcd}[row sep=tiny]
A \arrow[ddddr] & & & & & & C_{+} \arrow[dl]\\
& Y_{1} \arrow[ul] \arrow[ddddl] \arrow[r] & M_{1} \arrow[dddl] \arrow[r] & Y_{2} \arrow[dddl] \arrow[r] & \cdots \arrow[dddl] \arrow[r] & Y_{\frac{n-2}{2}} \arrow[dddl] \arrow[dr] \arrow[ddddr]& \\
B \arrow[uu,dash,dashed] \arrow[ur] & & & & \cdots & & D_{+} \arrow[uu,dash,dashed] \arrow[ddl]\\
& & & & \cdots & & C_{-} \arrow[uul]\\
& X_{1} \arrow[uul] \arrow[r] \arrow[uuu,dash,dashed] & N_{1} \arrow[uuul] \arrow[r] \arrow[uuu,dash,dashed] & X_{2} \arrow[uuul] \arrow[r] \arrow[uuu,dash,dashed] & \cdots \arrow[uuul] \arrow[r] & X_{\frac{n-2}{2}} \arrow[uuul] \arrow[uuu,dash,dashed] \arrow[ur] \arrow[uuuur] & \\
R \arrow[ur] & & & & & & D_{-} \arrow[ul] \arrow[uu,dash,dashed]\\
\end{tikzcd}$$

To find the corresponding Lagrangians, we consider the dual singularity $D_n^T$ given by 
$$C_{2,n-1} = x^2 + xy^{n-1}.$$
First, Milnor fiber of $C_{2,n-1}$ is  given as LHS of Figure \ref{fig:dncell}, which is a $4(n-1)$-gon with edges identified as  
$\pm 3$ pattern. We label immersed generators of Seidel Lagrangian as in Figure \ref{fig:dncell}.
 Then the potential of Seidel Lagrangian is $W_\bL=y^{n-1} + xyz$, and if we restrict to the hypersurface $z-x=0$, we get $D_n$ singularity $y^{n-1}+yx^2$.

We now describe Lagrangians in Milnor fiber of  $D_n^T$ corresponding to indecomposable matrix factorizations of $D_n$. Note that our potential is $y^{n-1}+yx^2$, hence we have to switch $x$ and $y$ after calculating matrix factorization to match them with the above list.
\begin{lemma}
Lagrangians $L_A, L_{Y_i},L_{M_i}$(defined as  in Figure \ref{fig:dncell}) correspond to matrix factorizations $A,Y_i,M_i$.
By Lemma \ref{lem:lagtr},
orientation reversals of these Lagrangians correspond to $B,X_i, N_i$ respectively.
\end{lemma}
\begin{proof}
\begin{figure}[h]
\includegraphics[scale=0.5]{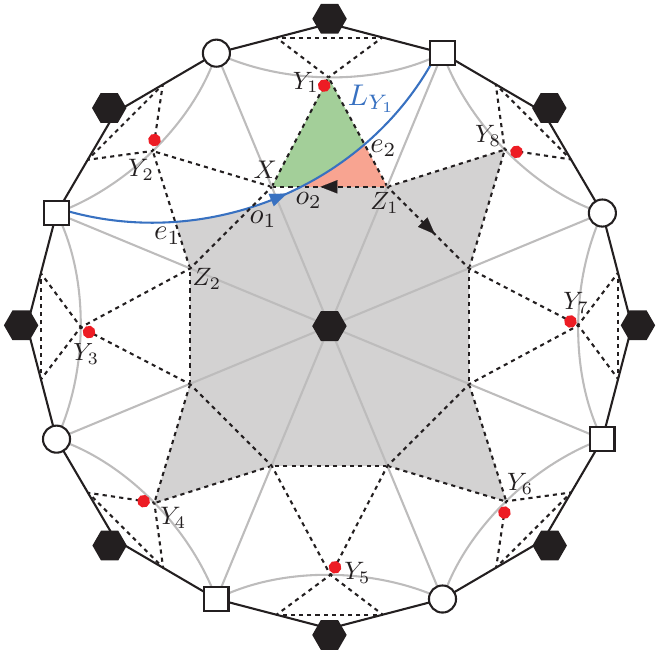}
\centering
\caption{Holomorphic polygons for $m^{0,\bb}_{1}$ in the case of $L_{Y_1}$ in $D_5^T$ Milnor fiber}
\label{fig:D5_disc}
\end{figure}
Potential $W_\bL$ comes from two polygons, minimal  $XYZ$ triangle, and $(n-1)$-gon with all $Y$-corners.
Note that Lagrangian $L_{A}$ cuts each of these two polygons into two parts. 
This gives the corresponding $1\times 1$ matrix factorization.
For the other cases, each of $L_{Y_i},L_{M_i}$ intersect  $\WT{\bL}$ at
two even and two odd points, and counting polygons with signs provides the desired $2\times2$ matrix factorizations.
Let us illustrate the case of $L_{Y_{1}}$ and leave the rest as an exercise. In this case, there are 8 polygons given by
$$o_{1} X Y_{2} e_{1}, o_{1} Y_{3} Y_{5} Y_{7} e_{2}, o_{2} Y_{2} e_{1}, o_{2} Z_{1} e_{2}, e_{1} Z_{2} o_{1}, e_{1} Y_{4} Y_{6} Y_{8} o_{2}, e_{2} Y_{1} o_{1}, e_{2} Y_{1} X o_{2}.$$

These polygons contribute to the following $m^{0,\bb}_{1}$ computations.
\begin{align*}
m^{0,\bb}_{1}(o_{1}) &=   xy \cdot e_{1} - y^{3} \cdot e_{2} &
m^{0,\bb}_{1}(e_{1}) &=   x \cdot o_{1} + y^{3} \cdot o_{2} \\
m^{0,\bb}_{1}(o_{2}) &=  y \cdot e_{1} + x \cdot e_{2} &
m^{0,\bb}_{1}(e_{2}) &=   -y \cdot o_{1} + xy \cdot o_{2}
\end{align*}
Note that some of the signs in $(CF(L_{Y_{1}},\mathbb{L}), -m_{1}^{0,\bb})$ are different from (\ref{eq:Dn2}), but by a simple change of basis, we can identify it with $(\eta_1,\xi_1)$.
\end{proof}

\begin{lemma}\label{lem:dnex}
All AR exact sequences for $D_n$ singularity can be realized as Lagrangian surgeries.
\end{lemma}

\begin{proof}
It is enough to check the following exact sequences (the other four will correspond to the orientation reversal of Lagrangians).
\begin{figure}[h]
\begin{subfigure}{0.45\textwidth}
\includegraphics[scale=0.45]{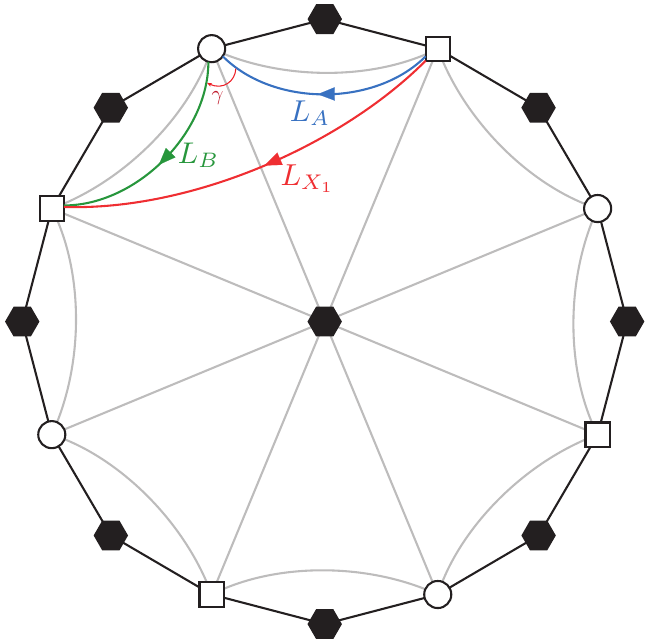}
\centering
\caption{Lagrangian surgery for \eqref{ar:d1}}
\end{subfigure}
\begin{subfigure}{0.45\textwidth}
\includegraphics[scale=0.45]{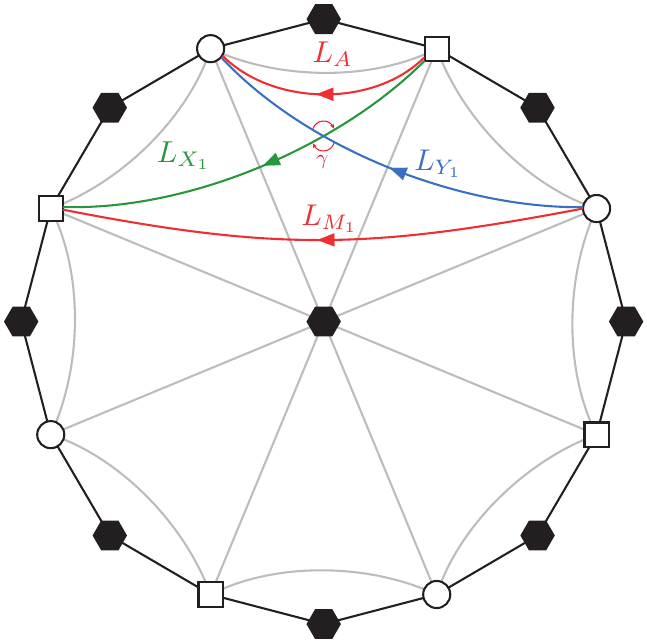}
\centering
\caption{Lagrangian surgery for \eqref{ar:d2}}
\end{subfigure}
\centering
\begin{subfigure}{0.45\textwidth}
\includegraphics[scale=0.45]{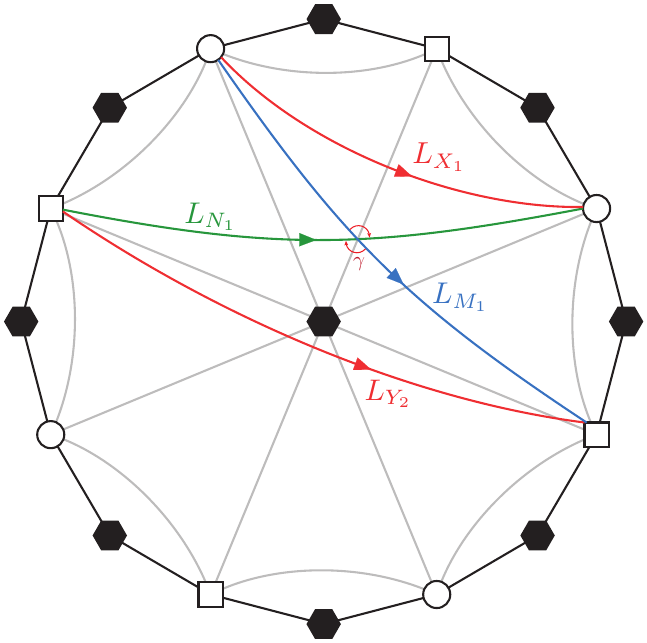}
\centering
\caption{Lagrangian surgery for \eqref{ar:d3}}
\end{subfigure}
\begin{subfigure}{0.45\textwidth}
\includegraphics[scale=0.45]{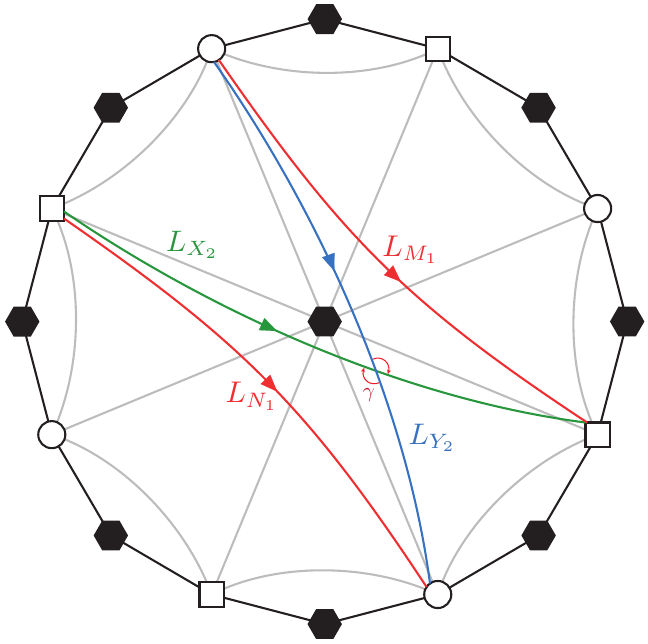}
\centering
\caption{Lagrangian surgery for \eqref{ar:d4}}
\end{subfigure}
\centering
\caption{Auslander--Reiten exact sequences for $D_5$ via Lagrangian surgery}
\label{fig:dnmor} 
\end{figure}

\centering{
\begin{subequations}
\noindent \begin{minipage}{0.45\textwidth}
\begin{align}
0 \to B \to Y_{1} \to A \to 0  \label{ar:d1}
\end{align}
\end{minipage}
\begin{minipage}{0.45\textwidth} 
\begin{align}
0 \to X_1 \to B \oplus  N_1 \to  Y_1 \to 0  \label{ar:d2}
\end{align}
\end{minipage}\\
\begin{minipage}{0.45\textwidth}
\begin{align}
0 \to N_i \to Y_i \oplus  X_i \to  M_i \to 0 \label{ar:d3}
\end{align}
\end{minipage}
\begin{minipage}{0.45\textwidth}
\begin{align}
0 \to X_i \to M_{i-1} \oplus N_i \to  Y_i \to 0 \label{ar:d4}
\end{align}
\end{minipage}
\end{subequations}}

Corresponding Lagrangian surgeries are illustrated in Figure \ref{fig:dnmor}, and
we obtain the above exact sequence, by applying Lemma \ref{lem:surc2}.
\end{proof}

\subsection{$\boldsymbol{E_6}$-case}
For $E_6$ singularity $x^3+y^4$, its AR quiver of indecomposable MF's is the following.
$$\begin{tikzcd}
& & B \arrow[dl,shift left] \arrow[r] & M_{1} \arrow[ddl] & \\
M_{2} \arrow[r,shift left] \arrow[out=240,in=300,dash,loop,swap,dashed]  & X \arrow[l,shift left] \arrow[ur,shift left] \arrow[dr,shift left] \arrow[out=235,in=305,dash,loop,swap,dashed] & & & R \arrow[ul] \\
& & A \arrow[ul,shift left] \arrow[r] \arrow[uu,dash,dashed] & N_{1} \arrow[ur] \arrow[uu,dash,dashed] \arrow[uul]& \\
\end{tikzcd}$$

Here, $M_j$ (resp. $N_j$) is a $2\times 2$ matrix factorization given by $(\phi_j,\psi_j)$ (resp. $(\psi_j,\phi_j)$) for $j=1,2$.
\begin{equation*} 
\phi_1 = \begin{pmatrix}
x& y\\
y^{3}& -x^2\\
\end{pmatrix}, \psi_1=
\begin{pmatrix}
x^2& y\\
y^3& -x\\
\end{pmatrix}
\end{equation*}
\begin{equation*}
\phi_2 = \begin{pmatrix}
x& y^2\\
y^{2}& -x^2\\
\end{pmatrix}, \psi_2=
\begin{pmatrix}
x^2& y^2\\
y^2& -x\\
\end{pmatrix}
\end{equation*}
$A$ (resp. $B$) is a $3\times 3$ matrix factorization given by $(\alpha,\beta)$ (resp. $(\beta,\alpha)$).
\begin{equation}\label{eq:mfe6b}
\alpha = \begin{pmatrix}
 y^{3}& x^{2}& xy^{2}\\
 xy&-y^{2}& x^{2}\\
 x^{2}&-xy&  -y^{3}\\
\end{pmatrix},  \beta= \begin{pmatrix}
 y& 0&  x\\
  x&  -y^{2}& 0\\
0&  x&  -y\\
\end{pmatrix}
\end{equation}
$X$ is a $4 \times 4$ matrix factorization given by $(\xi,\eta)$ where
\begin{equation*} 
\xi= \begin{pmatrix}
\phi_2& \epsilon_3\\
0 & \psi_2\\
\end{pmatrix}, \eta=
\begin{pmatrix}
\psi_2& \epsilon_4\\
0 &  \phi_2 \\
\end{pmatrix} 
\;\;\; \textrm{with} \;\; 
\epsilon_3= \begin{pmatrix}
0& y\\
-xy & 0\\
\end{pmatrix},
\epsilon_4= \begin{pmatrix}
0& xy\\
-y & 0\\
\end{pmatrix}.
\end{equation*}
\begin{figure}[h]
\begin{subfigure}{0.43\textwidth}
\includegraphics[scale=0.5]{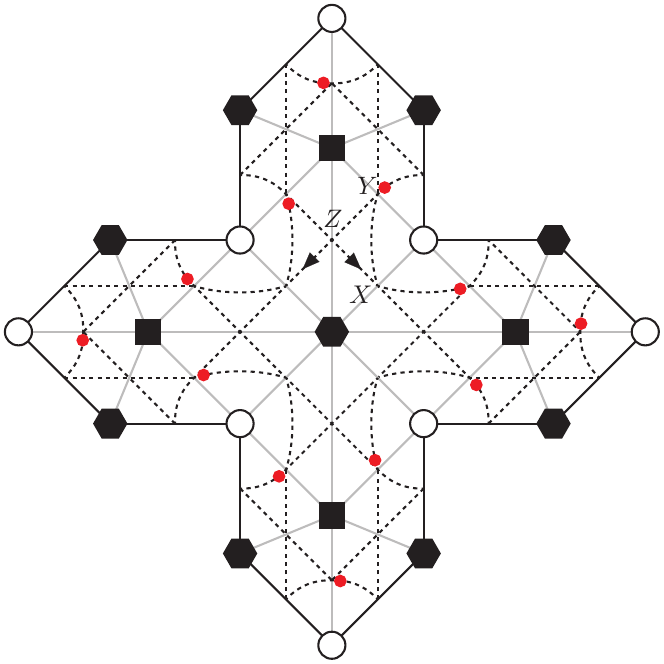}
\centering
\end{subfigure}
\begin{subfigure}{0.43\textwidth}
\includegraphics[scale=0.5]{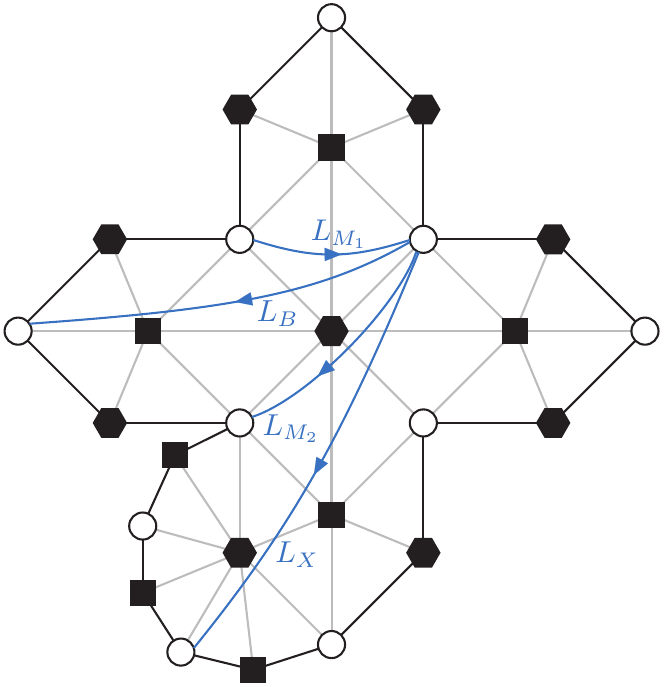}
\centering
\end{subfigure}
\centering
\caption{\label{fig:e6cell} Milnor fiber for $E_6$ and  Lagrangians corresponding to indecomposable MF's}
\end{figure}

Since $F_{3,4}^T = F_{3,4}$, we consider the Milnor fiber of $F_{3,4} = x^3+ y^{4}$.
First, it is given by $\Z/4$-copies of a hexagon glued as in Figure \ref{fig:e6cell}, whose boundary is identified with $(\pm 5)$ pattern.
$\Z/4$-action is the rotation at the center, and $\Z/3$-action is the simultaneous rotation in every hexagon. 

Lifts of Seidel Lagrangian are drawn in dotted lines, and $XXX$ and $XYZ$-triangles, and $YYYY$-quadrangle produce the potential  $W_\bL= x^3+y^4+xyz$
and after the restriction $z=0$, we get $E_6$ singularity $x^3+y^4$.

\begin{lemma}
Lagrangians $L_{M_i}, L_{B}, L_X$ in Figure \ref{fig:e6cell}  correspond to matrix factorizations $M_i,B,X$.
By Lemma \ref{lem:lagtr},
orientation reversals of these Lagrangians correspond to $N_i,A,X$ respectively.
\end{lemma}
\begin{proof}
We explain the case of $L_B$ only. There are 14 polygons;
$$o_{1} Y_{3} e_{1}, o_{1} X_{1} e_{2}, o_{2} Y_{6} Y_{8} e_{2}, o_{2} X_{7} e_{3}, o_{3} X_{10} e_{1}, o_{3} Y_{9} e_{3},$$
$$e_{1} Y_{5} Y_{7} Y_{1} o_{1}, e_{1} X_{6} Y_{4} o_{2}, e_{1} X_{6} X_{8} o_{3}, e_{2} Y_{2} Y_{4} o_{2}, e_{2} Y_{2} X_{8} o_{3}, e_{3} X_{9} Y_{7} Y_{1} o_{1}, e_{3} X_{9} X_{11} o_{2}, e_{3} Y^{\prime}_{4} Y^{\prime}_{6} Y^{\prime}_{8} o_{3}.$$
Last polygon is a part of polygon for $y^4$, cut out by $L_B$ along the edge $e_{3}o_3$.
It lies outside of the domain in Figure \ref{fig:E6_disc}, but can be obtained by a group action of the vertex without prime.
\begin{figure}[h]
\includegraphics[scale=0.6]{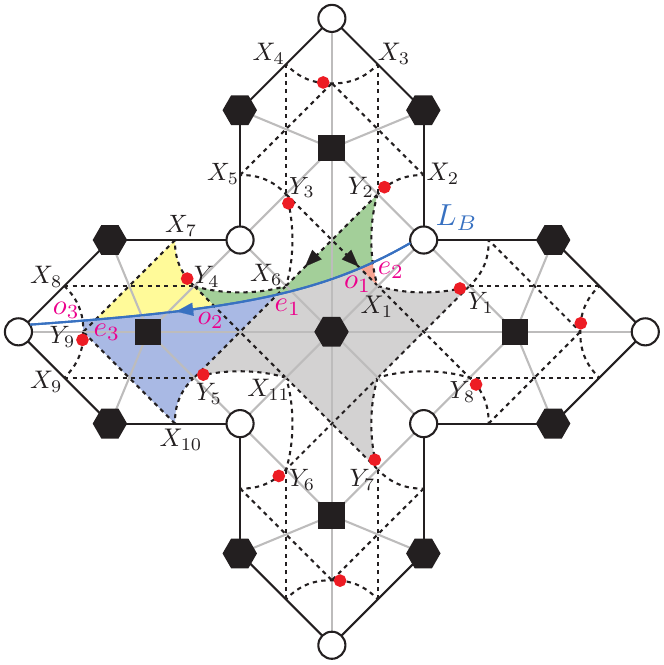}
\centering
\caption{
Holomorphic polygons for $m^{0,\bb}_{1}$ in the case of $L_B$ in $E_6$ Milnor fiber}
\label{fig:E6_disc}
\end{figure}
The above polygons contribute to $m^{0,\bb}_{1}$ as follows.
\begin{align*}
m^{0,\bb}_{1}(o_{1}) &=   -y \cdot e_{1} + x \cdot e_{2} &m^{0,\bb}_{1}(e_{1}) &=   -y^{3} \cdot o_{1} + xy \cdot o_{2} + x^{2} \cdot o_{3}\\
m^{0,\bb}_{1}(o_{2}) &=  y^{2} \cdot e_{2} - x \cdot e_{3} &m^{0,\bb}_{1}(e_{2}) &=   x^{2} \cdot o_{1} + y^{2} \cdot o_{2} + xy \cdot o_{3}\\
m^{0,\bb}_{1}(o_{3}) &=   x \cdot e_{1} + y \cdot e_{3} &m^{0,\bb}_{1}(e_{3}) &=   xy^{2} \cdot o_{1} - x^{2} \cdot o_{2} + y^{3} \cdot o_{3}\\
\end{align*}
These give the matrices in \eqref{eq:mfe6b} (up to change of a basis).
\end{proof}

\begin{lemma}
All AR exact sequences for $E_6$ singularity can be realized as Lagrangian surgeries.
\end{lemma}
\begin{proof}
We need to check following exact sequences (the other four are given by orientation reversals).
\begin{align}
0 \to M_1\to A \to N_1 \to 0  \label{ar:e61}\\
0 \to B \to X \oplus  M_1 \to A \to 0 \label{ar:e62}\\
0 \to X  \to M_2 \oplus A \oplus B \to  X \to 0 \label{ar:e63} \\
0 \to M_2 \to X \to M_2 \to 0 \label{ar:e64}
\end{align}

We illustrated the corresponding Lagrangian surgery in Figure \ref{fig:e6mor}.
\begin{figure}[h]
\begin{subfigure}{0.5\textwidth}
\includegraphics[scale=0.45]{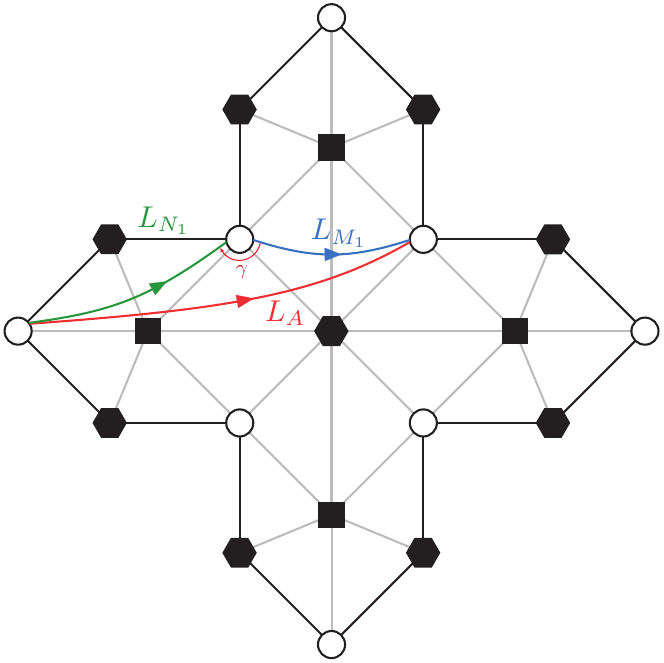}
\centering
\caption{Lagrangian surgery for the exact sequence \eqref{ar:e61}}
\end{subfigure}
\begin{subfigure}{0.33\textwidth}
\includegraphics[scale=0.45]{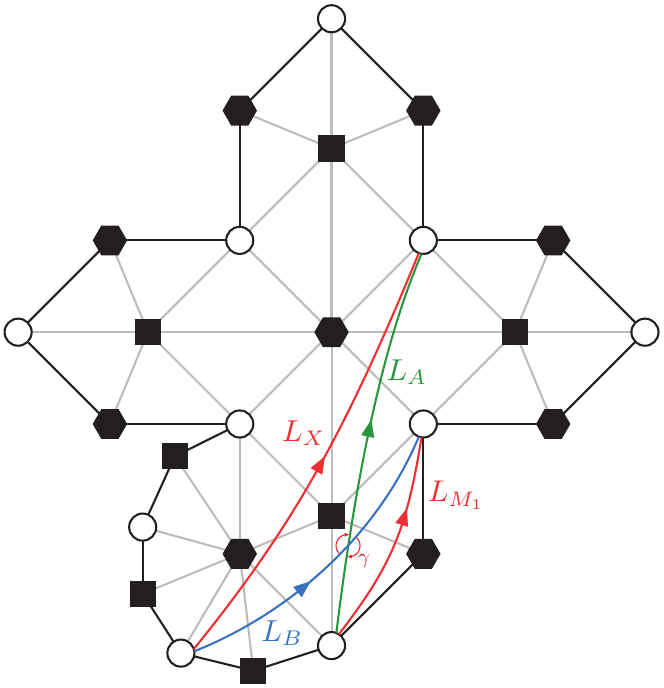}
\centering
\caption{for the exact sequence \eqref{ar:e62}}
\end{subfigure}
\newline
\begin{subfigure}{0.5\textwidth}
\includegraphics[scale=0.45]{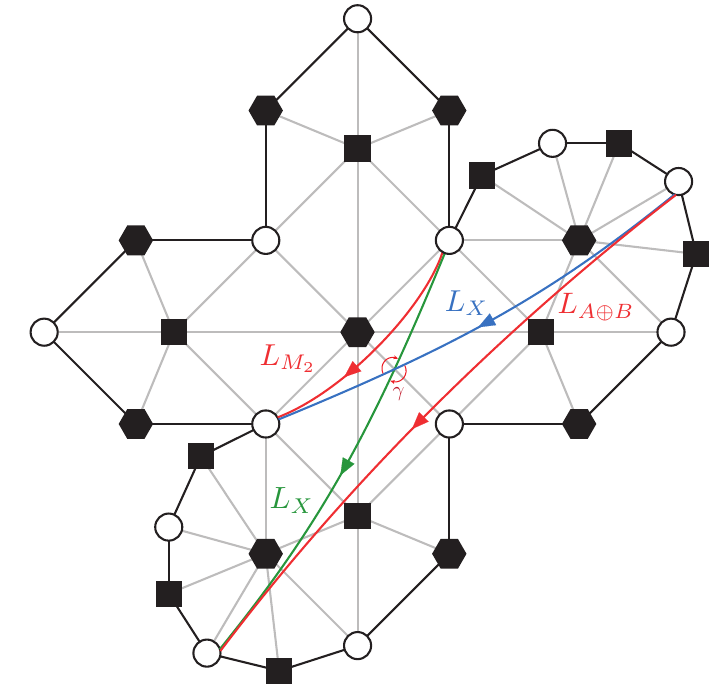}
\centering
\caption{Lagrangian surgery for the exact sequence \eqref{ar:e63}}
\end{subfigure}
\begin{subfigure}{0.33\textwidth}
\includegraphics[scale=0.45]{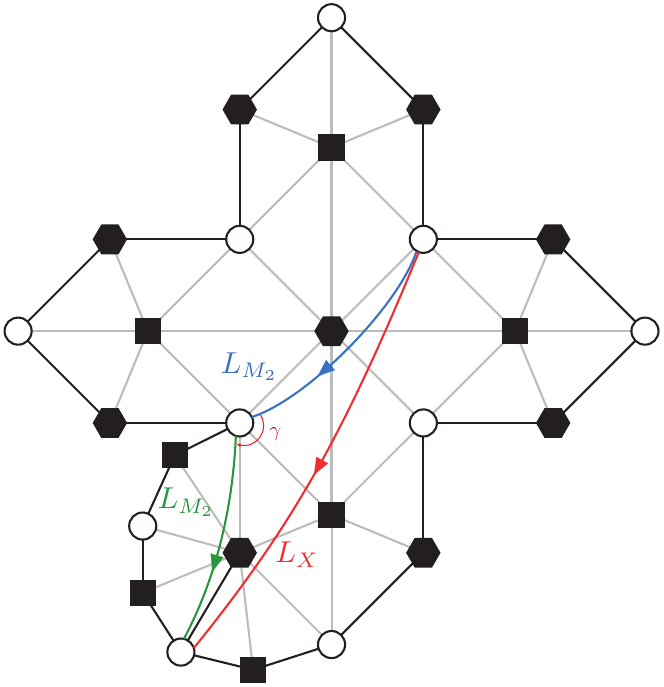}
\caption{for the exact sequence \eqref{ar:e64}}
\centering
\end{subfigure}
\centering
\caption{Auslander--Reiten exact sequences for $E_6$ via Lagrangian surgery  }
\label{fig:e6mor}
\end{figure}
Note that in the case of  \eqref{ar:e63}, a connected Lagrangian submanifold, denoted as  $L_{A \oplus B}$,
corresponds to the direct sum $A \oplus B$. This does not directly follow from the picture: $L_{A \oplus B}$
can be considered as a mapping cone of a morphism $\gamma$ from $L_A$ to $L_B$, but
$\gamma$ can be shown to be zero up to homotopy.
More precisely, matrix factorization from the Lagrangian $L_{A \oplus B}$ can be computed as 
\begin{equation*}
-\begin{pmatrix}
-y^{3}& x^{2}& xy^{2} & y^{2} &0 &0\\
xy&  y^{2}& -x^{2}& 0& y^{2}& 0\\
x^{2}&  xy& y^{3}& 0& 0& y^{2}\\
0& 0& 0& -y& 0& x\\
0& 0& 0& x& y^{2}& 0\\
0& 0& 0& 0& -x& y\\
\end{pmatrix} \cdot
-\begin{pmatrix}
-y& 0& x& -y^{2} &0 &0\\
x& y^{2}& 0& 0& -y^{2}& 0\\
0& -x& y& 0& 0& -y^{2}\\
0& 0& 0& -y^{3}& x^{2}& xy^{2}\\
0& 0& 0& xy& y^{2}& -x^{2}\\
0& 0& 0& x^{2}& xy& y^{3}\\
\end{pmatrix}
\end{equation*}
after basis change. It is a (minus of) mapping cone of $\gamma: B[1] \to A$ given by $(y^2 \cdot id, y^2 \cdot id)$.
$$\begin{tikzcd}[row sep=1.2cm, column sep=2cm, ampersand replacement=\&]
S^{3} \arrow{r}{-\phi} \arrow{d}{-y^{2} \cdot id} \& S^{3} \arrow{r}{-\psi} \arrow{d}{-y^{2} \cdot id} \arrow{dl}{h_{1}} \& S^{3} \arrow{d}{-y^{2} \cdot id} \arrow{dl}{h_{2}} \\
S^{3} \arrow{r}[swap]{-\phi} \& S^{3} \arrow{r}[swap]{-\psi} \& S^{3} \\
\end{tikzcd}$$
where $\phi = \begin{pmatrix}
-y^{3}& x^{2}& xy^{2}\\
xy& y^{2}& -x^{2}\\
x^{2}& xy& y^{3}\\
\end{pmatrix}$, $\psi = \begin{pmatrix}
-y& 0& x\\
x& y^{2}& 0\\
0& -x& y\\
\end{pmatrix}.$

We can find explicit null-homotopy $h = (h_{1}, h_{2})$ of $(y^2 \cdot id, y^2 \cdot id) \in \mbox{Hom} (A,A)$.

\begin{equation*}
\begin{pmatrix}
-y^{3}& x^{2}& xy^{2}\\
xy& y^{2}& -x^{2}\\
x^{2}& xy& y^{3}\\
\end{pmatrix}
\begin{pmatrix}
0& 0& 0\\
0& 1& 0\\
0& 0& 0\\
\end{pmatrix}
+
\begin{pmatrix}
-y& 0& x\\
0& 0& 0\\
0& 0& y\\
\end{pmatrix}
\begin{pmatrix}
-y& 0& x\\
x& y^{2}& 0\\
0& -x& y\\
\end{pmatrix} 
=
\begin{pmatrix}
y^{2}& 0& 0\\
0& y^{2}& 0\\
0& 0& y^{2}\\
\end{pmatrix} 
\end{equation*}

\begin{equation*}
\begin{pmatrix}
-y& 0& x\\
x& y^{2}& 0\\
0& -x& y\\
\end{pmatrix} 
\begin{pmatrix}
-y& 0& x\\
0& 0& 0\\
0& 0& y\\
\end{pmatrix}
+
\begin{pmatrix}
0& 0& 0\\
0& 1& 0\\
0& 0& 0\\
\end{pmatrix}
\begin{pmatrix}
-y^{3}& x^{2}& xy^{2}\\
xy& y^{2}& -x^{2}\\
x^{2}& xy& y^{3}\\
\end{pmatrix}
=
\begin{pmatrix}
y^{2}& 0& 0\\
0& y^{2}& 0\\
0& 0& y^{2}\\
\end{pmatrix} 
\end{equation*}
Hence, this mapping cone is isomorphic to the direct sum $A \oplus B$.

\end{proof}

\begin{remark} \label{rmk:inv}
 We can prove that if $L$ maps to a matrix factorization $P$, then its involution image maps to $P$ or $P[1]$ in ADE case. 
 \end{remark}

\subsection{$\boldsymbol{E_7}$-case}
For $E_7$ singularity $x^3+xy^3$, its AR quiver of indecomposable MF's is the following.
$$\begin{tikzcd}[row sep=small]
& & & C \arrow[dddd] \arrow[r,dash,dashed] & D \arrow[dd] & & & \\
& & & & & & & \\
A \arrow[r] & M_{2} \arrow[ddl] \arrow[r] & Y_{2} \arrow[ddl] \arrow[rr] & & Y_{3} \arrow[uul] \arrow[ddll] \arrow[r] & Y_{1} \arrow[ddll] \arrow[r] & M_{1} \arrow[ddl] & \\
& & & & & & & R \arrow[ul] \\
B \arrow[r] \arrow[uu,dash,dashed] & N_{2} \arrow[uul] \arrow[r] \arrow[uu,dash,dashed] & X_{2} \arrow[uul] \arrow[r] \arrow[uu,dash,dashed] & X_{3} \arrow[uuuur] \arrow[uul] \arrow[rr] \arrow[uur,dash,dashed] & & X_{1} \arrow[uul] \arrow[r] \arrow[uu,dash,dashed] & N_{1} \arrow[uul] \arrow[ur] \arrow[uu,dash,dashed] & \\
\end{tikzcd}$$

Here, $A$ (resp. $B$) is a $1\times 1$ factorization $(x, x^2+y^3)$ (resp. $( x^2+y^3, x)$).
$C$ (resp. $D$) is a $2\times 2$ matrix factorization given by $(\gamma,\delta)$ (resp. $(\delta,\gamma)$).
\begin{equation*} 
\gamma= \begin{pmatrix}
x^2& xy\\
xy^2&  -x^2\\
\end{pmatrix}, \delta=
\begin{pmatrix}
x& y\\
y^2& -x\\
\end{pmatrix}
\end{equation*}
$M_j$ (resp. $N_j$) is a $2\times 2$ matrix factorization given by $(\phi_j,\psi_j)$ (resp. $(\psi_j,\phi_j)$) for $j=1,2$.
\begin{equation*} 
\phi_1 = \begin{pmatrix}
x& y \\
xy^2& -x^2\\
\end{pmatrix}, \psi_1=
\begin{pmatrix}
x^2& y\\
xy^2& -x\\
\end{pmatrix}
\end{equation*}
\begin{equation*} 
\phi_2 = \begin{pmatrix}
x& y^2\\
xy& -x^2\\
\end{pmatrix}, \psi_2=
\begin{pmatrix}
x^2& y^2\\
xy& -x\\
\end{pmatrix}
\end{equation*}
$X_j$ (resp. $Y_j$) is a $3\times 3$ matrix factorization given by $(\xi_j,\eta_j)$ (resp. $(\eta_j,\xi_j)$) for $j=1,2$.
$$\xi_1 = \begin{pmatrix}
 xy^2& -x^{2}& -x^2y\\
 xy&y^{2}& -x^{2}\\
 x^{2}&xy&  xy^{2}\\
\end{pmatrix},  \eta_1= \begin{pmatrix}
 y& 0&  x\\
  -x&  xy& 0\\
0&  -x&  y\\
\end{pmatrix}
$$
$$\xi_2 = \begin{pmatrix}
 x^2& -y^{2}& -xy\\
 xy& x& -y^{2}\\
 xy^2& xy&  x^2\\
\end{pmatrix},  \eta_2= \begin{pmatrix}
 x& 0&  y\\
  -xy& x^{2}& 0\\
0&  -xy&  x\\
\end{pmatrix}
$$
Finally, $X_3$ (resp. $Y_3$) is a $4\times 4$ matrix factorization given by $(\xi_3,\eta_3)$ (resp. $(\eta_3,\xi_3)$).
\begin{equation*} 
\xi_3= \begin{pmatrix}
\gamma& \epsilon\\
0 & \delta\\
\end{pmatrix}, \eta_3=
\begin{pmatrix}
\delta& -\epsilon\\
0 &  \gamma \\
\end{pmatrix} 
\;\;\; \textrm{with} \;\; 
\epsilon= \begin{pmatrix}
y^{2}& 0\\
0 & y^{2}\\
\end{pmatrix}.
\end{equation*}
\begin{remark}
The matrix $\epsilon$ is defined as $\begin{pmatrix}
y& 0\\
0 & y\\
\end{pmatrix}$ in \cite{Yo} and it seems to be a typo.
\end{remark}

To find the corresponding symplectic geometry, we consider the Milnor fiber of $C_{3,3}= x^3 +xy^3$.

\begin{figure}[h]
\includegraphics[scale=0.50]{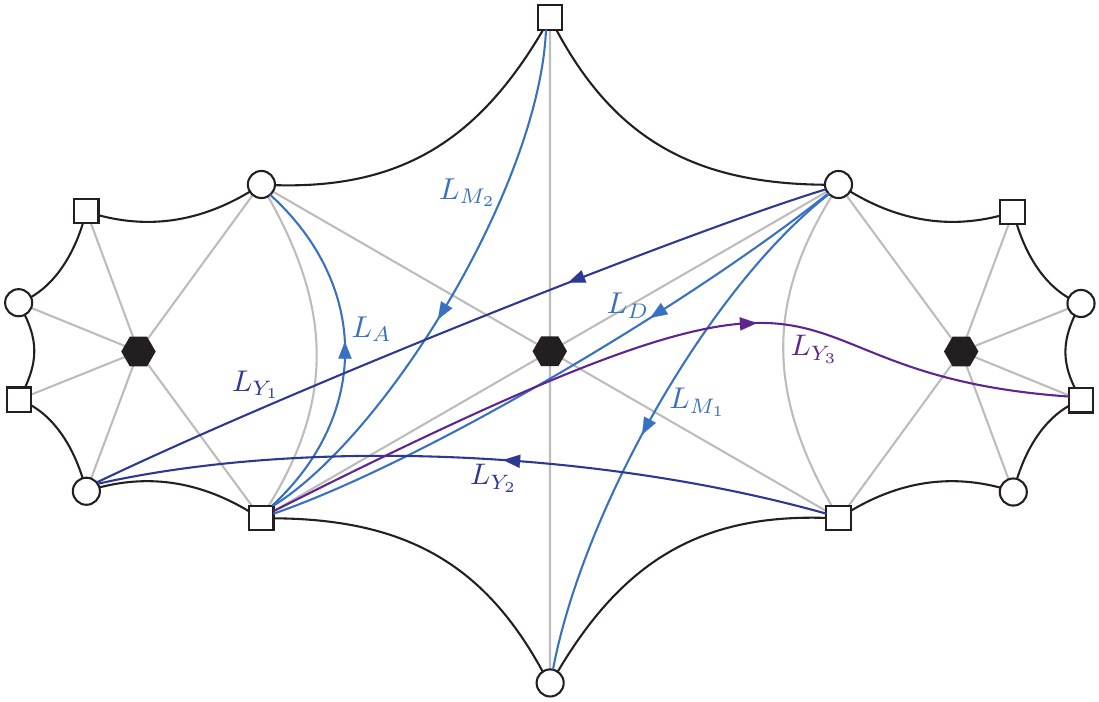}
\centering
\caption{\label{fig:e7cell} Lagrangians for indecomposable MF's for $E_7$ case}
\end{figure}

Recall that the fundamental domain is given by 3 copies of hexagons as we have seen in Figure \ref{fig:C33} (with $\Z/3$-fixed point 
at the center of each hexagon). The potential is $W_\bL= y^3+xyz$
and after we set $z=x^2$, we get $E_7$ singularity $y^3+yx^3$.
We remark that two punctures behave differently. Namely, puncture $A$ (drawn as a rectangle), and puncture $C$ (drawn as a circle)
have different monodromies  in Proposition \ref{prop:homo}, and this results in the relation $z=x^2$ via Kodaira--Spencer map. Because of a similar reason in $D_n$ cases, we need to switch $x$ and $y$ for complete correspondence.

\begin{lemma}
Lagrangians $L_{A}, L_{D}, L_{M_i}, L_{Y_i}$ in Figure \ref{fig:e7cell}  correspond to
indecomposable matrix factorizations $A,D,M_i,Y_i$ respectively.
By Lemma \ref{lem:lagtr},
orientation reversals of these Lagrangians correspond to $B,C,N_i,X_i$ respectively.
\end{lemma}
\begin{proof}
We explain the case of $Y_{3}$ only. There are 20 polygons;
$$o_{1} Y_{10} e_{1}, o_{1} Z_{5} e_{2}, o_{2} X_{4} e_{1}, o_{2} Y_{7} e_{2}, o_{3} Z_{3} e_{1}, o_{3} Y_{1} Y_{3} e_{3}, o_{3} Z_{3} Y_{6} e_{4}, o_{4} Y_{4} e_{1}, o_{4} Y_{4} X_{3} e_{3}, o_{4} Y_{4} Y_{6} e_{4},$$
$$e_{1} Y_{6} Y_{8} o_{1}, e_{1} Y_{6} Z_{4} o_{2}, e_{2} X_{5} Y_{8} o_{1}, e_{2} Y_{9} Y_{11} o_{2}, e_{3} Y_{5} o_{3}, e_{3} Z_{2} o_{4}, e_{4} Y_{8} o_{1}, e_{4} Z_{4} o_{2}, e_{4} X_{1} o_{3}, e_{4} Y_{2} o_{4}.$$
\begin{figure}[h]
\includegraphics[scale=0.5]{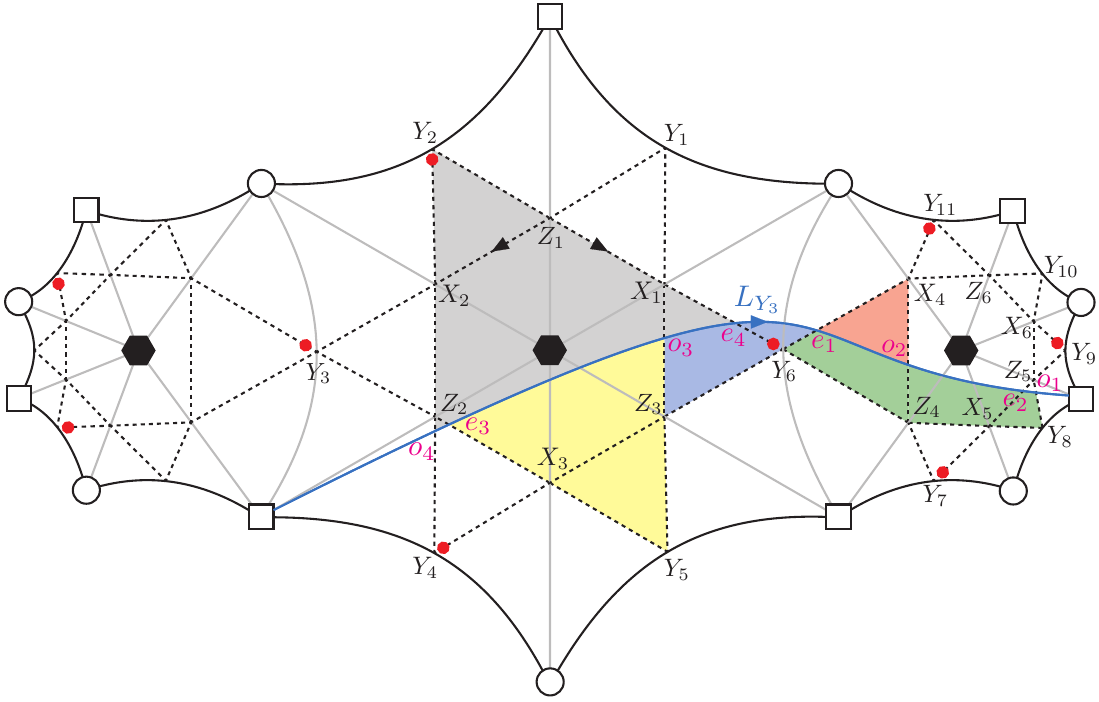}
\centering
\caption{Holomorphic polygons for $m^{0,\bb}_{1}$ calculation}
\label{fig:E7_disc}
\end{figure}
We have the following $m^{0,\bb}_{1}$ calculations.
\begin{align*}
m^{0,\bb}_{1}(o_{1}) &= y \cdot e_{1} + x^{2} \cdot e_{2} &m^{0,\bb}_{1}(e_{1}) &= y^{2} \cdot o_{1} + x^{2}y \cdot o_{2}\\
m^{0,\bb}_{1}(o_{2}) &= x \cdot e_{1} - y \cdot e_{2} &m^{0,\bb}_{1}(e_{2}) &= xy \cdot o_{1} - y^{2} \cdot o_{2}\\
m^{0,\bb}_{1}(o_{3}) &= x^{2} \cdot e_{1} + y^{2} \cdot e_{3} + x^{2}y \cdot e_{4} &m^{0,\bb}_{1}(e_{3}) &= y \cdot o_{3} + x^{2} \cdot o_{4}\\
m^{0,\bb}_{1}(o_{4}) &= -y \cdot e_{1} + xy \cdot e_{3} - y^{2} \cdot e_{4} &m^{0,\bb}_{1}(e_{4}) &= -y \cdot o_{1} - x^{2} \cdot o_{2} + x \cdot o_{3} - y \cdot o_{4}\\
\end{align*}
\end{proof}

\begin{lemma}
The following AR exact sequences (and their AR translation) for $E_7$ singularity can be realized as Lagrangian surgeries.

\centering{
\noindent \begin{subequations}
\begin{minipage}{0.45\textwidth}
\begin{align}
0 \to A \to M_2 \to B \to 0  \label{ar:e71}
\end{align}
\end{minipage}
\begin{minipage}{0.45\textwidth} 
\begin{align}
0 \to D \to Y_3 \to C \to 0 \label{ar:e72}
\end{align}
\end{minipage}\\
\begin{minipage}{0.45\textwidth}
\begin{align}
0 \to M_1 \to X_1 \to N_1 \to 0 \label{ar:e73}
\end{align}
\end{minipage}
\begin{minipage}{0.45\textwidth}
\begin{align}
0 \to M_2\to B \oplus Y_2 \to N_2 \to 0 \label{ar:e74}
\end{align}
\end{minipage}\\
\begin{minipage}{0.45\textwidth}
\begin{align}
0 \to Y_1 \to M_1 \oplus X_3 \to X_1 \to 0 \label{ar:e75}
\end{align}
\end{minipage}
\begin{minipage}{0.45\textwidth}
\begin{align}
0 \to Y_2 \to N_2 \oplus Y_3 \to X_2 \to 0 \label{ar:e76}
\end{align}
\end{minipage}\\
\begin{minipage}{0.45\textwidth}
\begin{align}
0 \to Y_3 \to X_2 \oplus C \oplus Y_1 \to X_3 \to 0 \label{ar:e77}
\end{align}
\end{minipage}
\end{subequations}}
\end{lemma}

\begin{proof}
In the case of \eqref{ar:e77}, we need to show that $L_{Y_{1} \oplus X_{2}}$ corresponds to the direct sum $Y_{1} \oplus X_{2}$. $L_{Y_{1} \oplus X_{2}}$ maps to matrix factorization  given by
\begin{small}
\begin{equation*} 
-\begin{pmatrix}
x& 0& y& 0& 0&0\\
-y& 0& x^{2}& 0& 0& 0\\
0& y& -x& -x^{2}& -x& -y\\
0& 0& 0& x^{3}+y^{2}& x^{2}& xy\\
0& 0& 0& 0& -y& x^{2}\\
0& 0& 0& 0& 0& x^{3}+y^{2}\\
\end{pmatrix} \cdot
-\begin{pmatrix}
x^{2}y& -y^{2}& 0& 0 &0 &0\\
xy& x^{2}& x^{3}+y^{2}& x^{2}& -xy& y\\
y^{2}& xy& 0& 0& 0& 0\\
0& 0& 0& y& x^{2}& -x\\
0& 0& 0& 0& -x^{3}-y^{2}& x^{2}\\
0& 0& 0& 0& 0& y\\
\end{pmatrix}
\end{equation*}
\end{small}
As a mapping cone, the corresponding morphism is homotopic to zero;

\begin{equation*} 
\begin{pmatrix}
x& 0& y\\
-y& 0& x^{2}\\
0& y& -x\\
\end{pmatrix} 
\begin{pmatrix}
-y& 0& -x\\
0& 0& -1\\
x& 1& 0\\
\end{pmatrix}
+
\begin{pmatrix}
0& 1& 0\\
-1& 0& 0\\
0& 0& 0\\
\end{pmatrix}
\begin{pmatrix}
x^{3}+y^{2}& x^{2}& xy\\
0& -y& x^{2}\\
0& 0& x^{3}+y^{2}\\
\end{pmatrix}
=
\begin{pmatrix}
0& 0& 0\\
0& 0& 0\\
-x^{2}& -x& -y\\
\end{pmatrix} 
\end{equation*}

\begin{equation*} 
\begin{pmatrix}
x^{2}y& -y^{2}& 0\\
xy& x^{2}& x^{3}+y^{2}\\
y^{2}& xy& 0\\
\end{pmatrix}
\begin{pmatrix}
0& 1& 0\\
-1& 0& 0\\
0& 0& 0\\
\end{pmatrix}
+
\begin{pmatrix}
-y& 0& -x\\
0& 0& -1\\
x& 1& 0\\
\end{pmatrix}
\begin{pmatrix}
y& x^{2}& -x\\
0& -x^{3}-y^{2}& x^{2}\\
0& 0& y\\
\end{pmatrix} 
=
\begin{pmatrix}
0& 0& 0\\
-x^{2}& xy& -y\\
0& 0& 0\\
\end{pmatrix} 
\end{equation*}
\end{proof}

\begin{figure}[H]
\begin{subfigure}{0.46\textwidth}
\includegraphics[scale=0.35]{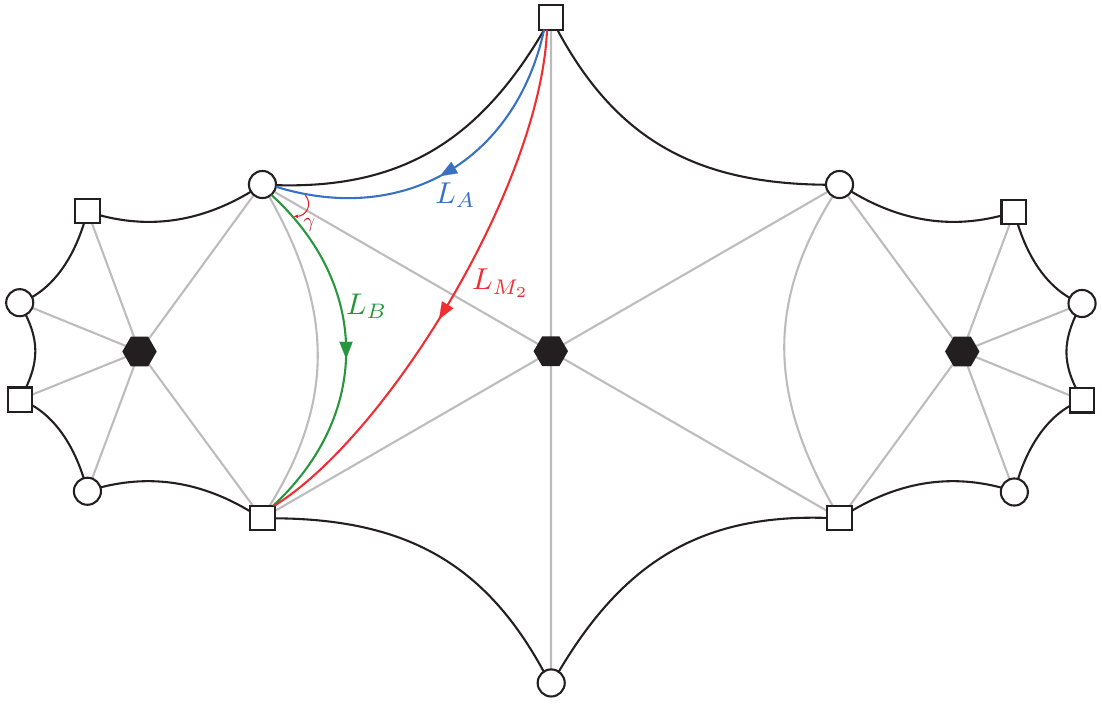}
\centering
\caption{Lagrangian surgery for \eqref{ar:e71}}
\end{subfigure}
\begin{subfigure}{0.46\textwidth}
\includegraphics[scale=0.35]{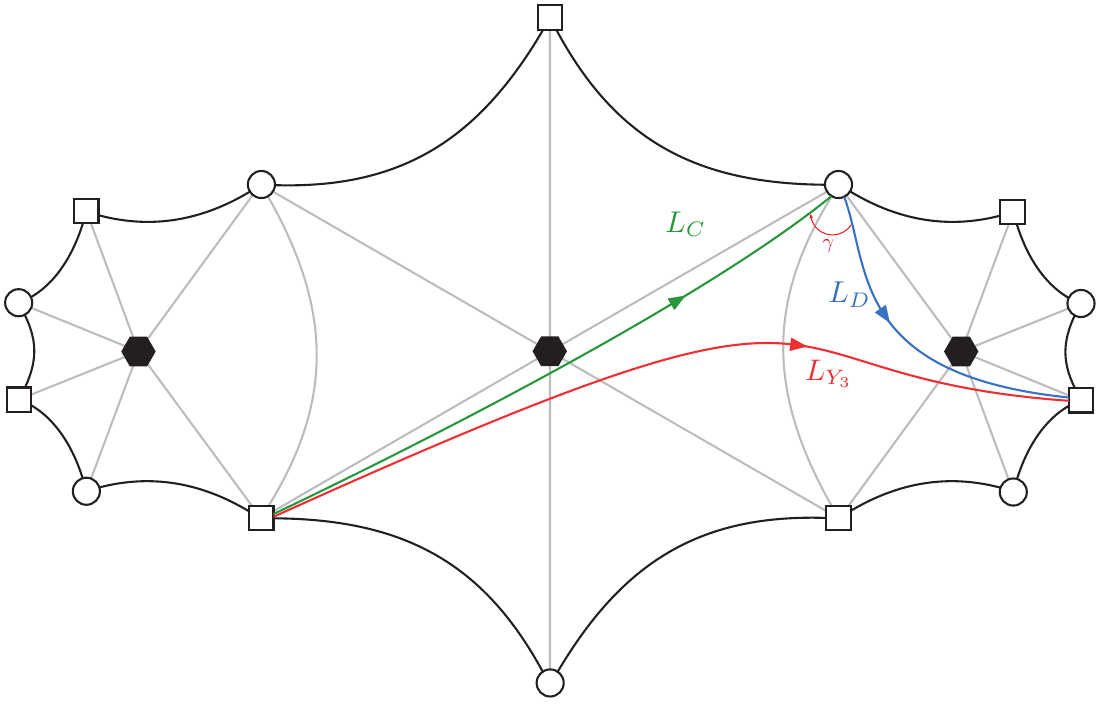}
\centering
\caption{Lagrangian surgery for \eqref{ar:e72}}
\end{subfigure}
\centering
\begin{subfigure}{0.46\textwidth}
\includegraphics[scale=0.35]{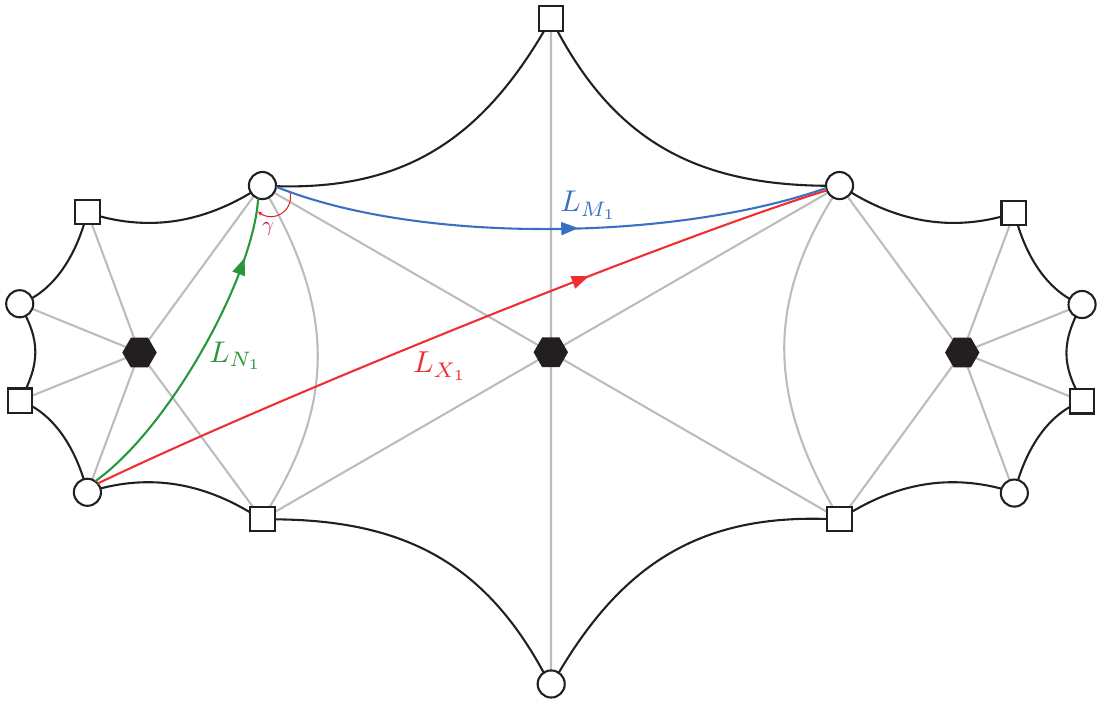}
\centering
\caption{Lagrangian surgery for \eqref{ar:e73}}
\end{subfigure}
\begin{subfigure}{0.46\textwidth}
\includegraphics[scale=0.35]{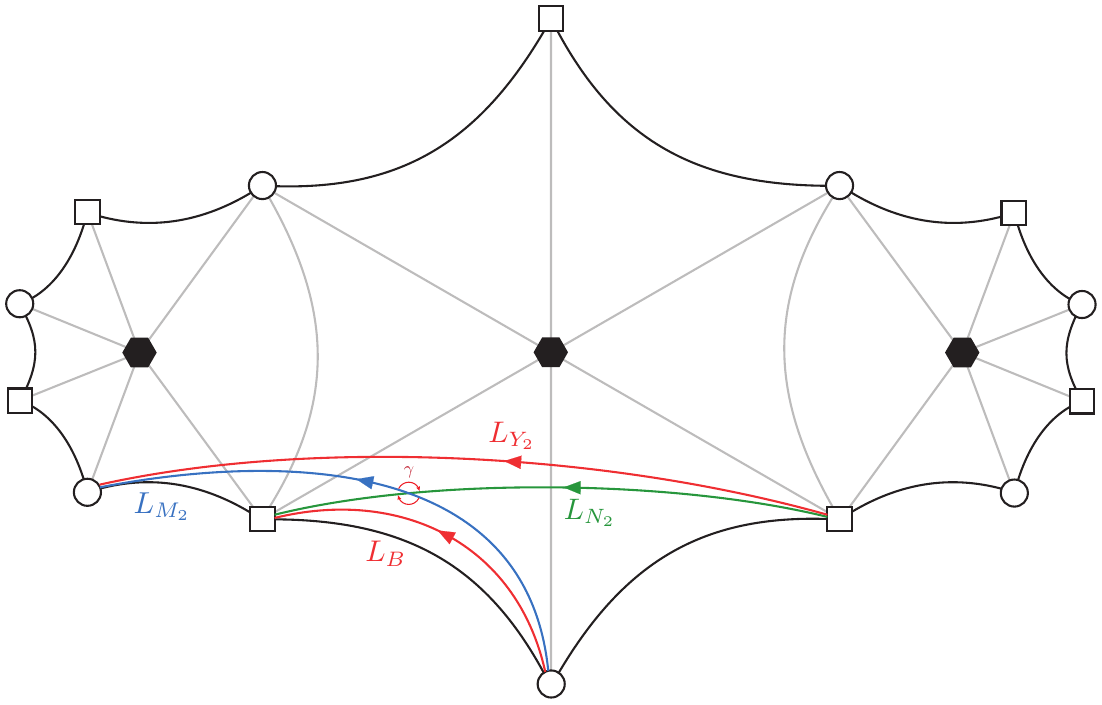}
\centering
\caption{Lagrangian surgery for \eqref{ar:e74}}
\end{subfigure}
\centering
\caption{Auslander--Reiten exact sequences for $E_7$ via Lagrangian surgery}
\end{figure}

\begin{figure}[h]
\begin{subfigure}{0.46\textwidth}
\includegraphics[scale=0.35]{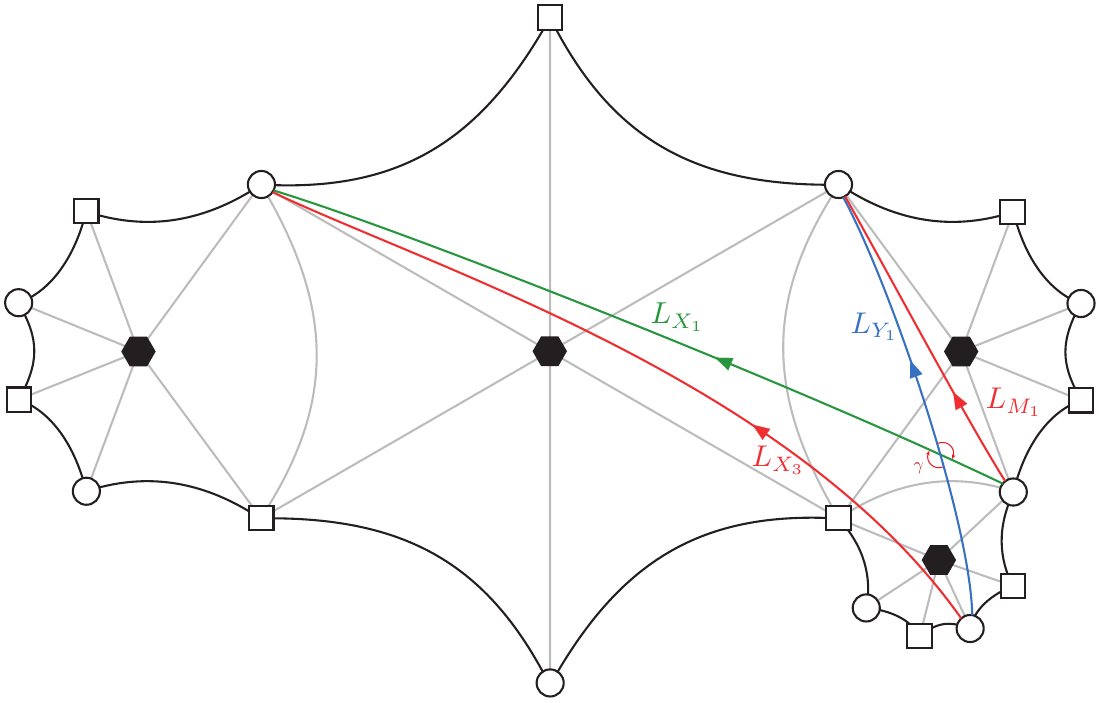}
\centering
\caption{Lagrangian surgery for \eqref{ar:e75}}
\end{subfigure}
\begin{subfigure}{0.46\textwidth}
\includegraphics[scale=0.35]{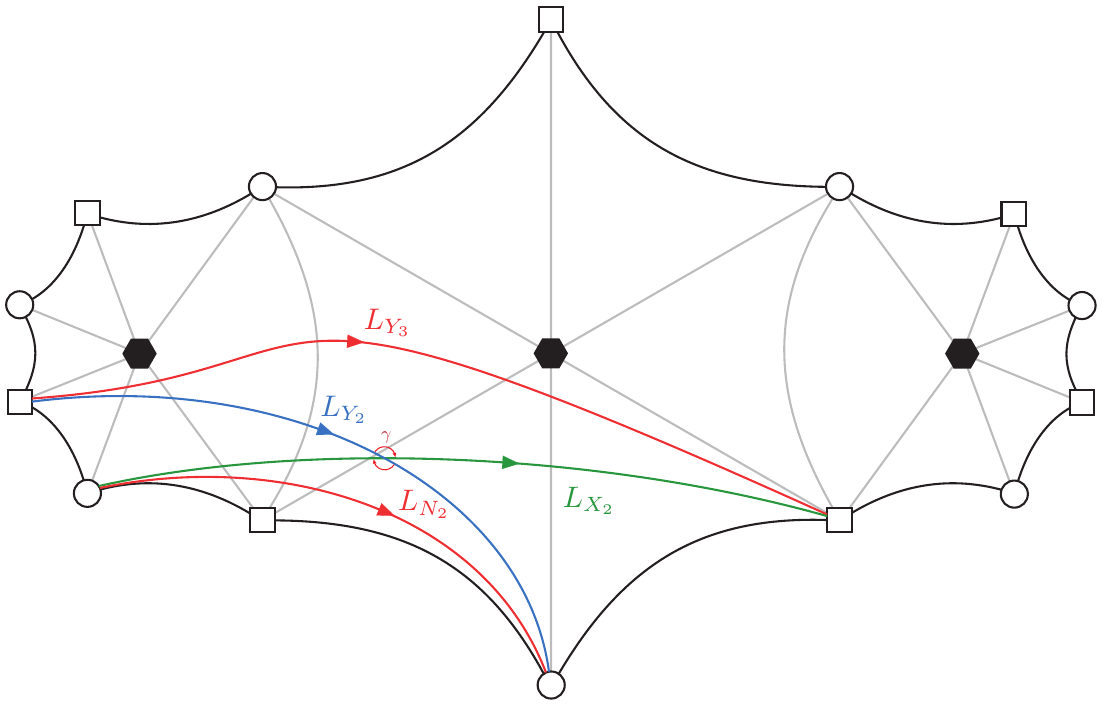}
\centering
\caption{Lagrangian surgery for \eqref{ar:e76}}
\end{subfigure}
\centering
\begin{subfigure}{0.5\textwidth}
\includegraphics[scale=0.35]{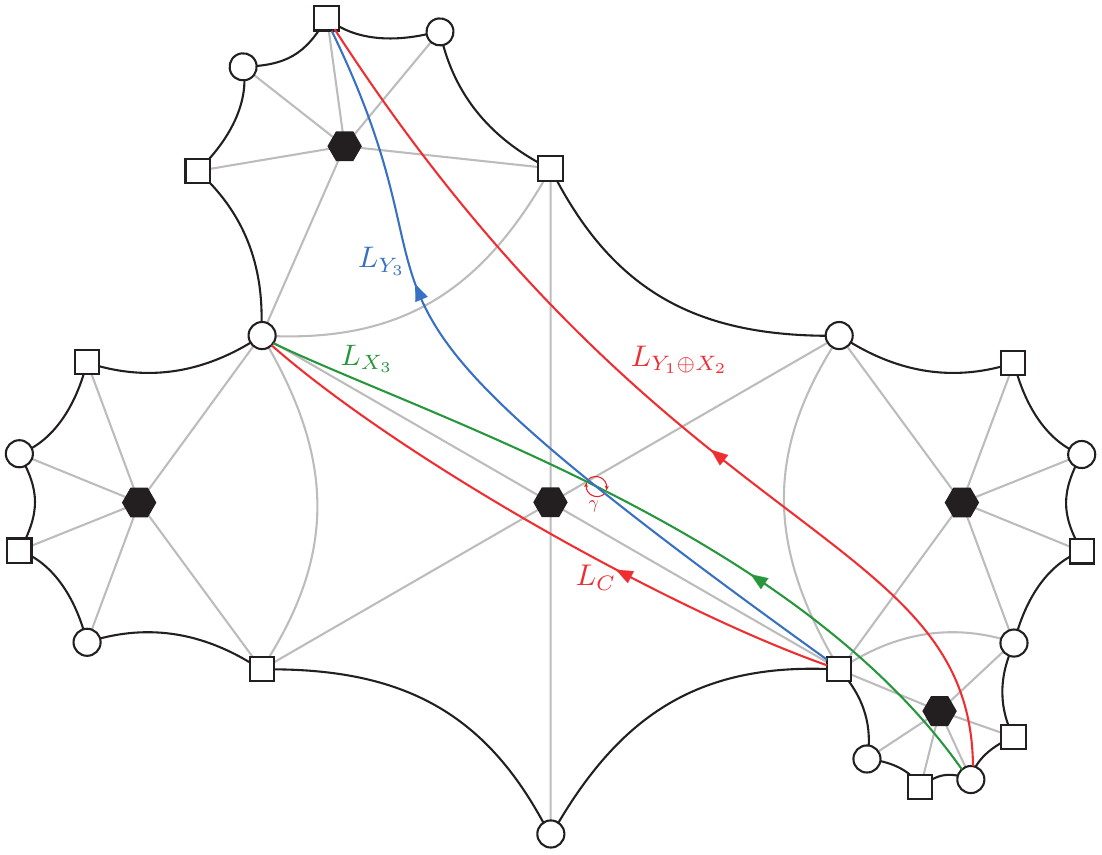}
\centering
\caption{Lagrangian surgery for \eqref{ar:e77}}
\end{subfigure}
\caption{Auslander--Reiten exact sequences for $E_7$ via Lagrangian surgery}
\label{fig:e7mor2}
\end{figure}

\subsection{$\boldsymbol{E_8}$-case}
For $E_8$ singularity $x^3+y^5$, its AR quiver of indecomposable MF's is the following.
$$\begin{tikzcd}[row sep=small,column sep=scriptsize]
& & A_{2} \arrow[dddd] \arrow[r,dash,dashed]  & B_{2} \arrow[dd] & & & & & \\
& & & & & & & & \\
N_{2} \arrow[r] & D_{2} \arrow[ddl] \arrow[rr] & & X_{1} \arrow[uul] \arrow[ddll] \arrow[r] & X_{2} \arrow[ddll] \arrow[r] & C_{1} \arrow[ddl] \arrow[r] & B_{1} \arrow[ddl] \arrow[r] & N_{1} \arrow[ddl] \arrow[dr] & \\
& & & & & & & & R \arrow[dl] \\
M_{2} \arrow[r] \arrow[uu,dash,dashed] & C_{2} \arrow[uul] \arrow[r] \arrow[uu,dash,dashed] & Y_{1} \arrow[uul] \arrow[rr] \arrow[uur,dash,dashed] \arrow[uuuur] & & Y_{2} \arrow[uul] \arrow[r] \arrow[uu,dash,dashed] & D_{1} \arrow[uul] \arrow[r] \arrow[uu,dash,dashed] & A_{1} \arrow[uul] \arrow[r] \arrow[uu,dash,dashed] & M_{1} \arrow[uul] \arrow[uu,dash,dashed] & \\
\end{tikzcd}$$

Here, $M_i$ (resp. $N_i$) is a $2\times 2$ matrix factorization given by $(\phi_i,\psi_i)$ (resp. $(\psi_i,\phi_i)$) for $i=1,2$.
\begin{equation*}
\phi_i = \begin{pmatrix}
x& y^i\\
y^{5-i}& -x^2\\
\end{pmatrix}, \psi_i=
\begin{pmatrix}
x^2& y^i\\
y^{5-i}& -x\\
\end{pmatrix}
\end{equation*}
$A_i$ (resp. $B_i$) is a $3\times 3$ matrix factorization given by $(\alpha_i,\beta_i)$ (resp. $(\beta_i,\alpha_i)$) for $i=1,2$.
$$\alpha_i = \begin{pmatrix}
 y& -x& 0\\
 0&y^i & -x\\
 x&0&  y^{4-i}\\
\end{pmatrix},  \beta_i= \begin{pmatrix}
 y^4& xy^{4-i}&  x^2\\
  -x^2&  y^{5-i}& xy\\
-xy^i&  -x^2&  y^{i+1}\\
\end{pmatrix}
$$
$C_i$ (resp. $D_i$) is a $4\times 4$ matrix factorization given by $(\gamma_i,\delta_i)$ (resp. $(\gamma_i,\delta_i)$) for $i=1,2$.
$$\gamma_1 = \begin{pmatrix}
 y& -x& 0 & y^3\\
 x& 0 & -y^3 & 0 \\
 -y^2&0 & -x^2 &0 \\
 0& -y^2 & -xy&  -x^2\\
\end{pmatrix},  \delta_1= \begin{pmatrix}
 0& x^2& -y^3 & 0\\
 -x^2& xy & 0& -y^3 \\
 0& -y^2&-x &0 \\
 y^2 &0& y&  -x\\
\end{pmatrix}
$$
\begin{equation*} 
\gamma_2= \begin{pmatrix}
\phi_2& \epsilon_1\\
0 & \psi_2\\
\end{pmatrix}, \delta_2=
\begin{pmatrix}
\psi_2& \epsilon_2\\
0 &  \phi_2 \\
\end{pmatrix} 
\;\;\; \textrm{with} \;\; 
\epsilon_1= \begin{pmatrix}
0& y\\
-xy^2 & 0\\
\end{pmatrix},
\epsilon_2= \begin{pmatrix}
0& xy\\
-y^2 & 0\\
\end{pmatrix}.
\end{equation*}
$X_1$ (resp. $Y_1$) is a $6\times 6$ matrix factorization given by $(\xi_1,\eta_1)$ (resp. $(\eta_1,\xi_1)$).
\begin{equation*}
\xi_1= \begin{pmatrix}
\beta_2& \epsilon_3\\
0 & \alpha_2\\
\end{pmatrix}, \eta_1=
\begin{pmatrix}
\alpha_2& \epsilon_4\\
0 &  \beta_2 \\
\end{pmatrix} 
\;\;\; \textrm{with} \;\; 
\epsilon_3= \begin{pmatrix}
0&0& xy\\
-x &0&0 \\
0& -xy & 0\\
\end{pmatrix},
\epsilon_4= \begin{pmatrix}
0&0& -x\\
xy&0&0 \\
0 &xy&  0\\
\end{pmatrix}.
\end{equation*}
$X_2$ (resp. $Y_2$) is a $5\times 5$ matrix factorization given by $(\xi_2,\eta_2)$ (resp. $(\eta_2,\xi_2)$).
$$\xi_2 = \begin{pmatrix}
 y^4& x^2& 0 & -xy^2 &0 \\
 -x^2& xy& 0 & -y^3 & 0 \\
0& -y^2&-x & 0 & y^3\\
-xy^2 & y^3 & 0 & x^2 & 0 \\
-y^3&  0& -y^2 & xy&  -x^2\\
\end{pmatrix},  \eta_2= \begin{pmatrix}
 y & -x & 0 & 0 & 0 \\
x & 0& 0 & y^2 & 0\\
 -y^2& 0& -x^2 & 0& -y^3 \\
 0& -y^2&0& x &0 \\
 0& 0& y^2 & y&  -x\\
\end{pmatrix}
$$

Since $F_{3,5}^T = F_{3,5}$, we consider the Milnor fiber of $F_{3,5} = x^3+ y^{5}$.
First, it is given by $\Z/5$-copies of a hexagon glued as in Figure \ref{fig:e8cell}, whose boundary is identified with $(\pm 5)$ pattern.
$\Z/5$-action is the rotation at the center, and $\Z/3$-action is the simultaneous rotation in every hexagons.  
Lifts of Seidel Lagrangian are drawn as dotted lines, and $XXX$ and $XYZ$-triangles, and pentagon with $Y$-corners produce the potential  $W_\bL= x^3+y^5+xyz$.
With the restriction $z=0$, we get $E_8$ singularity $x^3+y^5$.

\begin{figure}[h]
\begin{subfigure}{0.43\textwidth}
\includegraphics[scale=0.55]{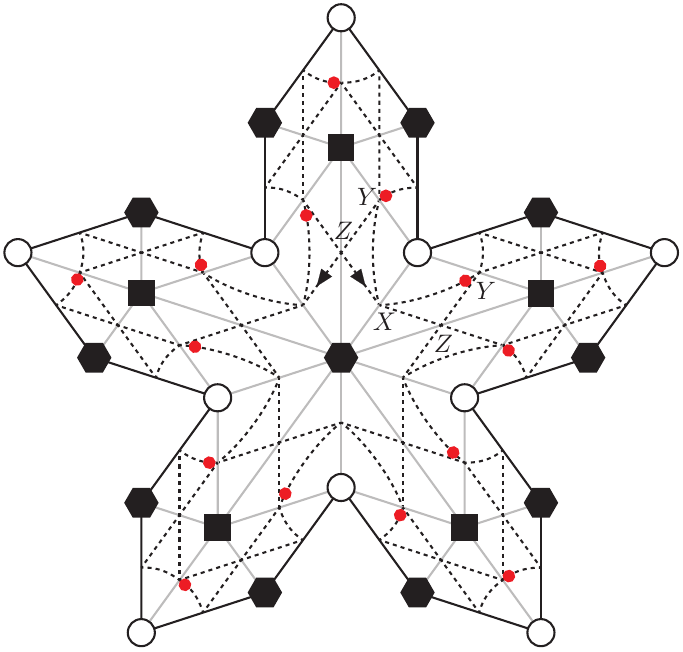}
\centering
\end{subfigure}
\begin{subfigure}{0.43\textwidth}
\includegraphics[scale=0.55]{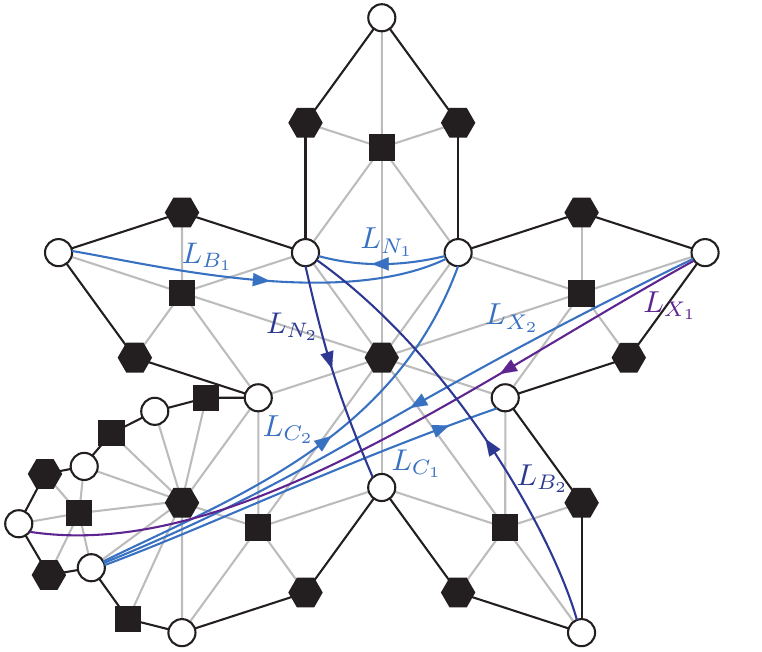}
\centering
\end{subfigure}
\centering
\caption{\label{fig:e8cell} Milnor fiber for $E_8$ and Lagrangians for indecomposable MF's}
\end{figure}
\begin{lemma}
Lagrangians  $ L_{B_i}, L_{C_i}, L_{N_i}, L_{X_i}$ in Figure \ref{fig:e8cell}  correspond to their respective matrix factorizations.
By Lemma \ref{lem:lagtr},
orientation reversals of these Lagrangians correspond to $A_i,D_i,M_i,Y_i$ respectively.
\end{lemma}
\begin{proof}
We explain the case of $X_{2}$. There are 28 polygons;
$$o_{1} Y^{\prime}_{3} Y^{\prime}_{5} Y^{\prime}_{7} Y^{\prime}_{9} e_{1}, o_{1} X_{5} X_{1} e_{2}, o_{1} X_{5} Y_{2} Y_{4} e_{4}, o_{2} X_{2} X_{4} e_{1}, o_{2} X_{2} Y_{9} e_{2}, o_{2} Y_{10} Y_{2} Y_{4} e_{4}, o_{3} Y_{7} X_{4} e_{1}, o_{3} Y_{7} Y_{9} e_{2}, o_{3} X_{9} e_{3},$$
$$o_{3} Y_{12} Y_{14} Y_{16} e_{5}, o_{4} Y_{5} Y_{7} X_{4} e_{1}, o_{4} Y_{5} Y_{7} Y_{9} e_{2}, o_{4} X_{10} X_{12} e_{4}, o_{5} Y_{17} Y_{19} e_{3}, o_{5} Y_{17} X_{12} e_{4}, o_{5} X^{\prime}_{4} X^{\prime}_{6} e_{5}, e_{1} Y_{20} o_{1}, e_{1} X_{6} o_{2},$$
$$e_{2} X_{3} o_{1}, e_{2} Y_{1} Y_{3} o_{4}, e_{3} X_{11} Y_{8} o_{2}, e_{3} X_{11} X_{7} o_{3}, e_{3} Y_{11} Y_{13} Y_{15} o_{5}, e_{4} Y_{6} Y_{8} o_{2}, e_{4} X_{8} o_{4}, e_{5} Y_{18} Y_{5} o_{3}, e_{5} Y_{18} o_{4}, e_{5} X_{13} o_{5}.$$
The vertices marked with prime in the above expression lie outside of Figure \ref{fig:E8_disc}
(but can be obtained by a group action of the vertex without prime).
\begin{figure}[h]
\includegraphics[scale=0.7]{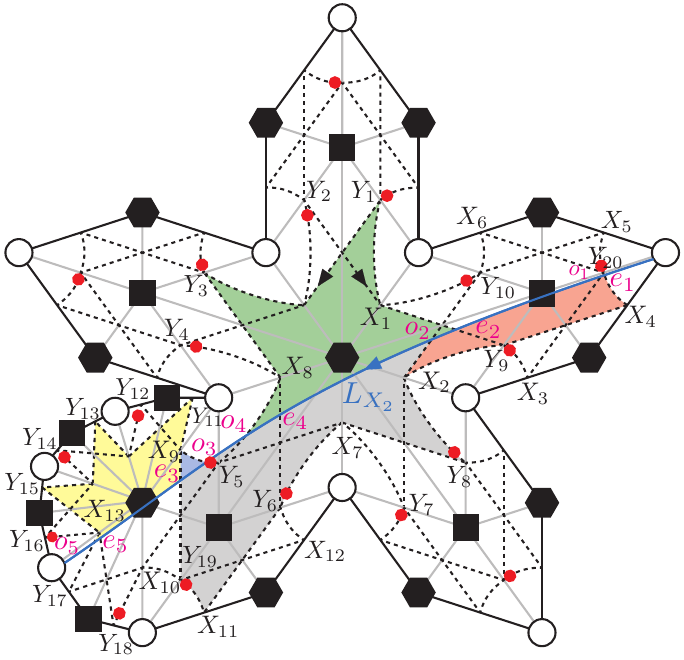}
\centering
\caption{Holomorphic polygons for $m^{0,\bb}_{1}$ calculation}
\label{fig:E8_disc}
\end{figure}
Then $m^{0,\bb}_{1}$ is given as follow.
\begin{align*}
m^{0,\bb}_{1}(o_{1}) &=   y^{4} \cdot e_{1} - x^{2} \cdot e_{2} + xy^{2} \cdot e_{4} &m^{0,\bb}_{1}(e_{1}) &=   y \cdot o_{1} + x \cdot o_{2}\\
m^{0,\bb}_{1}(o_{2}) &=  x^{2} \cdot e_{1} + xy \cdot e_{2} - y^{3} \cdot e_{4} &m^{0,\bb}_{1}(e_{2}) &=   -x \cdot o_{1} + y^{2} \cdot o_{4}\\
m^{0,\bb}_{1}(o_{3}) &=   xy \cdot e_{1} + y^{2} \cdot e_{2} - x \cdot e_{3} + y^{3} \cdot e_{5} &m^{0,\bb}_{1}(e_{3}) &=   xy \cdot o_{2} - x^{2} \cdot o_{3} + y^{3} \cdot o_{5}\\
m^{0,\bb}_{1}(o_{4}) &=   xy^{2} \cdot e_{1} + y^{3} \cdot e_{2} + x^{2} \cdot e_{4} &m^{0,\bb}_{1}(e_{4}) &=   -y^{2} \cdot o_{2} + x \cdot o_{4}\\
m^{0,\bb}_{1}(o_{5}) &=   y^{2} \cdot e_{3} + xy \cdot e_{4} + x^{2} \cdot e_{5} &m^{0,\bb}_{1}(e_{5}) &=   y^{2} \cdot o_{3} - y \cdot o_{4} + x \cdot o_{5}\\
\end{align*}
Putting them into two matrices, we obtain $X_2$.
\end{proof}

\begin{lemma}
The following AR exact sequences (and their AR translation) for $E_8$ singularity can be realized as Lagrangian surgeries.

\centering{
\noindent \begin{subequations}
\begin{minipage}{0.45\textwidth}
\begin{align}
0 \to N_1 \to A_1 \oplus R \to M_1 \to 0 \label{ar:e81}
\end{align}
\end{minipage}
\begin{minipage}{0.45\textwidth} 
\begin{align}
0 \to B_1 \to D_1 \oplus N_1 \to A_1 \to 0 \label{ar:e82}
\end{align}
\end{minipage}\\
\begin{minipage}{0.45\textwidth}
\begin{align}
0 \to N_2 \to D_2 \to M_2 \to 0 \label{ar:e83}
\end{align}
\end{minipage}
\begin{minipage}{0.45\textwidth}
\begin{align}
0 \to B_2 \to X_1 \to A_2 \to 0 \label{ar:e84}
\end{align}
\end{minipage}\\
\begin{minipage}{0.45\textwidth}
\begin{align}
0 \to C_1 \to B_1 \oplus Y_2 \to D_1 \to 0 \label{ar:e85}
\end{align}
\end{minipage}
\begin{minipage}{0.45\textwidth}
\begin{align}
0\to D_2 \to M_2 \oplus X_1 \to C_2 \to 0 \label{ar:e86}
\end{align}
\end{minipage}\\
\begin{minipage}{0.45\textwidth}
\begin{align}
0 \to X_2 \to Y_1 \oplus C_1 \to Y_2 \to 0 \label{ar:e87}
\end{align}
\end{minipage}
\begin{minipage}{0.45\textwidth}
\begin{align}
0 \to X_1 \to A_2 \oplus C_2 \oplus X_2  \to Y_1 \to 0 \label{ar:e88}
\end{align}
\end{minipage}
\end{subequations}}
\end{lemma}
\begin{proof}

\begin{figure}[b]
\begin{subfigure}{0.46\textwidth}
\includegraphics[scale=0.46]{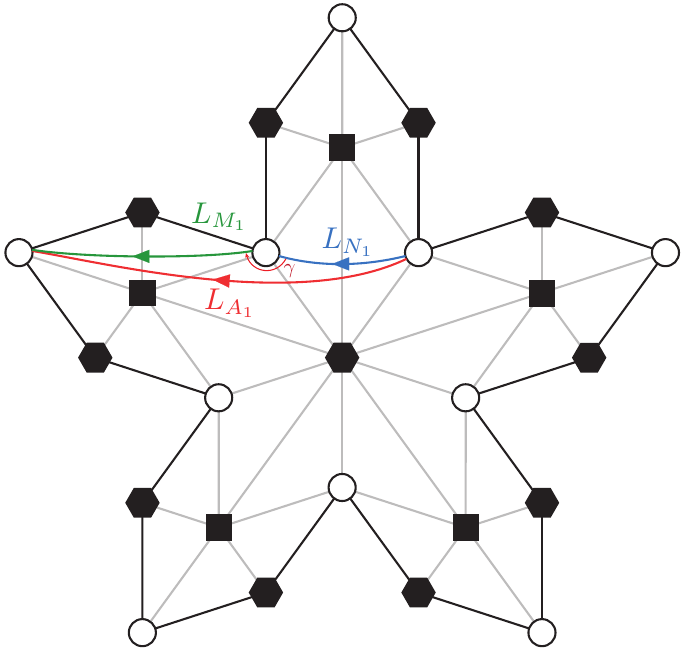}
\centering
\caption{Lagrangian surgery for \eqref{ar:e81}}
\end{subfigure}
\begin{subfigure}{0.46\textwidth}
\includegraphics[scale=0.46]{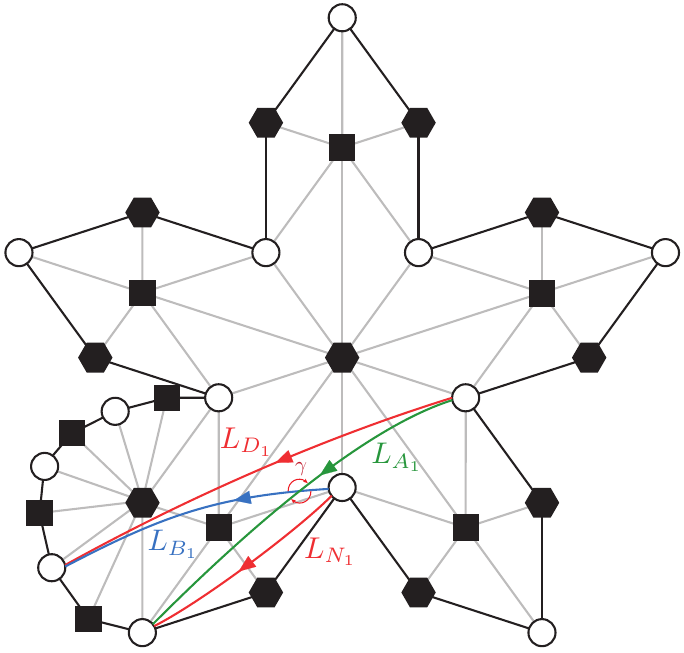}
\centering
\caption{Lagrangian surgery for \eqref{ar:e82}}
\end{subfigure}
\centering
\begin{subfigure}{0.46\textwidth}
\includegraphics[scale=0.46]{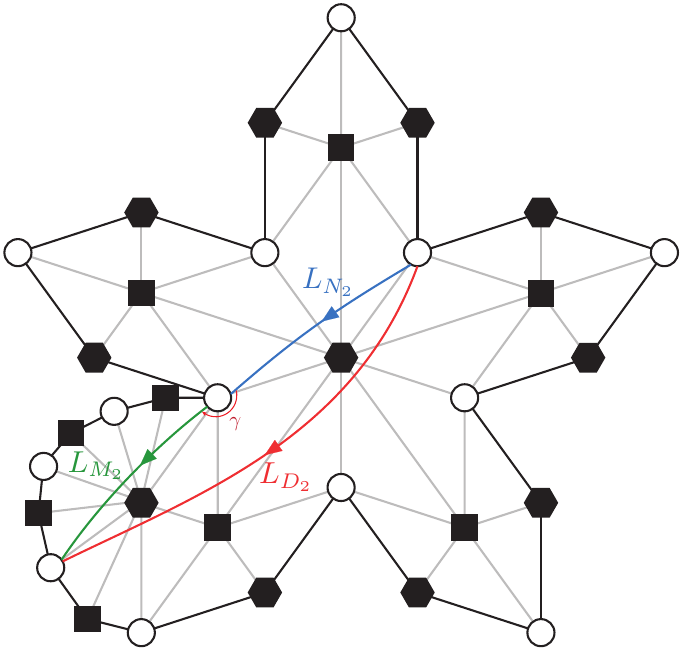}
\centering
\caption{Lagrangian surgery for \eqref{ar:e83}}
\end{subfigure}
\begin{subfigure}{0.46\textwidth}
\includegraphics[scale=0.46]{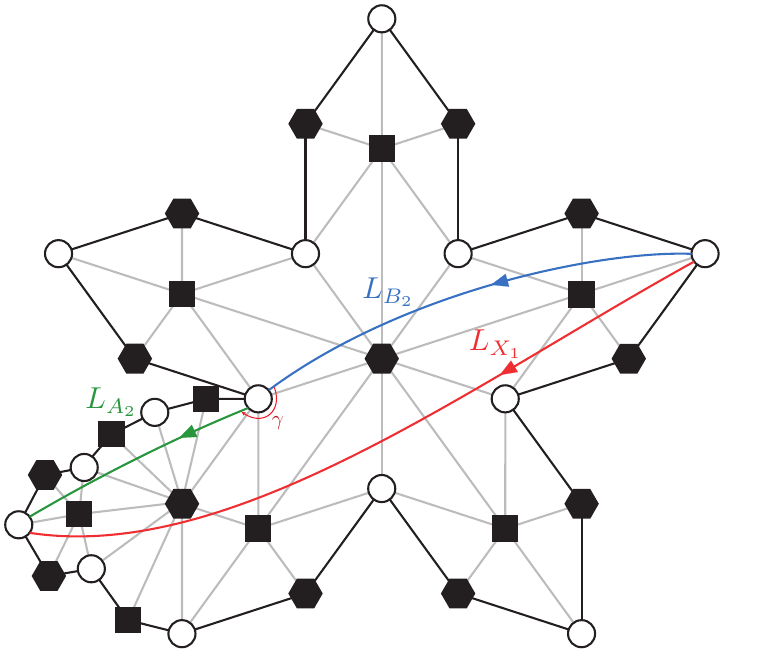}
\centering
\caption{Lagrangian surgery for \eqref{ar:e84}}
\end{subfigure}
\centering
\caption{\label{fig:e8mor3}  Auslander--Reiten exact sequences for $E_8$ via Lagrangian surgery}
\end{figure}

\begin{figure}[h]
\begin{subfigure}{0.46\textwidth}
\includegraphics[scale=0.46]{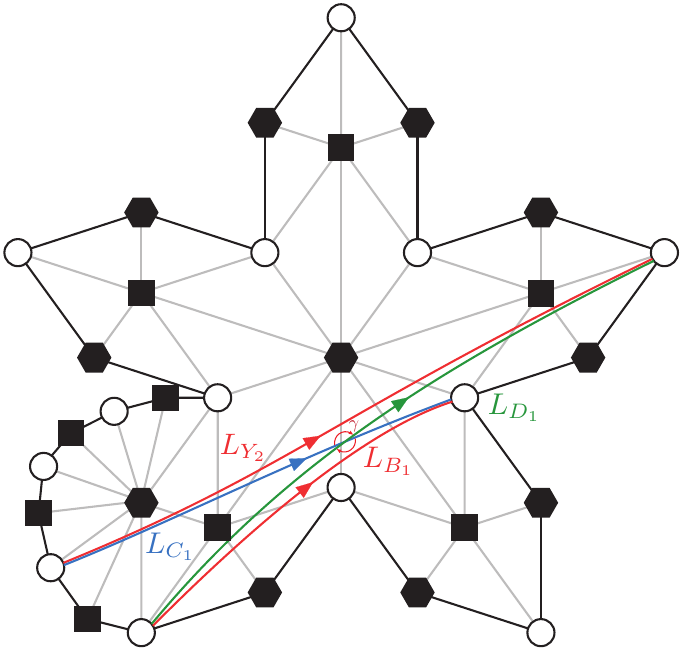}
\centering
\caption{Lagrangian surgery for \eqref{ar:e85}}
\end{subfigure}
\begin{subfigure}{0.46\textwidth}
\includegraphics[scale=0.45]{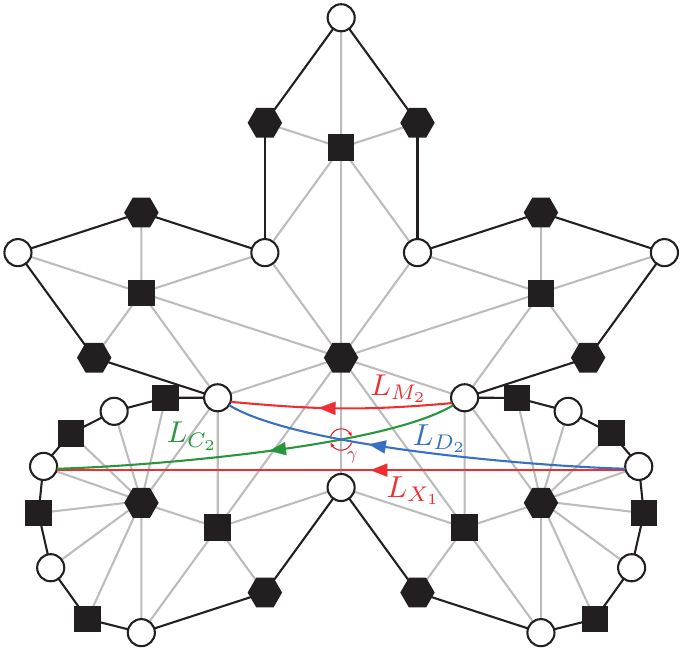}
\centering
\caption{Lagrangian surgery for \eqref{ar:e86}}
\end{subfigure}
\centering
\begin{subfigure}{0.46\textwidth}
\includegraphics[scale=0.5]{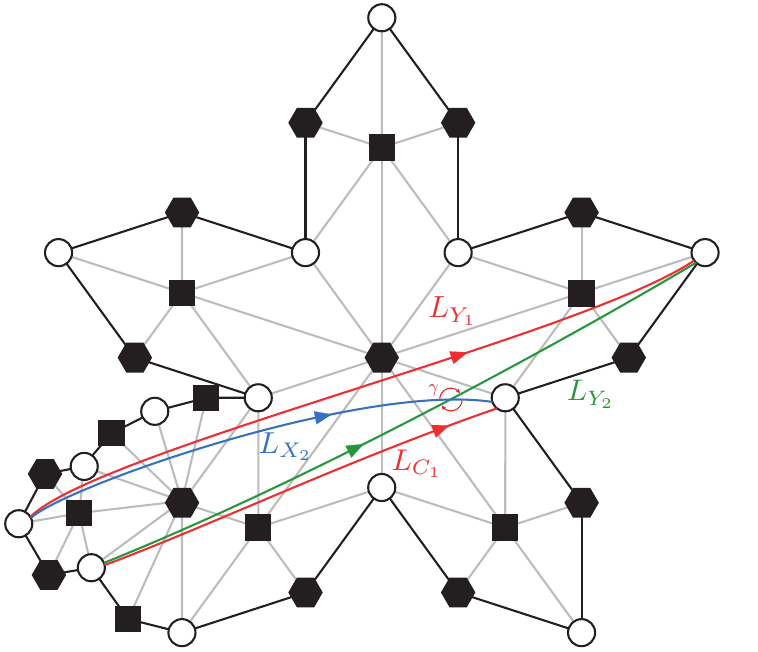}
\centering
\caption{Lagrangian surgery for \eqref{ar:e87}}
\label{fig:e8cell2}
\end{subfigure}
\begin{subfigure}{0.46\textwidth}
\includegraphics[scale=0.5]{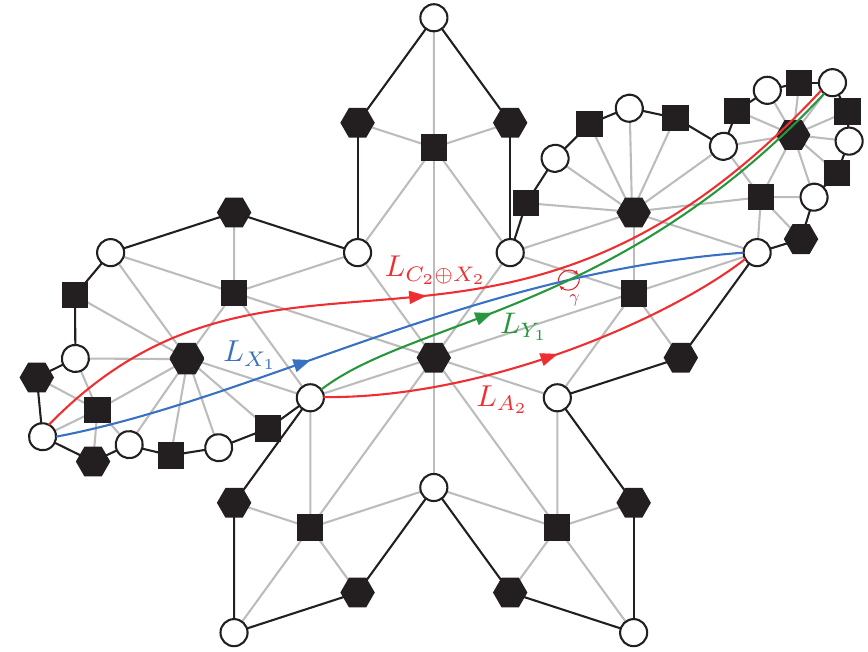}
\centering
\caption{Lagrangian surgery for \eqref{ar:e88}}
\label{fig:e8cell3}
\end{subfigure}
\caption{\label{fig:e8mor4}  Auslander--Reiten exact sequences for $E_8$ via Lagrangian surgery}
\end{figure}

We show that $L_{C_{2} \oplus X_{2}}$ corresponds to the direct sum $C_{2} \oplus X_{2}$ in \eqref{ar:e88}. $L_{C_{2} \oplus X_{2}}$ maps to $9 \times 9$ matrix factorization. As a mapping cone, we give a null-homotopy of the corresponding morphism.

\begin{equation*}
\begin{split}
\small \begin{pmatrix}
x& y^{3}& 0& 0\\
y^{2}& -x^{2}& 0& 0\\
0& xy^{2}& x^{2}& y^{3}\\
-y& 0& y^{2}& -x\\
\end{pmatrix}
\begin{pmatrix}
x& -y& 0& 0& 0\\
0& 0& 0& -1& 0\\
0& -1& 0& 0& 0\\
-y& 0& 0& 0& -1\\
\end{pmatrix}
+&
\small \begin{pmatrix}
0& 1& 0& 0& 0\\
0& 0& 0& -1& 0\\
1& 0& y& 0& 0\\
0& 0& 0& 0& 0\\
\end{pmatrix}
\begin{pmatrix}
y^{4}& x^{2}& xy& xy^{2}& 0\\
-x^{2}& xy& y^{2}& y^{3}& 0\\
0& 0& -x& 0& y^{2}\\
xy^{2}& -y^{3}& 0& x^{2}& xy\\
0& 0& y^{3}& 0& x^{2}\\
\end{pmatrix} \\
&\hspace{36mm}=
\small \begin{pmatrix}
0& 0& y^{2}& 0& 0\\
0& 0& 0& 0& -xy\\
0& 0& 0& 0& 0\\
0& 0& 0& 0& x\\
\end{pmatrix}
\end{split}
\end{equation*}
\begin{equation*}
\begin{split}
\small \begin{pmatrix}
x^{2}& y^{3}& 0& 0\\
y^{2}& -x& 0& 0\\
0& y^{2}& x& y^{3}\\
-xy& 0& y^{2}& -x^{2}\\
\end{pmatrix}
\begin{pmatrix}
0& 1& 0& 0& 0\\
0& 0& 0& -1& 0\\
1& 0& y& 0& 0\\
0& 0& 0& 0& 0\\
\end{pmatrix}
+&
\small \begin{pmatrix}
x& -y& 0& 0& 0\\
0& 0& 0& -1& 0\\
0& -1& 0& 0& 0\\
-y& 0& 0& 0& -1\\
\end{pmatrix}
\begin{pmatrix}
y& -x& 0& 0& 0\\
x& 0& xy& -y^{2}& 0\\
0& 0& -x^{2}& 0& y^{2}\\
0& y^{2}& 0& x& -y\\
0& 0& y^{3}& 0& x\\
\end{pmatrix} \\
&\hspace{37.5mm}=
\small \begin{pmatrix}
0& 0& -xy^{2}& 0& 0\\
0& 0& 0& 0& y\\
0& 0& 0& 0& 0\\
0& 0& 0& 0& -x\\
\end{pmatrix}
\end{split}
\end{equation*}

\end{proof}


\section{Homology categories for invertible curve singularities}\label{sec:A3}
In \cite{CCJ20}, we constructed a new $\AI$-category $\cF(W, G)$ associated to $(W,G)$ of log Fano or Calabi-Yau type. Invertible curve singularities are of log general type (except $F_{2,2}$), hence $\AI$-structural equations may have obstructions to hold. In this section, we show that the construction of \cite{CCJ20} defines at least an $A_3$-structure so that we can define the homology category $\mathcal{H}(W,G)$ for invertible curve singularities.

Let us consider the case of $G=G_W$ for simplicity, as the case of general $G$ follows from that of $G_W$ as in Section 7.6 of \cite{CCJ20}. Recall that $\cF(W, G_{W})$ and $\mathcal{WF}([M_W/G_W])$ has the same set of objects. A morphism space between $L_{1}$, $L_{2} \in \cF(W, G_{W})$ is given by
$$\hom_{\cF(W, G_W)}(L_1, L_2) \coloneqq CW(L_{1},L_{2}) \oplus CW(L_{1},L_{2}) \epsilon$$
where $\epsilon$ is a formal variable of degree $(-1+\deg \Gamma_W)$.  The differential $M_1$ is defined so that the cohomology of $\cF(W, G_{W})$ is the cohomology of the cone complex \eqref{eq:capi} of quantum cap action $\cap \Gamma_W$. Namely, the cap action is the component of $M_1$ from the $\epsilon$-component $CW(L_{1},L_{2}) \epsilon$ to $CW(L_{1},L_{2})$ in the above. Recall that quantum cap action has an insertion on the geodesic line between the input and output marked point.

In \cite{CCJ20}, the new $\AI$-structure map generalizes this quantum cap action with arbitrary many insertions of $\Gamma_W$. Namely, whenever an input of $\AI$-operation is from the $\epsilon$-components,  we insert $\Gamma_W$ on the geodesic from the corresponding input marked point to the output marked point.  Discs with such geodesics and interior markings are called popsicles, and they were first introduced by Abouzaid and Seidel to define a wrapped Fukaya category \cite{AS}. In their case, interior marked points are used to locate the support of suitable one forms for telescoping and hence were not used as inputs.

Interior marked points are allowed to overlap each other in \cite{AS} but this is not the case in our setup. 
This requires a different compactification. The new compactification turns out to be quite tricky because some of the conformal structures of the discs and sphere are correlated (due to the popsicle lines). We introduced the notion of alignment data to keep track of such relations between conformal structures and defined the new compactification in \cite{CCJ20}. In this setting, new codimension 1 strata with sphere bubbles may appear and these are new obstructions to define the $\AI$-structure (Figure \ref{fig:popdeg} for example).  Under the log Fano and Calabi-Yau condition, we have shown that such sphere bubbles of codimension one does not appear, and we get the $A_{\infty}$-structure on $\cF(W, G_{W})$. We refer readers to \cite{CCJ20} for details.

\begin{figure}[h]
\includegraphics[scale=0.5]{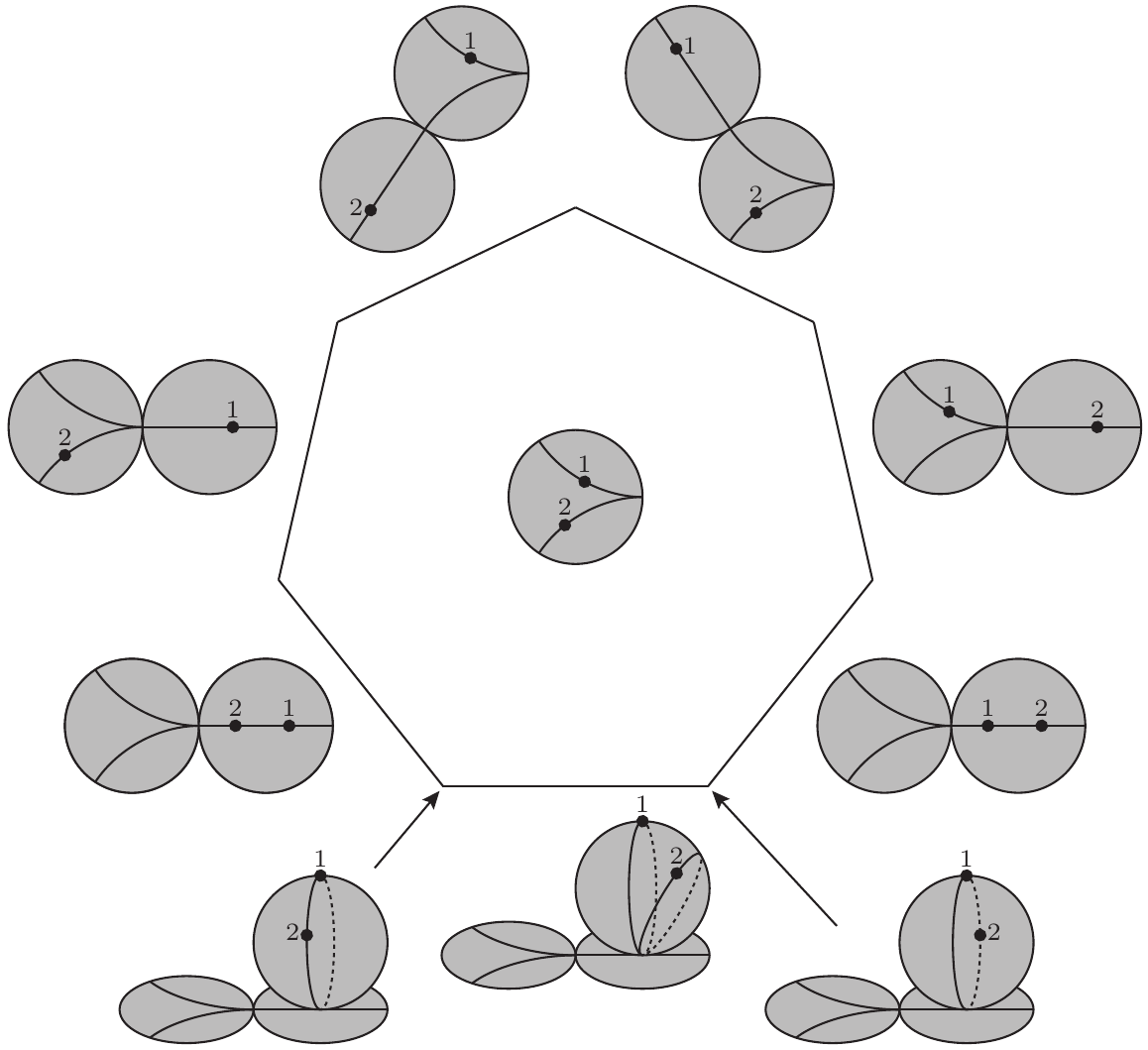}
\centering
\caption{Moduli space compactification having codimension one sphere bubbles}
\label{fig:popdeg}
\end{figure}

\subsection{$A_3$-category structure on $\cF(W, G_W)$}

\begin{defn}
An $A_{3}$-structure on the category is a collection of three operators $M_{1}$, $M_{2}$ and $M_{3}$ satisfying $M_1^2 =0$ and 
\begin{equation} \label{eqn:cond1}
M_{1}(M_{2}(a,b)) \pm M_{2}(M_{1}(a),b) \pm M_{2}(a,M_{1}(b)) = 0,
\end{equation}
\begin{equation} \label{eqn:cond2}
\begin{aligned}
M_{1}(M_{3}(a,b,c)) &\pm M_{2}(M_{2}(a,b),c) \pm M_{2}(a,M_{2}(b,c)) \\
&\pm M_{3}(M_{1}(a),b,c) \pm M_{3}(a,M_{1}(b),c) \pm M_{3}(a,b,M_{1}(c)) = 0
\end{aligned}
\end{equation}
where the signs are the standard Koszul signs as in \cite{FOOO}.
\end{defn}
 As we mentioned, $M_1$ is the differential of the cone complex  \eqref{eq:capi}, hence $M_1^2=0$ already holds. A construction of three operators $M_{1}$, $M_{2}$ and $M_{3}$ is the same as in \cite{CCJ20}. 
 Recall that the $\AI$-identities are deduced from the compactification of the moduli space of popsicle discs, or more precisely from their codimension one faces. The obstructions are given by the possible popsicle sphere bubbles that appear in codimension 1 strata.

\begin{figure}[h]
\begin{subfigure}{0.43\textwidth}
\includegraphics[scale=1]{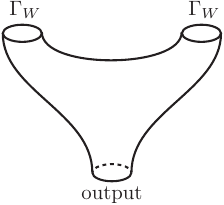}
\centering
\caption{Sphere bubble with 2 $\Gamma_W$ inputs}
\end{subfigure}
\begin{subfigure}{0.43\textwidth}
\includegraphics[scale=1]{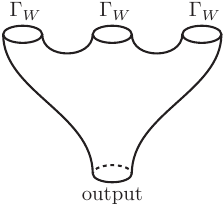}
\centering
\caption{Sphere bubble with 3 $\Gamma_W$ inputs}
\end{subfigure}
\centering
\caption{ }
\end{figure}

\begin{thm} \label{thm:A3}
For the invertible curve singularity $W$ except three cases $F_{3,3}$, $C_{3,2}$, and $L_{2,2}$, there exist the $A_{3}$-structure on $\cF(W, G_W)$.
\begin{proof}
We need to show the vanishing of obstructions for the $A_3$-algebra equations.
 We already have $M_1^2 =0$. Also note that a component of $M_1$ counts popsicle disc with one $\Gamma_W$ insertion, but the proof of $M_1^2=0$ only involves popsicle disc with one $\Gamma_W$ insertion instead of two $\Gamma_W$ insertions. This is because the output of a popsicle disc of a quantum cap action lies in a component without $\epsilon$
 
Similarly, the $A_3$-equations amounts to show the vanishing of codimension one popsicle sphere bubbles in the case of two or three $\Gamma_W$-insertions. The proof will be a repeated application of the following observations. 

\begin{lemma}
\label{lemma: degree}
Write $\gamma_\mathrm{out}$ as a possible nontrivial output of a popsicle sphere bubble with $N\geq 2$ many $\Gamma_W$-insertions. Then,
\begin{enumerate}
	\item $\deg \gamma_\mathrm{out} < N\cdot\deg\Gamma_W-N+1$
	\item $[\gamma_\mathrm{out}] \in H_1^\mathrm{orb}$ (see \eqref{eqn:1sthlgy}) must be represented as $-K_i\cdot \gamma_i$ for some $1\leq i \leq 3$ with $K_i$ nonnegative. 
	\item Moreover, such $K_i$ can be chosen to be strictly less than the winding number of $N$-th iterate of $\Gamma_W$ around $\gamma_i$.
\end{enumerate}
\end{lemma}
\begin{proof}
The first item follows from \cite[Proposition 5.4]{CCJ20} where such operation is shown to be of a strictly negative degree. For the second item observe that $\gamma_\mathrm{out}$ is a (possibly a constant) Reeb orbit and therefore either zero or the multiple of a single $\gamma_i$, which also carries an orientation opposite from the induced orientation on the boundary. The third item follows from the first one. If not, the degree of a Hamiltonian orbit represents $K_i\cdot\gamma_i$ is bigger or equal to $N \deg \Gamma_W$, a contradiction. 
\end{proof}

The following case analysis shows that a nontrivial output $\gamma_\mathrm{out}$ of a popsicle sphere bubble operation cannot satisfy all the conditions in Lemma \ref{lemma: degree}.

\begin{itemize}
\item Fermat cases $F_{p,q}$:
Since $F_{2,2}$ is log Calabi--Yau, any popsicle sphere bubble does not exist as proved in \cite{CCJ20}.
Let us assume that $(p,q) \neq (2,2)$. Notice that the orbifold $[M_W/G_W]$ has only one boundary component.

Let $\gamma_\mathrm{out}$ be an output of a popsicle sphere bubble with two $\Gamma_W$-insertion.
Observe that the only possible way to represent $[\gamma_\mathrm{out}]$ as a non-positive multiple of a single $\gamma_i$ is  $-2\gamma_3$. We conclude that there are no such $\gamma_\mathrm{out}$ by applying Lemma \ref{lemma: degree}.

Further assume that $(p,q) \neq (3,3)$. For a sphere bubble with three $\Gamma_W$ insertions, the only possible way to represent $[\gamma_\mathrm{out}]$ as a non-positive multiple of a single $\gamma_i$ is  $-3\gamma_3$. The above argument works exactly the same.

\item Chain cases $C_{p,q}$:
The quotient $M_{C_{p,q}}/G_{C_{p,q}}$ has two boundary components and $\Gamma_W$ also has two components $(1-p)\gamma_{1}$ and $-\gamma_{3}$ in each boundary component. 

For a sphere bubble of two $\Gamma_W$ insertions, we divide a possible homology class of $\gamma_\mathrm{out}$ into the following three cases.
\begin{itemize}
\item $2(1-p)\gamma_{1}$: This is the only possible way to represent it as a non-positive multiple of a single $\gamma_i$, which violates (3) of Lemma \ref{lemma: degree}.
\item $(1-p)\gamma_{1} -\gamma_{3}$: It does not satisfy (2) of Lemma \ref{lemma: degree}.
\item $-2\gamma_{3}$ : Again, it violates (3) of Lemma \ref{lemma: degree} because this is the only possible way tof representing $[\gamma_\mathrm{out}]$ as a non-positive multiple of a single $\gamma_i$.
\end{itemize}

For a sphere bubble of three $\Gamma_W$ insertions, we divide a possible output into four cases:
\begin{itemize}
\item $3(1-p)\gamma_{1}$: Does not satisfy (3) of Lemma \ref{lemma: degree}.
\item $2(1-p)\gamma_{1}-\gamma_{3}$: Does not satisfy (2) of Lemma \ref{lemma: degree}.
\item $(1-p)\gamma_{1} - 2\gamma_{3}$: When $q \neq 2$, it does not satisfy (2) of Lemma \ref{lemma: degree}. If $q=2$, It can be represented as $(3-p)\gamma_{1}$ since $2\gamma_{1} + 2\gamma_{3} = 0$. If $(p,q) = (2,2)$, then the coefficient is positive, so we may exclude this case. If $(p, q) = (p \geq 4$, 2), the argument is a little more involved; according to \cite[Proposition 8.15]{CCJ20}, one can compute the degree of the Morse-Bott family of Reeb orbits $\Sigma_{g,l}$ representing $(3-p)\gamma_{1}$.
They are twisted by the element $J^{-3} \in G_W$ with a period $\frac{p-3}{p-1}$. Following the notation therein, we have
$$
(g, l) = \left(J^{-3}, \frac{p-3}{p-1} \right), \hspace{5pt} I = \{y\}, \hspace{5pt} \theta_{J} = \left(\frac{1}{p}, \frac{p-1}{2p}\right), \hspace{5pt} \theta_{g} =  \left(\frac{p-3}{p}, \frac{p-3}{2p}\right).
$$
Plugging in this data to the formula, we get $\mu_\mathrm{RS}(\Sigma_{g,l}) = 1-\frac{3(p-3)}{p}$ and therefore 
\begin{equation} \label{eqn:deg1}
\deg \gamma_\mathrm{out} \in \left\{ 1- \mu_\mathrm{RS}(\Sigma_{g,l}), 2- \mu_\mathrm{RS}(\Sigma_{g,l}) \right\}  = \left\{ \frac{3(p-3)}{p} , \frac{3(p-3)}{p} +1\right\}.
\end{equation}
On the other hand, we have
\begin{equation} \label{eqn:deg2}
3 \deg \Gamma_W - 2 = 3 \left( \frac{p-1}{p} \right) - 2 = \frac{p-3}{p}.
\end{equation}
which is strictly less than the elements in \eqref{eqn:deg1}. It contradicts (1) of Lemma \ref{lemma: degree}.
\item $-3\gamma_{3}$ : If $q=3$ then it violates (2) of Lemma \ref{lemma: degree}. Otherwise, it violates (3) of Lemma \ref{lemma: degree}. 
\end{itemize}

\item Loop cases $L_{p,q}$:
 $\Gamma_W$ has three components $(1-p)\gamma_{1}$, $(1-q)\gamma_{2}$ and $-\gamma_{3}$. 
Let us we assume $(p,q) \neq (2,2)$ first. When all inputs of a sphere bubble are the same class, then its possible output cannot satisfy (3) of Lemma \ref{lemma: degree}. Otherwise, it violates (2) of Lemma \ref{lemma: degree}. 
 \end{itemize}
Thus, we can construct an $A_{3}$-structure except the cases $F_{3,3}$, $C_{3,2}$, and $L_{2,2}$.
\end{proof}
\end{thm}

\begin{remark}
\label{not enough}
The above topological argument has a few exceptions. For $F_{3,3}$ case, $-3\gamma_3$ is trivial in $H_1^{\mathrm{orb}}$ and hence can be represented by generators in the interior of the orbifold $[M_W/G_W]$. Hence, we cannot exclude the existence of such sphere bubbles using purely topological arguments for the case of  $(p,q) = (3,3)$. For example, there is a $3:1$ map from the pair of pants to the maximal quotient.

A similar phenomenon happens for $L_{2,2}$ case.  From the relation $\gamma_{1} + \gamma_{2} + \gamma_{3} = 0$, the maximal quotient which is a pair of pants itself represents a possible sphere bubble with three inputs $-\gamma_{1}$, $-\gamma_{2}$, $-\gamma_{3}$, and one interior output. Hence, we cannot exclude this case topologically.
\end{remark}

By definition, $M_{2}$ induces an associative product on a cohomology of morphism space. Therefore, we can define a cohomology category of $\cF(W, G_W)$ and denote it by $\cH(W, G_W)$.

\subsection{Computations in $\cH(W,G_W)$}
We compute the cohomology of a Lagrangian $K$ in $[M_W/G_W]$, by explicitly computing
the quantum cap action of $\Gamma_W$. Wrapped Floer cohomology $K$ is infinite dimensional but 
most of them lie in the image of cap actions of $\Gamma_W$, and the resulting cohomology becomes finite dimensional.

At the cohomology level, 
$\Gamma_W$ acts on $HW^\bullet(L,L)$, and is the same as multiplication by $\mathrm{CO}_L(\Gamma_W) \in HW^\bullet(L,L)$
by \cite[Proposition 4.2]{CCJ20}. 
For a non-compact Lagrangian $L$, there is an obvious candidate for $\mathrm{CO}_L(\Gamma_W)$.
Our $\Gamma_W$ corresponds to time-1 periodic Reeb orbits of $[L_{W, \delta}/G_W] $. 
In a similar way, we can apply the time-1 Reeb flow to generate Hamiltonian chords on the intersection points $L \cap [L_{W, \delta}/G_W]$. We call it a \textit{monodromy chord of $L$} and denote it by $\gamma_{W, L} \in CW^\bullet(L,L)$.

\begin{prop} We have $\mathrm{CO}_{L}(\Gamma_W)=\gamma_{W,L}$. Hence, 
  \[\Hom_{\cF(W, G_W)}(L, L) \simeq \Cone \left(m_2(-, \gamma_{W,L}):HW^\bullet(L, L) \to HW^\bullet(L,L)\right).\]
\end{prop}

\begin{proof}
By Lemma \ref{topological grading}, $CW^\bullet(L, L)$ carries an $H_1$-grading since $L$ is simply connected. 
Also, closed-open map preserves $H_1$-grading as well. Note that for each component of $\partial L$, there is the unique wrapped
generator which has the same $H_1$-grading as $\Gamma_W$ on that conical end. The sum of these chords is $\gamma_{W, L}$. Since our Milnor fiber is a Riemann surface, it is not hard to see that
$\mathrm{CO}_L(\Gamma_W)$ is exactly $\gamma_{W,L}$, following the standard example of cylinder (see \cite{PJ} for more details).
\end{proof}

Let $K$ be a split-generator of $\WF([M_W/G_W])$ given as in Figure \ref{fig:splitgen}. More precisely, it is a Lagrangian whose boundary lies in the specific boundary component of $[M_W/G_W]$, which is
the c-vertex in $\mathbb{P}^1_{a,b,c}$ in Proposition \ref{prop:abc} (This is represented by a circle in Section \ref{sec:AR}).

\begin{thm} \label{thm:qiso}
Let $\mathcal{G}:\mathcal{WF}([M_W/G_W]) \to \mathcal{MF}(W^T)$ be an $\AI$-functor which is the composition of localized mirror functor and restriction functor in Theorem \ref{thm:intro1}. Then, the following isomorphism holds:
$$\Hom_{\cH(W, G_W)}(K,K) \cong \Hom_{\mathcal{MF}(W^T)}(\mathcal{G}(K),\mathcal{G}(K)).$$
\end{thm}

\begin{proof}
In each case, $\mathcal{G}(K)$ is given as follows:
\begin{itemize}
\item Fermat  $$\mathcal{G}(K) = \begin{pmatrix}
y& x \\
-x^{p-1}& y^{q-1}
\end{pmatrix}
\begin{pmatrix}
y^{q-1} & -x \\
x^{p-1}& y \\
\end{pmatrix},$$
\item Chain  $$\mathcal{G}(K) = \begin{pmatrix}
y& x \\
0& x^{p}+y^{q-1}
\end{pmatrix}
\begin{pmatrix}
x^{p}+y^{q-1} & -x \\
0& y \\
\end{pmatrix},$$
\item Loop $$\mathcal{G}(K) = \begin{pmatrix}
y& x \\
0& x^{p}+xy^{q-1}
\end{pmatrix}
\begin{pmatrix}
x^{p}+xy^{q-1} & -x \\
0& y \\
\end{pmatrix}.$$
\end{itemize}
After some elementary operations, it is easy to see that $\mathcal{G}(K)$ is in fact $k^{\mathrm{stab}}_{W^T}$.

We can compute $\Hom_{\cH(W, G_W)}(K,K)$ as in Lemma \ref{lem:fermatww}. We have
\[CW^\bullet(K, K)\simeq T(a,b)/\mathcal R_{F_{p,q}}, \hskip 0.3cm \mathcal R_{F_{p,q}} = \langle a\otimes a =\delta_{2,p}, b\otimes b=\delta_{2,q}.\rangle\]
Since $K$ intersect the boundary twice, we have $\gamma_{W, K}=ab+ba$. An $m_2$-multiplication is injective, so we conclude that the cone of $m_2(-,\gamma_{W, K})$ is a cokernel of it. Hence, we have
\[\Hom_{\cH(W,G_W)}(K,K) \simeq \langle 1, a,b, a\otimes b\rangle.\]
This result matches to the computation of $\mathrm{End}_{\MF(W^T)}\left(k^{\mathrm{stab}}_{W^T}\right)$ in \cite{Dy}.
\end{proof}

\begin{remark}
Since $\cF(W, G_W)$ is not an $\AI$-category, the notion of a split-generator is not defined.
But if it is, it comes with a canonical $\AI$-functor from $\mathcal{WF}([M_W/G_W])$. (which has the same set of objects) and
hence split generators will map to split generators.
\end{remark}

\appendix

\section{Determination of Auslander--Reiten sequence}\label{sec:ac}
This appendix is due to Osamu Iyama. The result in this appendix guarantees that the exact sequences from Lagrangian Floer theory
indeed coincides with those in Auslander--Reiten theory.
  \begin{defn}
Let $\mathcal{E}$ be an exact category.
We say that \emph{an exact sequence $0\to A\xrightarrow{a}B\xrightarrow{b}C\to0$ is determined by its terms} if, for each exact sequence $0\to A\xrightarrow{a'}B\xrightarrow{b'}C\to0$, there exists a commutative diagram whose vertical maps are isomorphism:
\[
      \begin{tikzcd}
        0 \arrow[r]  & A \arrow[r,"a^{\prime}"] \arrow[d, "\simeq"] & B \arrow[r,"b^{\prime}"] \arrow[d, "\simeq"]&  C \arrow[r] \arrow[d, "\simeq"] & 0\\
        0 \arrow[r] & A \arrow[r,"a"]& B \arrow[r,"b"]&  C \arrow[r] & 0
      \end{tikzcd}
      \]
\end{defn}

For a Cohen-Macaulay local ring $R$, we denote by $\mbox{CM} R$ the category of maximal Cohen--Macaulay $R$-modules.

\begin{thm}\label{determined by terms}
Let $R$ be a complete local Cohen--Macaulay isolated singularity.\begin{enumerate}
\item Each split exact sequence in $\mbox{CM} R$ is determined by its terms.
\item Each almost split sequence in $\mbox{CM} R$ is determined by its terms.
\end{enumerate}
\end{thm}

Let $\mathcal{E}$ be an exact category with enough projectives.
We denote by $\underline{\mathcal{E}}$ the stable category.
It has the same objects as $\mathcal{E}$, and the morphisms between $X,Y\in\mathcal{E}$ are given by
\[\underline{\Hom}_{\mathcal{E}}(X,Y) := \Hom_{\mathcal{E}}(X,Y)/P(X,Y),\]
where $P(X,Y)$ is the subgroup of $\Hom_{\mathcal{E}}(X,Y)$ consisting of morphisms factoring through projective objects in $\mathcal{E}$.
For example, we denote by $\underline{\mbox{CM}} R$ the stable category of $\mbox{CM} R$.

\begin{prop}\label{Hom finite}
Let $R$ be a complete local Cohen--Macaulay isolated singularity. Then the stable category $\underline{\mbox{CM}}R$ is Hom-finite, that is, the $R$-module $\underline{\Hom}_R(X,Y)$ has finite length for each $X,Y\in\mbox{CM} R$.
\end{prop}

We are ready to prove Theorem \ref{determined by terms}.

\begin{proof}[Proof of Theorem \ref{determined by terms}]
(a) We need to show that each exact sequence $0\to A\xrightarrow{a} B\xrightarrow{b} C\to0$ in $\mbox{CM} R$ with $B\cong A\oplus C$ splits.

It is well-known that we have an induced exact sequence
\begin{equation*}
\underline{\Hom}_R(-,A)\xrightarrow{a}\underline{\Hom}_R(-,B)\xrightarrow{b}\underline{\Hom}_R(-,C)
\end{equation*}
of functors on $\mod R$. Evaluating $C$, we obtain an exact sequence
\begin{equation}\label{evaluate C}
\underline{\Hom}_R(C,A)\xrightarrow{a}\underline{\Hom}_R(C,B)\xrightarrow{b}\underline{\Hom}_R(C,C).
\end{equation}
On the other hand, since $B\cong A\oplus C$,
\[\mbox{length}_R\underline{\Hom}_R(C,A)+\mbox{length}_R\underline{\Hom}_R(C,C)=\mbox{length}_R\underline{\Hom}_R(C,B)\]
holds, where each term is finite by Proposition \ref{Hom finite}. In particular, the right map in \eqref{evaluate C} has to be surjective. Thus
\[\mbox{End}_R(C)=b\cdot\Hom_R(C,B)+P(C,C)\]
holds. Since $b\colon B \to C$ is surjective, each morphism in $P(C,C)$ factors through $b$. Thus $P(C,C)\subset b\cdot\Hom(C,B)$ and hence $\mbox{End}_R(C)=b\cdot\Hom_R(C,B)$ holds. In particular, $b$ is a split epimorphism.

(b) Let $\alpha\colon0\to A\xrightarrow{a}B\xrightarrow{b}C\to0$ be an almost split sequence, and $\beta\colon0\to A\xrightarrow{a'}B\xrightarrow{b'}C\to0$ an arbitrary exact sequence.

If $\beta$ splits, then (a) implies that $\alpha$ also splits, a contradiction.
Thus $\beta$ does not split.
Since $b\colon B\to A$ is right almost split and $b'\colon B\to A$ is not a split epimorphism, there exists $f\colon B\to B$ such that $b'=bf$. Thus we obtain a commutative diagram:
\[
      \begin{tikzcd}
        0 \arrow[r]  & A \arrow[r,"a^{\prime}"] \arrow[d, "g"] & B \arrow[r,"b^{\prime}"] \arrow[d, "f"]&  C \arrow[r] \arrow[d, "id"] & 0\\
        0 \arrow[r] & A \arrow[r,"a"]& B \arrow[r,"b"]&  C \arrow[r] & 0
      \end{tikzcd}
      \]
The left commutative square gives rise to an exact sequence
\[0\to A\xrightarrow{\left[\begin{smallmatrix}g\\ -a'\end{smallmatrix}\right]}A\oplus B\xrightarrow{\left[\begin{smallmatrix}a\ f\end{smallmatrix}\right]}B\to0.\]
 This sequence splits by (a). In particular, $\left[\begin{smallmatrix}a\ f\end{smallmatrix}\right]\colon A\oplus B\to B$ is a split epimorphism.
Since $a\colon A\to B$ belongs to the radical of the Krull--Schmidt category $\mbox{CM} R$, the morphism $f\colon B\to B$ has to be a split epimorphism.
Thus $f$ is an isomorphism, and so is $g$.
\end{proof}

Note that above proof works for morphisms in $\underline{\mbox{CM}} R$ since only difference is that we consider morphisms up to homotopy in $\underline{\mbox{CM}} R$.

\begin{cor}
Theorem \ref{Hom finite} holds in $\underline{\mathrm{CM}} R$ also.
\end{cor}

\begin{remark}
Not all exact sequences are determined by their terms.
A simple example is given by a simple singularity $R=k[[x,y,z]]/(x^3+xy^2+z^2)$ of type $D_4$.
For a unique indecomposable object $X\in\mbox{CM} R$ with rank $2$, there exists a one-parameter family of non-isomorphic exact sequences of the form $0\to X\to R^{\oplus2}\oplus X\to X\to0$.
\end{remark}

The stable matrix factorization category is equivalent to $\underline{\mathrm{CM}} R$. Hence we get a following corollary for Lagrangian surgery exact sequences.

\begin{cor}
Let $0\to A\xrightarrow{a} B \xrightarrow{b} C\to0$ be some AR sequence in $\underline{\mathrm{CM}} R$. If Lagrangian surgery exact sequence is given by $0\to A\xrightarrow{a^{\prime}} B \xrightarrow{b^{\prime}} C\to0$, this exact sequence is isomorphic to AR sequence.
\end{cor}


\bibliographystyle{amsalpha}
\bibliography{FukayaSing}

\end{document}